\documentclass{article}

\usepackage{graphicx}
\usepackage{pifont}

\usepackage{amsmath,amssymb,amsthm}
\usepackage{mathtools}
\usepackage{xcolor}
\usepackage{algorithm}
\usepackage{booktabs}
\usepackage{algpseudocode}
\usepackage{float}

\usepackage{mathrsfs} 
\usepackage{eufrak,euscript}
\usepackage{color}
\usepackage[shortlabels]{enumitem}
\usepackage{bm}
\usepackage{charter}
\definecolor{labelkey}{rgb}{0,0.08,0.45}
\definecolor{refkey}{rgb}{0,0.6,0.0}
\definecolor{Brown}{rgb}{0.45,0.0,0.05}
\definecolor{dgreen}{rgb}{0.00,0.49,0.00}
\definecolor{dblue}{rgb}{0,0.08,0.75}
\RequirePackage[colorlinks,hyperindex]{hyperref} %
\hypersetup{linktocpage=true,citecolor=dblue,linkcolor=dgreen}

\usepackage{fullpage}

\providecommand{\scalT}[2]{\left\langle{#1},{#2}\right\rangle}

\DeclareMathOperator*{\argmin}{arg\,min}
\newcommand{\prox}[2]{\text{prox}_{#1}(#2)}

\newtheorem{theorem}{Theorem}
\newtheorem{corollary}{Corollary}
\newtheorem{lemma}{Lemma}
\newtheorem{proposition}{Proposition}
\newtheorem{remark}{Remark}

\newtheorem{assumption}{Assumption}

\providecommand{\scalT}[2]{\left\langle{#1},{#2}\right\rangle}

\begin{document}

\title{A Proximal Stochastic Variance Reduced Zeroth-order Algorithm using Structured Finite-Differences}

\author{Marco Rando\thanks{Universit\'e C\^ote d'Azur, Inria, CNRS, LJAD, Nice, FR({\tt marco.rando@inria.fr}).}
    \quad Cheik Traor\'e\thanks{Toulouse School of Economics, Toulouse Capitole University, Toulouse, FR ({\tt cheik.traore@tse-fr.eu})}
	\quad Cesare Molinari\thanks{ Malga-DIMA, University of Genova, Genova, IT
		({\tt molinari@dima.unige.it}, {\tt silvia.villa@unige.it}).}
	\quad Lorenzo Rosasco \thanks{ Malga - DIBRIS, University of Genova, Genova, IT and SAIL Unit, Istituto Italiano di Tecnologia, Genova, IT ({\tt lorenzo.rosasco@unige.it})}
	\quad Silvia Villa\footnotemark[3]
}

\date{}

\maketitle

\begin{abstract}
    Minimizing finite sums of functions is a central problem in optimization, arising in numerous practical applications. Such problems are commonly addressed using first-order optimization methods. However, these procedures cannot be used in settings where gradient information is unavailable. Finite-difference methods provide an alternative by approximating gradients through function evaluations along a set of directions. For finite-sum minimization problems, it was shown that incorporating variance-reduction techniques into finite-difference methods can improve convergence rates. Additionally, recent studies showed that imposing structure on the directions (e.g., orthogonality) enhances performance. However, the impact of structured directions on variance-reduced finite-difference methods remains unexplored. In this work, we close this gap by proposing a structured variance-reduced finite-difference algorithm for non-smooth finite-sum minimization. We analyze the proposed method, establishing convergence rates for non-convex functions and those satisfying the Polyak-Łojasiewicz condition. Our results show that our algorithm achieves state-of-the-art convergence rates while incurring lower per-iteration costs. Finally, numerical experiments highlight the strong practical performance of our method.
\end{abstract}


\section{Introduction}\label{sec:introduction}

Finite-sum optimization is a class of problems that consists of minimizing the sum of a (typically large) number of functions. These problems are very relevant in many fields, such as machine learning, and are typically tackled using first-order methods, which rely on the availability of gradient information. However, in many practical scenarios such information is unavailable, and only function values can be accessed. These problems are called black-box finite-sum optimization problems and are particularly important, as they arise in a wide range of real-world applications \cite{bb_univ_pert,var_red_llm,pmlr-v80-ilyas18a,pmlr-v119-huang20j}. Notable examples are the generation of universal adversarial perturbations \cite{Moosavi-Dezfooli_2017_CVPR}, where the gradient is not available, and the fine-tuning of pre-trained LLMs \cite{raffel_llms}, where black-box optimization methods bypass the need for backpropagation \cite{mezo,zo_llm_benchmark}.

\noindent In the literature, various methods have been proposed to address black-box optimization problems - see e.g. \cite{salimans2017evolution,pmlr-v151-rando22a,Totzeck2022,LEWIS2000191,tutorial_BO,nesterov2017random,intro_schein,sartore2024automaticgaintuninghumanoid,pmlr-v80-ilyas18a,pmlr-v97-guo19a,demetrio2021functionality}. A large class of these methods consists of finite-difference methods. These are iterative procedures that mimic first-order optimization techniques by approximating the gradient of the objective function using finite differences along, possibly random, directions \cite{nesterov2017random}.

\noindent For finite-sum problems, recent works propose finite-difference methods that mimic first-order variance-reduction approaches. These techniques, in their simplest form, alternate one step requiring a full-gradient computation with a number of steps requiring the computation of a stochastic gradient \cite{svrg_johnson}. Existing finite-difference variance-reduction techniques mimic such approaches by replacing gradients and stochastic gradients with surrogates constructed with random directions sampled i.i.d. from a probability distribution ({\it unstructured directions}), or by using all canonical basis vectors \cite{zo_svrg_avg,liu2018stochastic}. The first class of approaches has the advantage of allowing cheap gradient approximations, since in principle few random directions can be used; however, to obtain accurate gradient approximations using i.i.d. directions, many of them are required \cite{berahas2022theoretical}. The second class instead provides accurate approximations but, since each (stochastic) gradient surrogate is built using all coordinate directions, these methods require a large number of function evaluations per iteration, making them impractical for many real-world problems \cite{Kazemi2024,var_red_llm}. To address this limitation, hybrid approaches have been proposed. These methods approximate the full gradient using finite differences with coordinate directions, and the stochastic gradients using approximations with unstructured directions, balancing the cost of the approximations with performance. 
However, as observed in \cite{berahas2022theoretical}, unstructured approximations can be very imprecise if a sufficient number of directions is not used. Such an imprecision affects the practical performance of the algorithm: to compensate for inaccurate gradient estimates, one is forced either to use a small stepsize (which slows downconvergence), to increase the number of sampled directions (which increases the cost per iteration), or to frequently update and recompute the full-gradient approximation, which, being built with coordinate directions, is itself expensive.  Several recent works on finite-difference methods have theoretically and empirically observed that constructing the directions used for gradient approximations by imposing structural constraints, such as orthogonality ({\it structured directions}), leads to better performance than using unstructured directions \cite{berahas2022theoretical,str_zo_applied,rando2025structuredtouroptimizationfinite}. Intuitively, this happens because orthogonal directions inherently prevent the sampling of redundant or similar directions, improving the quality of the stochastic gradient estimate. Moreover, it has been observed that a better quality of the gradient estimate can be obtained even when fewer directions than the dimension of the input space $d$ are used, thus allowing to save function evaluations \cite{rando2025structuredtouroptimizationfinite}. 
Moreover, structured directions can lead to improved performance in settings with unfavorable geometry, since their randomness allows the method to explore different subspaces across iterations, mitigating ill-conditioning effects \cite[Section 3.4]{polyak1987introduction}.
Despite these theoretical and empirical advantages in standard finite-difference methods \cite{berahas2022theoretical,str_zo_applied,rando2025structuredtouroptimizationfinite}, the integration of structured directions into variance-reduced algorithms has remained unexplored, and no existing variance-reduced algorithm employs structured finite differences. It is therefore unclear whether structured directions can improve the performance of zeroth-order variance-reduced methods, for instance by reducing the number of function evaluations per iteration required for stochastic gradient surrogates and by allowing the use of a larger stepsize.

\noindent In this work, we close this gap by proposing VR-SZD, a variance-reduced finite-difference algorithm based on structured directions for black-box finite-sum minimization. Our algorithm exploits a variance-reduction technique and alternates the computation of full-gradient approximations (with coordinate directions) and stochastic gradient approximations built using a set of $\ell \leq d$ orthogonal directions, instead of unstructured directions. We analyze our procedure, providing convergence rates for non-convex functions and achieving a complexity, in terms of function evaluations, of $\mathcal{O}(d n^{2/3} \varepsilon^{-1})$, where $d$ is the input dimension, $n$ the number of functions and $\varepsilon \in (0,1)$ the accuracy. We then extend the analysis to the setting of non-convex Polyak-\L{}ojasiewicz functions, deriving a rate for the objective function values. Additionally, we compare our algorithm to state-of-the-art approaches on different experiments, showing that VR-SZD achieves better performance than existing methods. In particular, our results show that this improvement stems from the quality of the stochastic gradient estimator built with $\ell \leq d$ structured directions, which is good enough to allow a larger stepsize than unstructured methods while requiring fewer function evaluations per iteration than coordinate-based methods. Unstructured methods, in contrast, must either sample a large number of directions or adopt a smaller stepsize, while coordinate-based methods, due to their per-iteration cost, exhaust the function evaluation budget within too few iterations. VR-SZD allows to have more accurate approximations of unstructured methods (and thus allow choosing a larger stepsize) using a number of directions that is smaller that the dimension (and thus performing more iterations than coordinate-based methods within the same budget).

\paragraph*{Contributions.} 
\begin{itemize}
    \item We introduce a novel finite difference algorithm that uses structured directions and a variance reduction technique to improve the efficiency of gradient estimations. 
    \item We provide a rigorous theoretical analysis of the proposed algorithm in optimizing non-convex and non-convex  Polyak-\L{}ojasiewicz targets. We derive the complexity of our method in terms of the number of function evaluations which matches the state of the art one (see Table~\ref{tab:complexities}) and allows for the presence of a non-smooth term in the sum, and without requiring boundedness of the stochastic gradient estimators.
    \item  We benchmark VR-SZD against state-of-the-art variance reduced unstructured and hybrid algorithms on non-smooth minimization tasks. Our numerical experiments empirically demonstrate that the use of structured directions enables larger step sizes and yields faster convergence. We complement the comparison with an ablation study, assessing the impact of VR-SZD's parameters and the use of structured directions on its empirical performance.
\end{itemize}
\noindent The rest of the paper is organized as follows. In Section \ref{sec:problem_setting}, we introduce the problem and describe our algorithm. In Section \ref{sec:main_results}, we state and discuss the main results. In Section \ref{sec:experiments} we provide some numerical experiments and in Section \ref{sec:conclusion} some final remarks.

\section{Problem Setting \& Algorithm}\label{sec:problem_setting}

We consider the problem of minimizing a non-smooth objective $F$ that is expressed as the sum of two functions,
\begin{equation}\label{eqn:problem}
    \min\limits_{x \in \mathbb{R}^d} F(x) := f(x) + h(x), \quad \text{where} \quad f(x) := \frac{1}{n} \sum\limits_{i = 1}^n f_i(x).    
\end{equation}
In particular, the function $f$ is a finite sum of  $n > 0$ differentiable, potentially non-convex functions $\{f_i\}_{i=1}^n$ for which the direct access to the gradients is unavailable. Further, the function $h$ is convex, possibly non-differentiable, and {\it proximable}, meaning that for any  $x \in \mathbb{R}^d$ and $\gamma > 0$, we can compute cheaply
\begin{equation} \label{eqn:prox_operator}
    \prox{\gamma h}{x} := \argmin\limits_{y} \left\{ h(y) + \frac{1}{2\gamma} \| y - x \|^2 \right\}.    
\end{equation}
This setting is well studied \cite{Kazemi2024,Huang_Gu_Huo_Chen_Huang_2019} and is particularly relevant because many real-world optimization problems can be formulated in this way - see e.g. \cite{bb_univ_pert,zo_llm_benchmark}. Notably, the assumptions on $h$ are standard in composite optimization \cite{szd_prox_weakly_conv} and a typical example %
is the indicator function over a convex set, for which the proximal operator (eq. \eqref{eqn:prox_operator}) is the projection onto that set. To address Problem \eqref{eqn:problem}, we propose an iterative procedure that relies solely on stochastic function evaluations and prox computations.  

\noindent Before detailing the formal setup, we provide an intuitive overview of the proposed Variance Reduced Structured Zeroth-order Descent (VR-SZD) algorithm. Our method addresses black-box finite-sum minimization by adapting the established first-order SVRG framework to a zeroth-order setting. In standard first-order SVRG, optimization proceeds in two loops: an outer loop that computes an exact full gradient at a reference point to anchor the trajectory, and an inner loop that takes computationally cheap stochastic steps whose variance is explicitly reduced by the anchor. VR-SZD mimics this architecture using only function evaluations. In the outer loop, we approximate the full gradient using forward finite differences along all canonical coordinate directions. While accurate, this is computationally expensive. Therefore, in the inner loop, we build a stochastic gradient surrogate using a reduced set of $l \le d$ structured directions, represented by an orthogonal matrix $G$. Previous variance-reduced methods typically rely either on unstructured random directions (which suffer from high variance and restrict the stepsize) or full coordinate directions (which incur a prohibitive per-iteration cost). By imposing orthogonality on a subset of directions, VR-SZD significantly improves the accuracy of the gradient estimate without requiring a full $\mathcal{O}(d)$ evaluation budget per step. In the next we discuss more in detail the two types of gradient approximations used by the algorithm to construct the search direction.
The first surrogate is computed by finite differences along the canonical basis $(e_j)_{j=1}^d$. Specifically, given $\beta > 0$, we define it as
\begin{equation}\label{eqn:full_approx}
    g(x, \beta) := \sum\limits_{j = 1}^d \frac{f(x + \beta e_j) - f(x)}{\beta} e_j.    
\end{equation}
In the latter, the positive scalar $\beta$ is the finite difference stepsize, also called discretization or smoothing parameter, governing the trade-off between the approximation accuracy and numerical stability. Indeed, notice that $g(x, \beta)$ recovers the gradient of the objective function in the limit $\beta \to 0$. The second surrogate is an approximation of a stochastic gradient and it is constructed using finite differences along a set of $\ell \leq d$ structured directions. These are represented as orthogonal transformations of the first  $\ell$  canonical basis. Formally, let $i \in [n] := \{1, \dots, n\}$  and $G \in O(d) := \{ G \in \mathbb{R}^{d \times d} \, | \, \det(G) \neq 0 \, \wedge \, G^{-1} = G^\intercal \}$. We define the structured approximation of the stochastic gradient of the function $f_i$ as
\begin{equation}\label{eqn:stochastic_approx}
    \hat{g}_{i}(x, G, \beta) := \frac{d}{\ell} \sum\limits_{j = 1}^\ell \frac{f_i(x + \beta G e_j) - f_i(x)}{\beta} G e_j.
\end{equation}
In the latter, the scaling $d/ \ell$ is needed to ensure that $\hat{g}_{i}(x, G, \beta)$ is an unbiased estimator of the objective gradient in the limit case $\beta \to 0$ and for an opportune sampling of the matrix $G$ (for instance, uniformly on $O(d)$). Structured direction matrices has the advantage of allowing one to use fewer than $d$ directions, thereby reducing the number of function evaluations required to construct the surrogate (compared to coordinate-based estimators), while still preserving higher accuracy than unstructured directions (e.g., i.i.d. vectors sampled uniformly from the unit sphere) \cite{berahas2022theoretical,rando2023optimal}.  Several methods have been proposed in the literature for constructing such matrices \cite{rando2025structuredtouroptimizationfinite}, and we describe the specific approach used in our experiments in Appendix \ref{app:exp_details}. Notice that this approximation differs from the one proposed in \cite{rando2023optimal}. While their approach relies on central finite differences, our method employs forward finite differences, which are computationally cheaper. Specifically, our approximation requires only $\ell + 1$ function evaluations compared to the $2\ell$ required by their method. 
The algorithm we study is defined by the following iteration.
\begin{algorithm}[H]

\caption{VR-SZD: Variance Reduced Structured Zeroth-order Descent}\label{algo:osvrz}
\begin{algorithmic}[1]
\State \textbf{Input:} $x^0_0 \in \mathbb{R}^d$, $T \in \mathbb{N}_+$, $m \in \mathbb{N}_+$, $\gamma \in \mathbb{R}_+$, $(\beta_\tau)_{\tau \in \mathbb{N}} \subset \mathbb{R}_+$, $\ell \in \mathbb{N}$ s.t. $1 \leq \ell \leq d$, $b > 0$
\For{$\tau = 0$ to $T$}
    \State $g_\tau := g(x_0^\tau, \beta_\tau)$ %
    \For{$k = 0$ to $m - 1$}
        \State Sample $i_{1, k}^\tau, \cdots, i_{b, k}^\tau$ uniformly and independently from $[n]$ %
        \State Sample $G_{1,k}^\tau, \cdots, G_{b,k}^\tau$ uniformly and independently from $O(d)$  
        \State Compute
        \begin{equation*}
        v_k^\tau = \frac{1}{b} \sum\limits_{j = 1}^b (\hat{g}_{i^{\tau}_{j, k}}(x_k^\tau, G^{\tau}_{j, k}, \beta_\tau) - \hat{g}_{i_{j, k}^\tau}(x_0^\tau, G^{\tau}_{j, k}, \beta_\tau)) + g_\tau    
        \end{equation*}
        \State $x_{k + 1}^\tau = \text{prox}_{\gamma h}(x_k^\tau - \gamma  v_k^\tau)$
    \EndFor
    \State $x_0^{\tau + 1} = x_m^{\tau}$
\EndFor
\State \textbf{Output:} $x_0^{T}$
\end{algorithmic}
\end{algorithm}
\noindent Given an initial guess $x_0^0 \in \mathbb{R}^d$, at every iteration $\tau \in \mathbb{N}$, the algorithm computes the full gradient surrogate using eq. \eqref{eqn:full_approx} at the current (outer) iterate $x_0^\tau \in \mathbb{R}^d$. Then, for $k =0, \cdots, m - 1$ iterations, it samples a batch of $b$ indices $(i_{j, k}^\tau)_{j=1}^b \subset [n]$ (with replacement) and, for every index $i_{j, k}^\tau$, an orthogonal matrix $G_{j,k}^\tau$ uniformly from $O(d)$. %
Then, the  direction $v_k^\tau \in \mathbb{R}^d$ is computed as
\begin{equation}\label{eqn:osvrz_v}
    v_k^\tau = \frac{1}{b} \sum\limits_{j = 1}^b (\hat{g}_{i_{j, k}^\tau}(x_k^\tau, G_{j,k}^\tau, \beta_\tau) - \hat{g}_{i_{j, k}^\tau}(x_0^\tau, G_{j,k}^\tau, \beta_\tau)) + g(x_0^\tau, \beta_\tau)
\end{equation}
and the inner iterate is updated as
\begin{equation}\label{eqn:inner_iter}
    x_{k+1}^\tau = \prox{\gamma h}{x_k^\tau - \gamma v_k^\tau}.
\end{equation}
After $m$ steps, the new outer iterate is set to be the last inner iterate i.e. 
\begin{equation*}
    x_0^{\tau + 1} = x_m^\tau.
\end{equation*}
Notice that Algorithm \ref{algo:osvrz} depends on different parameters, namely the sequence of discretization parameters $(\beta_\tau)_{\tau \in \mathbb{N}}$, the batch size $b$, the number of directions $\ell$, the stepsize $\gamma$, and the number of inner iterations $m$. Decreasing $\beta$ makes the finite difference a more accurate surrogate of the gradient, but may cause numerical instabilities. On the other hand, increasing the batch size $b$ or the number of directions $\ell$ again increases the accuracy of the surrogate gradient, but also the functional evaluations required per iteration. One of our main contribution is the theoretical and empirical analysis of this algorithm with respect to the choice of these parameters and their interaction.

\begin{remark}
Notice that the generation of orthogonal matrices in the inner iteration has a high computational cost. In practice, for $\ell < d$, instead of generating $b$ orthogonal $d \times d$ matrices, we can reduce the cost using a QR factorization with diagonal adjustment \cite{mezzadri2006generate} of $d \times \ell$ random matrices with entries sampled from $\mathcal{N}(0, 1)$. This operation has a computational cost of $\mathcal{O}(b d \ell^2)$. Alternatively, by relaxing the theoretical requirements, cheaper methods can be used \cite{rando2023optimal,rando2025structuredtouroptimizationfinite}. Moreover, we recover the iteration of \cite{Kazemi2024} as a special case when $\ell = 1$.  With this choice, we show that our method achieves the same complexity as state-of-the-art algorithms such as \cite{Kazemi2024,ji_improved_vr}, while requiring sampling less functions per iteration in the inner loops - see Section \ref{sec:main_results}.  In  Appendix \ref{app:ablation}, %
we also compare our procedure with $b = 1$ and $\ell > 1$ against the algorithm proposed in \cite{Kazemi2024} with $b = \ell$, observing that our algorithm provides better performance. 
\end{remark}

\paragraph*{Notation.} 
In the following sections, we denote with 
\begin{equation*}
    \begin{aligned}
         &\mathcal{F}_k^\tau := \sigma(i_{1, 0}^0, \cdots,i_{b, 0}^0, i_{1, 1}^0, \cdots, i_{1, m - 1}^0, i_{1, 0}^1, \cdots, i_{b, k - 1}^\tau)\\
         &\mathcal{H}_k^\tau := \sigma(G_{1, 0}^0, \cdots,G_{b, 0}^0, G_{1, 1}^0, \cdots, G_{1, m - 1}^0, G_{1, 0}^1, \cdots, G_{b, k - 1}^\tau),
    \end{aligned}
\end{equation*}
and, to improve the readability of the proofs, with
\begin{equation*}
    \begin{aligned}
        \mathbb{E}_{I_k^\tau}[ \cdot ] &:= \mathbb{E}_{i_{1, k}^\tau}[\mathbb{E}_{i_{2, k}^\tau}[\cdots \mathbb{E}_{i_{b, k}^\tau} [ \cdot \, | \, \mathcal{F}_k^\tau] \, | \, \mathcal{F}_k^\tau] \, | \, \mathcal{F}_k^\tau]\\
        \mathbb{E}_{G_k^\tau}[ \cdot ] &:= \mathbb{E}_{G_{1, k}^\tau}[ \mathbb{E}_{G_{2, k}^\tau}[\cdots \mathbb{E}_{G_{b, k}^\tau} [ \cdot \, | \, \mathcal{H}_k^\tau] \, | \, \mathcal{H}_k^\tau] \, | \, \mathcal{H}_k^\tau].
    \end{aligned}
\end{equation*}
Furthermore, note that the random variables in
$$\Big\{\big\{i_{j,k}^\tau, G_{j, k}^\tau\big\}: j \in \{1,\ldots,b\}, k \in \{0, \ldots, m-1\}, \tau \in \{0, \ldots, T\}\Big\}$$
are assumed independent. Unless otherwise specified, $\|\cdot\|$ denotes the Euclidean $\ell_2$ norm i.e. for every $x\in \mathbb{R}^d$, $\|x\| = \left(\sum\limits_{i=1}^d x_i^2 \right)^{1/2} $. We denote the unit ball and the unit sphere as follow
\begin{equation*}
    \mathbb{B}^d := \{u \in \mathbb{R}^d \,|\, \|u\| \leq 1\} \qquad \text{and} \qquad \mathbb{S}^{d - 1} := \{v \in \mathbb{R}^d \,|\, \|v\| = 1\}.
\end{equation*}
\noindent Moreover, we denote with $\mu$ and $\sigma^{d - 1}$ the normalized Haar measure and the normalized spherical measure on $\mathbb{S}^{d - 1}$.

\subsection{Related Works} \label{sec:related_works}
Finite-difference methods have been widely studied in the literature in various settings, with both single and multiple, structured and unstructured directions - see, e.g. \cite{nesterov2017random,ghadimi_lan,duchi_power_of_two,shamir2017optimal,kozak2021zeroth,gasnikov_sph,rando2023optimal,Rando2024,shamir2017optimal} and references therein. Additionally, several gradient-free methods have been proposed for black-box finite-sum optimization (i.e., problem \eqref{eqn:problem}) - see \cite{zo_svrg_avg,liu2018stochastic,ji_improved_vr,fang_spider,Kazemi2024,mu2024variancereducedgradientestimatornonconvex}. These algorithms mimic first-order variance reduction methods such as SVRG \cite{svrg_johnson}, SAGA \cite{saga}, and SPIDER \cite{fang_spider}, replacing (stochastic) gradients with finite-difference approximations. Notably, these methods either use unstructured directions or rely on all coordinate directions to approximate both full and stochastic gradients \cite{zo_svrg_avg,ji_improved_vr,Kazemi2024,Huang_Gu_Huo_Chen_Huang_2019}. Each approach has its pros and cons. Using unstructured directions to approximate stochastic gradients allows for fewer function evaluations per inner iteration than coordinate directions but results in worse gradient approximations and only a smaller step size. Conversely, approximating stochastic gradients with all coordinate directions enables the use of a larger step size but requires a high number of function evaluations per iteration, which can be prohibitive for some applications. Our approach aims to combine the advantages of both. It approximates stochastic gradients using $\ell \leq d$ directions, avoiding the high cost of coordinate directions, while imposing orthogonality on these directions to achieve better approximations than unstructured directions.

\noindent In the following, we review the most related works and highlight the distinctions between these methods and our proposed algorithm. To the best of our knowledge, for finite-sum optimization with variance reduction, no prior work used structured directions.

\paragraph*{Finite-difference Methods.} Most existing works have focused on the theoretical analysis of zeroth-order methods that use unstructured directions to approximate gradients - see, e.g., \cite{nesterov2017random,duchi_power_of_two,ghadimi_lan,shamir2017optimal,gasnikov_sph,salimans2017evolution,flaxman2005online,gfm_lin_zheng_jordan,chen2015randomized,zoro,ZO-BCD} and reference therein. Notably, most results in the non-convex setting analyze the convergence of the expected norm of the gradient of $f$ and do not extend to composite optimization, where the function $h$ may be non-smooth. Convergence rates for the quantities of interest are typically expressed in terms of number of function evaluations ({\it complexity}), explicitly characterizing the dependence on the dimension of the input space $d$. However, they often do not show the dependence on the number of functions $n$, as the specific case of finite-sum objectives is often not analyzed. In the following, we express the complexity explicitly showing the dependence on $n$. In \cite{nesterov2017random,duchi_power_of_two}, the authors analyze finite-difference algorithms using a single Gaussian direction. In \cite{nesterov2017random}, a complexity of $\mathcal{O}(n d \varepsilon^{-1})$ is achieved in the smooth non-convex setting. In \cite{duchi_power_of_two}, both single and multiple directions are considered, and lower bounds are derived. However the analysis is limited to convex functions. In \cite{ghadimi_lan} a similar zeroth-order algorithm is proposed for stochastic optimization and the analysis shows a complexity of $ \mathcal{O}(d \varepsilon^{-2})$ for smooth non-convex functions. Structured finite-difference methods have also been analyzed in recent works \cite{rando2023optimal,Rando2024,kozak2021zeroth,stief_zeroth,str_zo_applied}. In \cite{rando2023optimal}, a structured finite-difference algorithm is studied across various settings, achieving a complexity of $ \mathcal{O}(n d \varepsilon^{-1})$ for smooth non-convex functions. Similarly, \cite{kozak2021zeroth} introduces a structured method, which is extended to stochastic setting in \cite{Rando2024}. However these methods have not been analyzed for general non-convex functions. In \cite{stief_zeroth}, a finite-difference method that samples direction matrices from the Stiefel manifold \cite{chikuse2012statistics} is proposed, but its analysis is restricted to smooth functions satisfying \L{}ojasiewicz inequalities \cite{lojasiewicz1963topological}. In \cite{str_zo_applied} a finite-difference method that uses Hadamard-Rademacher matrices as directions is proposed but no theoretical analysis is provided. In the domain of adversarial machine learning, \cite{rando2024newformulationzerothorderoptimization} introduces a structured approach, but analysis is provided only for non-smooth non-convex functions. For composite optimization, \cite{Ghadimi2016} presents an unstructured finite-difference method for stochastic optimization, achieving a complexity of $\mathcal{O}(d \varepsilon^{-2})$ in non-convex setting. However, to the best of our knowledge, no structured finite-difference algorithm for problem \eqref{eqn:problem} has been proposed yet. 

\paragraph*{Variance-reduced Finite-difference Methods.} Most of the approaches, including our algorithm, mimic SVRG algorithm; see, e.g., \cite{zo_svrg_avg,liu2018stochastic,ji_improved_vr,Kazemi2024,Huang_Gu_Huo_Chen_Huang_2019,mu2024variancereducedgradientestimatornonconvex}. In \cite{zo_svrg_avg}, two methods are proposed: the first approximates both gradient and stochastic gradients using $\ell$ directions sampled uniformly from a sphere, while the second uses all canonical basis. Both approaches achieve a complexity of $\mathcal{O}( n d \varepsilon^{-1})$ in the smooth non-convex setting. Similarly, in \cite{liu2018stochastic}, a method that approximates the full gradient using $\ell$ Gaussian vectors is proposed. It differs from \cite{zo_svrg_avg} from the fact that stochastic gradients are approximated using finite differences with a single Gaussian direction randomly selected from those used for the full gradient approximation. In \cite{ji_improved_vr}, two strategies are introduced. The first is a hybrid method that approximates the full gradient using coordinate directions and the stochastic gradients with a mini-batch of finite-difference approximations using a single direction sampled uniformly from a sphere. The second method uses coordinate directions for both gradient and stochastic gradient approximations. These methods achieve a complexity of $\mathcal{O}(n^{2/3} d \varepsilon^{-1})$ for smooth non-convex functions. However, these algorithms are not designed for problem \eqref{eqn:problem} and the results do not extend to the cases where $h$ is non-smooth. For composite optimization i.e. problem \eqref{eqn:problem}, \cite{Huang_Gu_Huo_Chen_Huang_2019} introduces a method that uses all coordinate directions to approximate both the full gradient and stochastic gradients; and an algorithm that approximates stochastic gradients using finite differences with Gaussian directions. These methods achieve for non-convex functions a complexity of $\mathcal{O}(n^{2/3} d \varepsilon^{-1})$ and $\mathcal{O}(n^{2/3} d \varepsilon^{-1} + d \sigma^2)$ respectively. However, their analysis relies on the assumption that the norm of the stochastic gradients is bounded by some constant $\sigma^2$, an assumption that our approach does not require. In \cite{Kazemi2024}, a hybrid method is introduced that approximates the full gradient using all coordinate directions and the stochastic gradients with unstructured directions. This approach achieves a complexity of $\mathcal{O}(n^{2/3} d \varepsilon^{-1})$ in the same setting. Furthermore, \cite{Kazemi2024} also proposes an algorithm that uses coordinate directions for both gradient and stochastic gradient approximations, achieving the same rate. While our results match with those of \cite{Kazemi2024}, our approach has several key differences. First, \cite{Kazemi2024} uses central finite differences to construct the full gradient approximation, while our method employs forward finite differences (see equation \eqref{eqn:full_approx}), resulting in a cheaper estimator. Second, in \cite{Kazemi2024}, the direction $v_k^\tau$ is computed by sampling $b$ functions $f_i$ and averaging their stochastic gradient approximations, each constructed with a single direction sampled uniformly from a sphere. In contrast, our method samples $b$ functions and approximates each stochastic gradient using $\ell \leq d$ random orthogonal directions. This makes our approach more general, as it can  recover the inner iteration of \cite{Kazemi2024} as a special case when $\ell = 1$, since every vector $v \in \mathbb{S}^{d - 1}$ can be expressed as a rotation $G \in O(d)$ of the first canonical base $e_1$ \footnote{Notice that in \cite{Kazemi2024}, the authors also consider the case where the gradient approximation in the outer loop is computed on a mini-batch of functions. We do not consider this setting since, as shown in \cite{Kazemi2024}, performing the analysis would require assuming bounded variance, leading to a worse rate.}. Other methods inspired by Spider or SAGA algorithms have been proposed in the literature \cite{ji_improved_vr,Huang_Gu_Huo_Chen_Huang_2019,Kazemi2024,fang_spider,Zhu_Zhang_2023}, but none uses structured directions. In this work, we focus on SVRG-based algorithms, leaving the development of other variants as a research direction. %

\noindent In Table~\ref{tab:complexities}, we summarize the complexity of different algorithms to optimize non-convex functions. Specifically, for algorithms analyzed under the smooth non-convex case (i.e., when $h$ is a constant function), the reported complexity corresponds to the number of function evaluations required by the algorithm to obtain a point $x \in \mathbb{R}^d$ such that $\mathbb{E}[\|\nabla f(x) \|^2] \leq \varepsilon$. In contrast, for methods addressing problem~\eqref{eqn:problem} (i.e. where $h$ can be a convex, non-smooth, and proximable function), we use the gradient mapping \cite{Ghadimi2016,Kazemi2024} instead of the gradient  -  see eq.~\eqref{eqn:generalized_gradient}.

\begin{table}[H]
    \centering
    \caption{Comparison of different zeroth-order algorithms in terms of complexity for finite-sum optimization. For each algorithm, we specify whether the (stochastic) gradient approximations are constructed using unstructured (UST), structured (STR), or coordinate (COO) directions, the achieved complexity in the non-convex setting, and whether the results extend to cases where the function $h$ is non-smooth (\checkmark = yes, \ding{55} = no).}
    \label{tab:complexities}    
    \begin{tabular}{l l l l}
        \toprule
        Algorithm & Grad. Approx. & Complexity &  $h$ non smooth\\
        \midrule
         ZO-GD \cite{nesterov2017random} & UST & $\mathcal{O}(n d \varepsilon^{-1})$ & \ding{55}\\
         ZO-SGD \cite{ghadimi_lan} & UST & $\mathcal{O}(d \varepsilon^{-2})$ & \ding{55} \\
         O-ZD \cite{rando2023optimal} & STR & $\mathcal{O}(n d \varepsilon^{-1})$ &  \ding{55}\\
         RSPGF \cite{Ghadimi2016} & UST & $\mathcal{O}(d^2 \varepsilon^{-2})$ &  \checkmark\\
         ZO-SVRG-Ave \cite{zo_svrg_avg} & UST & $\mathcal{O}(n d \varepsilon^{-1})$ & \ding{55}\\
         ZO-SVRG-Coord \cite{ji_improved_vr} & COO & $\mathcal{O}(n^{2/3} d \varepsilon^{-1})$ & \ding{55}\\
         ZO-SVRG-CoordRand \cite{ji_improved_vr} & UST & $\mathcal{O}(n^{2/3} d \varepsilon^{-1})$ & \ding{55}\\
         ZO-ProxSVRG [GauSGE] \cite{Huang_Gu_Huo_Chen_Huang_2019} & UST & $\mathcal{O}(n^{2/3} d \varepsilon^{-1} + d \sigma^2)$ & \checkmark\\
         ZO-ProxSVRG [CooSGE] \cite{Huang_Gu_Huo_Chen_Huang_2019} & COO & $\mathcal{O}(n^{2/3} d \varepsilon^{-1})$ & \checkmark\\
         ZO-PSVRG+ \cite{Kazemi2024} &  COO & $\mathcal{O}(n^{2/3} d \varepsilon^{-1})$ &  \checkmark\\
         ZO-PSVRG+ [RandSGE] \cite{Kazemi2024} &  UST & $\mathcal{O}(n^{2/3} d \varepsilon^{-1})$ &  \checkmark\\
         {\bf VR-SZD (Our)} &  STR & $\mathcal{O}(n^{2/3} d \varepsilon^{-1})$ & \checkmark \\
        \bottomrule
    \end{tabular}

\end{table}

\section{Main Results}\label{sec:main_results}
In this section, we provide the main results considering two settings: non-convex and non-convex Polyak-\L{}ojasiewicz \cite{bolte}.  For both settings, we assume that the objective function satisfies the following hypothesis
\begin{assumption}[$L$-smoothness]\label{ass:l_smooth}
    There exists a constant $L > 0$ such that for every $i \in [n]$, %
    the function $f_i$ is $L$-smooth; i.e. $f_i$ is differentiable and, for every $x, y \in \mathbb{R}^d$,
    \begin{equation*}
        \| \nabla f_i(x) - \nabla f_i(y) \| \leq L \| x - y \|.
    \end{equation*}
    Moreover, the function $h$ is convex and proximable.
\end{assumption}
\noindent This is a standard assumption considered also in other works - see e.g. \cite{ji_improved_vr,redd_svrg_nonconv,prox_svrg_nconv,Ghadimi2016,Rando2024}. In the following sections, we provide the rates for non-convex and non-convex Polyak-\L{}ojasiewicz settings while we provide the proofs in Appendix \ref{app:addproofs}. %

\subsection{Non-convex Setting}
To analyze convergence properties of our algorithm in the setting in which $f$ is non-convex, we introduce the gradient mapping (or generalized gradient) as follows  
\begin{equation}\label{eqn:generalized_gradient}
    (\forall x \in \mathbb{R}^d,\,\forall\gamma > 0) \qquad\mathcal{G}_\gamma(x) := \frac{1}{\gamma} \left( x - \prox{\gamma h}{x - \gamma \nabla f(x)} \right).    
\end{equation}
Note that when $h$ is a constant function and so the proximal operator is the identity, the gradient mapping simplifies to $\mathcal{G}_\gamma(x) = \nabla f(x)$ for every $\gamma > 0$. A point $x \in \mathbb{R}^d$ is defined stationary if $\mathcal{G}_\gamma(x) = 0$. The use of gradient mapping as a convergence metric has also been studied in previous works \cite{Ghadimi2016,Kazemi2024}. Next, we provide bounds on the expected squared norm of the gradient mapping. In particular, we use the following notation
\begin{equation*}
    \eta_m^T := \frac{1}{(T + 1)m} \sum\limits_{\tau = 0}^T \sum\limits_{k = 0}^{m - 1} \mathbb{E}[\| \mathcal{G}_\gamma(x_k^\tau) \|^2].
\end{equation*}
Now, we state the main theorem on the rate for non-convex functions.
\begin{theorem}[Rates in Non-convex Setting]
\label{thm:nonconv_rate}
    Under Assumption \ref{ass:l_smooth}, let $m > 1$, $(\beta_\tau )_{\tau \in \mathbb{N}}$ with $\beta_\tau > 0$ for every $\tau \in \mathbb{N}$, $\gamma < \min\left( \frac{1}{4L} , \frac{\sqrt{\ell b}}{\sqrt{32(e - 1) d} L m}\right)$ where $e$ denotes the Nepero's constant and $\Delta := (\frac{1}{2} - L \gamma)$ . 
    Let $x_k^\tau$ be the sequence generated by Algorithm \ref{algo:osvrz} for every $0 \leq k \leq m - 1$ and $\tau \geq 0$,  
        \begin{equation*}
            \begin{aligned}
             \eta_m^T &\leq \frac{ F(x_0^0) - \min F }{\Delta \gamma m (T + 1)}  + \left( \frac{3 d}{\ell b } + 2 + 3d \right) \frac{L^2 d}{\Delta (T + 1)} \sum\limits_{\tau = 0}^T \beta_\tau^2. 
            \end{aligned}
        \end{equation*}
\end{theorem}
\noindent In the next corollary, we derive explicit rates for specific choices of the parameters.
\begin{corollary}
\label{cor:cor1}
    Under Assumption \ref{ass:l_smooth}, let $\gamma = \frac{\sqrt{\ell b}}{10 L m \sqrt{d}}$ and $m \geq \sqrt{b}$. Then:\\
    (i) choosing $\beta_\tau = \frac{\beta}{d} (\tau + 1)^{-\alpha}$ with $\alpha > 1/2$ and $\beta > 0$, we have
    \begin{equation*}
        \eta_m^T \leq 40\frac{ L \sqrt{d}}{ \sqrt{\ell b}}\frac{F(x_0^0) - \min F}{T + 1} + \mathcal{O}\left(\frac{1}{T + 1}\right);
    \end{equation*}
    (ii) choosing $\beta_\tau = \frac{\beta}{d}$ with $\beta > 0$, we get
    \begin{equation*}
        \eta_m^T \leq 40\frac{ L \sqrt{d}}{ \sqrt{\ell b}} \frac{F(x_0^0) - \min F}{T + 1} + C_2 \beta^2,
    \end{equation*}
    where $ C_2 := 4\left( \frac{3 d}{\ell b } + 2 + 3d \right) \frac{L^2}{d} > 0$.\\
    (iii) for $\varepsilon \in (0,1)$, fix $T$ such that $T + 1 \geq \frac{80 L \sqrt{d} (F(x^{0}_0) - \min F)}{\sqrt{\ell b} \Delta} \varepsilon^{-1} $. Choosing $\beta_\tau = \sqrt{\frac{\varepsilon}{2C_3}}$ where $C_3 := 4\left( \frac{3 d}{\ell b } + 2 + 3d \right) L^2 d$, we have $\eta_m^T \leq \varepsilon$ and the complexity is 
    \begin{equation*}
    \mathcal{O}\left( \left( n \frac{d \sqrt{d}}{\sqrt{\ell b} } + \frac{m b (\ell + 1) \sqrt{d}}{ \sqrt{\ell b}} \right) \frac{1}{\varepsilon}\right).
    \end{equation*}
    In particular, choosing $b =\lceil n^{2/3} \rceil$, $m = \lceil \sqrt{b} \rceil$ and $\ell = \lceil d / c \rceil$,  for some $c > 1$, the complexity is 
    \begin{equation*}
    \mathcal{O}\left( d n^{2/3} \varepsilon^{-1}\right).
    \end{equation*}

\end{corollary}

\paragraph*{Discussion.} The bound in Theorem \ref{thm:nonconv_rate} is composed two terms: one dependent on the initialization and another representing the approximation error. Specifically, the latter arises from two sources of error: the approximation of (stochastic) gradients through finite differences and the use of outdated gradient surrogates in the inner loop iterations. Notice that this error term depends on $\sum\limits_{\tau = 0}^T \beta_\tau^2$, namely, the discretization parameters used in the finite difference approximation at each step. By selecting  $\beta_\tau^2 \in \ell^{1}$, the error can be made vanish as the (outer) iterations progress. In Corollary \ref{cor:cor1}, we explicitly present the convergence rates by fixing the parameters and exploring several choices for $\beta_\tau$. In the (i) case, $ \beta_\tau $ is chosen to be a square-summable sequence. With this choice, we show that our method converges to a stationary point with a rate of $ \mathcal{O}(1/T) $, which matches the rate achieved by first-order SVRG/ProxSVRG algorithms in the same setting \cite{redd_svrg_nonconv,prox_svrg_nconv}. Moreover, our rate is better than the one achieved by stochastic finite-difference methods — see, e.g., \cite{ghadimi_lan,Ghadimi2016}. Notice that previous works \cite{ji_improved_vr,zo_svrg_avg,liu2018stochastic} do not analyze their algorithm selecting this parameter as a square-summable sequence. %
In point (ii) of the Corollary, we consider $\beta_\tau$ constant. With this choice, our method converges to a region close to a stationary point, with the radius depending on $\beta^2$. Notably, by fixing the total number of iterations $ T $ and setting $ \beta_\tau = \mathcal{O}(1/\sqrt{T m}) $, we find that the size of the error region depends on the total number of iterations as for \cite{ji_improved_vr,zo_svrg_avg,liu2018stochastic}.  %
In point (iii), we derive the complexity of our method. In particular, we find that our approach achieves better complexity than the methods proposed in \cite{zo_svrg_avg}, while matching the complexity of the algorithms introduced in \cite{ji_improved_vr}. However, the analysis of \cite{zo_svrg_avg,ji_improved_vr} does not hold for Problem \ref{eqn:problem}, where the function $h$ can be non-smooth. Moreover, the complexity they derive holds only for specific parameter choices, while in Corollary \ref{cor:cor1}, we propose a complexity holding for any choice of  $b$, $\ell$ and $m$. Our complexity matches also with the complexity of ZO-ProxSVRG (with Gaussian directions) \cite{Huang_Gu_Huo_Chen_Huang_2019}. However, to get such a result, the authors had to assume bounded norm of the stochastic gradients, an assumption that is not required in our analysis. In ZO-PSVRG+ with random directions \cite{Kazemi2024}, they must choose $b \propto n^{2/3} d$ to achieve our same complexity. So in our case, being $b$ smaller, we achieve the same complexity using a smaller batch size at each iteration. On the other hand $m=\sqrt{b}$ in both cases, requiring the call to the full gradient approximation more often in our case.  %
Notice that, even with this choice,  our algorithm requires performing less function evaluations per iteration since the full gradient approximation is computed using forward finite differences instead of central finite differences. Moreover, for $\ell > 1$, our algorithm allows for a larger stepsize than ZO-PSVRG+ (with random directions) \cite{Kazemi2024}. Moreover, from the general complexity result in point (iii), choosing $\ell=1$, $b \propto n^{2/3}d$, and $m=n^{1/3}$ still yields the optimal complexity. However, while these parameter choices are of the same order as those in previous works, the algorithm looses the advantages of structured directions. Indeed, %
for $\ell=1$, our algorithm recovers ZO-PSVRG+ with unstructured spherical directions (with the only difference that we employ a forward finite-difference estimator for the full gradient approximation instead of the central ones used in \cite{Kazemi2024}) and the resulting complexity reflects this equivalence. Therefore, our analysis recovers the existing results as a special case, showing that our framework is more general while preserving the same convergence rate.
Our results match those achieved by ZO-PSVRG (with coordinate directions) \cite{Kazemi2024}, ZO-SVRG-Coord \cite{ji_improved_vr}, and ZO-ProxSVRG (with coordinate directions) \cite{Huang_Gu_Huo_Chen_Huang_2019}. However, while these methods require using all $d$ canonical bases to approximate the stochastic gradients, our approach needs only $\ell \leq d$ structured directions, which makes our iteration cheaper in terms of the number of function evaluations per iteration. When comparing our results to the ones for O-ZD \cite{rando2023optimal} and ZO-GD \cite{nesterov2017random} algorithms, we observe that our method obtains a better dependence on the number of functions  from $n$ to $n^{2/3}$ in the complexity. %

\subsection{Non-convex Polyak-\L{}ojasiewicz Setting}
In this section, we analyze the performance of Algorithm \ref{algo:osvrz} under the (revised) Polyak-\L{}ojasiewicz (PL) condition. The standard PL condition \cite{polyak1987introduction} states that a differentiable function $F$ is $\mu$-PL if there exists $\mu > 0$ such that, for every $x \in \mathbb{R}^d$,  
\begin{equation}\label{eqn:std_pl}  
    \| \nabla F(x) \|^2 \geq 2 \mu \left(F(x) - \min F \right).  
\end{equation}  
Recall that if for instance $F$ is strongly convex than such a condition holds. However, since in problem \eqref{eqn:problem} the function $h$ is non-smooth, we cannot assume that $F$ satisfies this condition. Instead, we adapt it by replacing the gradient with the gradient mapping $\mathcal{G}_\gamma$ defined in eq. \eqref{eqn:generalized_gradient}. Formally, we consider functions satisfying the following hypothesis. 
\begin{assumption}[Revised Polyak-\L{}ojasiewicz]\label{ass:pl}
    The function $F$ is $\mu$-RPL; namely, it exists $\mu > 0$ such that for every $x \in \mathbb{R}^d$,
    \begin{equation*}
        \| \mathcal{G}_\gamma (x) \|^2 \geq 2 \mu \left( F(x) - \min F  \right).
    \end{equation*}
\end{assumption}
\noindent Notice that this assumption has been proposed and studied in several works \cite{Kazemi2024,prox_svrg_li}. Now, we state the theorem on the rate for non-convex RPL functions.
\begin{theorem}[Rates in Non-convex RPL Setting]\label{thm:pl}
Let Assumptions \ref{ass:l_smooth} and \ref{ass:pl} hold and let $\alpha = (1 - \frac{\gamma \mu}{2})$. If $\displaystyle \gamma < \min\left\{\frac{1}{4L}, \frac{2}{\mu}\frac{1}{(2m+1)}, \frac{\sqrt{b\theta}}{10mL}\sqrt{\frac{\ell}{d}}\right\}$, with $\theta = 1 - \frac{\mu \gamma}{2} - \mu\gamma m$, %
\begin{equation*}
    \begin{aligned}
    \mathbb{E}\left[F\left(x_{0}^{\tau+1}\right) - \min F \right] &\leq \alpha^{m(\tau+1)} \left( \mathbb{E}\left[F(x_0^\tau) - \min F \right]  +  \frac{1-\alpha^m}{1-\alpha}\gamma\left[\frac{3L^2 d^2}{\ell}  + \frac{L^2 d}{2} \right]\sum_{i=0}^{\tau}\frac{\beta_i^2}{\alpha^{m(i + 1)}} \right).        
    \end{aligned}
\end{equation*}
\end{theorem}
\begin{remark}
    Notice that the condition $\gamma < \frac{\sqrt{b\theta}}{10mL}\sqrt{\frac{\ell}{d}}$ and $\theta > 0$ holds for
    \begin{equation*}
        \gamma < \frac{ \sqrt{\left( \frac{\mu}{2} + \mu m \right)^2 b \ell^2 + 400 b m^2 L^2 d \ell} - \left( \frac{\mu}{2} + \mu m\right)b  }{200 m^2 L^2 d}.
    \end{equation*}
    However, we kept the recursive condition on the stepsize to improve readability and simplify the comparison with state-of-the-art methods since their stepsize is expressed in the same way.
\end{remark}
\noindent In the following corollary, we show rates for specific choices of the parameters.
\begin{corollary}\label{cor:cor2}
    Under the assumptions of Theorem \ref{thm:pl}, we have:\\ %
    (i) let $\beta_\tau = \alpha^{m(\tau + 1)/2} \pi_{\tau}$ with $\sum\limits_{i =0}^{+\infty} \pi_{i}^{2} \leq \tilde{C} < + \infty$. Then,
    \begin{equation*}
        \begin{aligned}
            \mathbb{E}\left[F\left(x_{0}^{\tau+1}\right) - \min F\right] &\leq %
            \alpha^{m(\tau+1)} \left( \left[F(x_0) - \min F\right] + \frac{2}{\mu}\left[\frac{3L^2 d^2}{\ell}  + \frac{L^2 d}{2} \right] \tilde{C} \right).%
        \end{aligned}
    \end{equation*}
    (ii) let $\beta_\tau = \beta$ with $\beta > 0$. Then
    \begin{equation*}
        \begin{aligned}
            \mathbb{E}\left[F\left(x_{0}^{\tau+1}\right) - \min F \right] &\leq \alpha^{m(\tau+1)}\left[F(x_0^\tau) - \min F \right] + \bar{C} \beta^2.%
        \end{aligned}
    \end{equation*}
    where $\bar{C} = \frac{2}{\mu} \left[\frac{3L^2 d^2}{\ell}  + \frac{L^2 d}{2} \right] > 0$.\\
    (iii) For $\varepsilon \in (0,1)$, fix $T \geq \mathcal{O}(\frac{1}{m \gamma \mu} \log \frac{1}{\varepsilon})$ and let $\beta_\tau = \sqrt{\frac{\varepsilon}{2 \bar{C}}}$.
    Then, we have
    \begin{equation*}
        \mathbb{E}[F(x^T_0) - \min F] \leq \varepsilon,
    \end{equation*}
    and the complexity is
    \begin{equation*}
        \mathcal{O} \left( \left(\frac{nd}{\gamma \mu m}  + \frac{b \ell}{\gamma \mu} \right) \log \frac{1}{\varepsilon} \right).
    \end{equation*}
\end{corollary}

\paragraph*{Discussion.} The rate in Theorem \ref{thm:pl} is composed of two terms: one dependent on the initialization and another representing the approximation error. The first term, since $\alpha < 1$, decreases exponentially fast, while the second term depends on the choice of the sequence $\beta_\tau$. This indicates that the convergence rate is strongly influenced by the choice of this parameter. Notice that, the result obtained matches with the results of first-order counterparts \cite{prox_svrg_nconv,prox_svrg_li}, with an additional error term. In Corollary \ref{cor:cor2}, we present rates for several choices of $\beta_\tau$. Under choice (i), where $\beta_\tau = \alpha^{(\tau + 1) m} \pi_{\tau}$ with $\pi_{\tau}$ a square-summable sequence, we observe that Algorithm \ref{algo:osvrz} converges to a global minimum exponentially fast. %
In point (ii), where $\beta_\tau$ is chosen as a constant $\beta > 0$, the algorithm converges to an error region near a stationary point, where the size of this region depends on $\beta^2$.  To the best of our knowledge, the only other zeroth-order methods providing analysis under Assumption \ref{ass:pl} are those introduced in \cite{Kazemi2024}. However, the authors focused solely on the complexity of these algorithms and did not perform the analysis under different choices of $\beta_\tau$. In particular, they did not consider to choose $\beta_\tau$ as a sequence, which would allow convergence to a global minimum rather than an error region. In point (iii), we analyze the complexity of Algorithm \ref{algo:osvrz}. Our results show that Algorithm \ref{algo:osvrz} achieves the same complexity as these methods. Moreover, notice that our stepsize can be larger than the one proposed in ProxSVRG+[RandSGE] \cite{Kazemi2024}. It is also larger than the one proposed in ProxSVRG+ (with coordinate directions) \cite{Kazemi2024}. Notice that the complexity term depends on $\varepsilon$ as $\log(1/\varepsilon)$, as for the first-order ProxSVRG in the same setting \cite{prox_svrg_li}.

\section{Experiments}\label{sec:experiments}
In this section, we present numerical experiments to evaluate the performance of our algorithm. Specifically, we compare Algorithm \ref{algo:osvrz} with ZO-ProxSVRG+\cite{Kazemi2024,Huang_Gu_Huo_Chen_Huang_2019}, ZO-PSpider+ \cite{Kazemi2024} and RSPGF \cite{Ghadimi2016}. 
We consider three experiments: minimization of synthetic functions, generation of universal adversarial perturbations and binary classification. 
\begin{table}[ht]
    \centering
    \caption{Number of (stochastic) function evaluations required to perform an (outer) iteration.}
    \label{tab:algo_cost}
    \begin{tabular}{ll}
        \toprule
        Algorithm & Number of function evaluations per iteration \\
        \midrule
        RSPGF \cite{Ghadimi2016} & $\ell + 1$\\
        ZO-PSVRG (GaussSGE) \cite{Huang_Gu_Huo_Chen_Huang_2019}& $2 nd + 4 m b$\\
        ZO-PSVRG+ (RandSGE) \cite{Kazemi2024}& $2  nd + 4 m b$\\
        ZO-PSVRG+ (CoordSGE) \cite{Kazemi2024,Huang_Gu_Huo_Chen_Huang_2019}& $2 nd + 4 m b d$\\
        ZO-PSpider+ (RandSGE) \cite{Kazemi2024} & $2 nd + 4 m b$\\
        ZO-PSpider+ (CoordSGE) \cite{Kazemi2024} & $2nd + 4 m b d$\\
        {\bf VR-SZD (our)} & $n(d + 1) + 2 m b (\ell + 1)$\\
        \bottomrule
    \end{tabular}
\end{table}
To present our findings, we run every experiment $10$ times and report the mean and standard deviation of the results. For every experiment, the smoothing parameters sequence $\beta_\tau = 10^{-5}$ for every $\tau \in \mathbb{N}$.  
Results are reported in function of number of (stochastic) function evaluations required to perform each experiment and, since
for every algorithm performing an outer iteration requires performing several (stochastic) function evaluations, the values are repeated by that amount - see Table \ref{tab:algo_cost}. Details on the objective functions and the choice of the parameters for the algorithms are reported in Appendix \ref{app:exp_details}.

\paragraph*{Synthetic experiments.} We compare our method against several state-of-the-art zeroth-order optimization algorithms on three synthetic optimization problems: least-squares, logistic, and phase retrieval targets, each regularized with the $\ell_1$ norm of the input - see Appendix~\ref{app:exp_details} for details. We fix a budget of $10^7$ function evaluations and consider an input space with dimension $d=100$. For each algorithm, all hyper-parameters are selected via grid search.
\begin{figure}[H]
\centering
\includegraphics[width=0.32\linewidth]{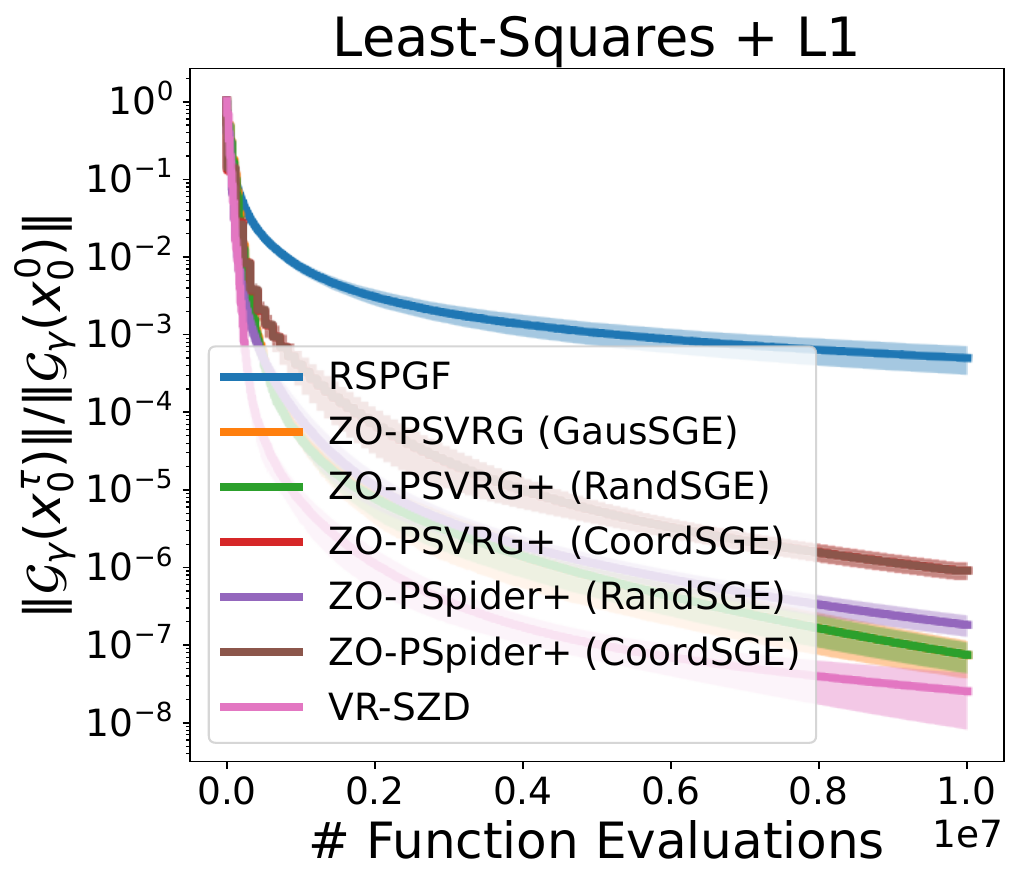}
\includegraphics[width=0.32\linewidth]{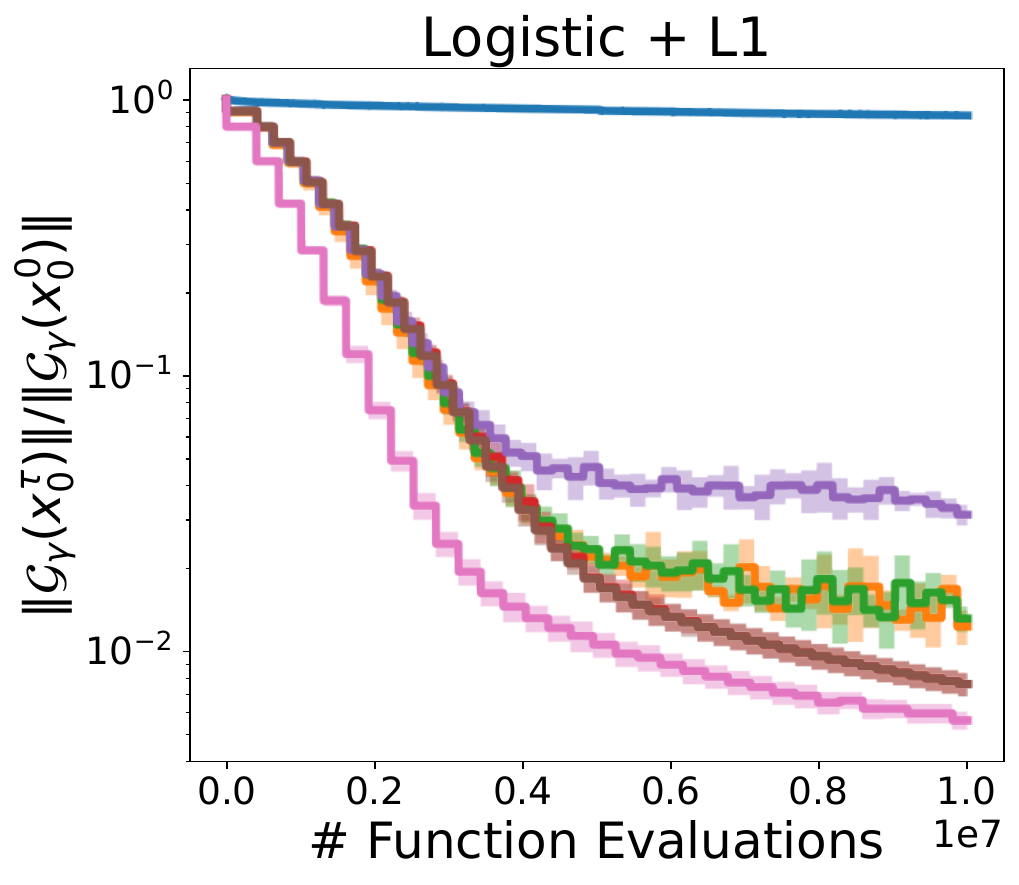}
\includegraphics[width=0.32\linewidth]{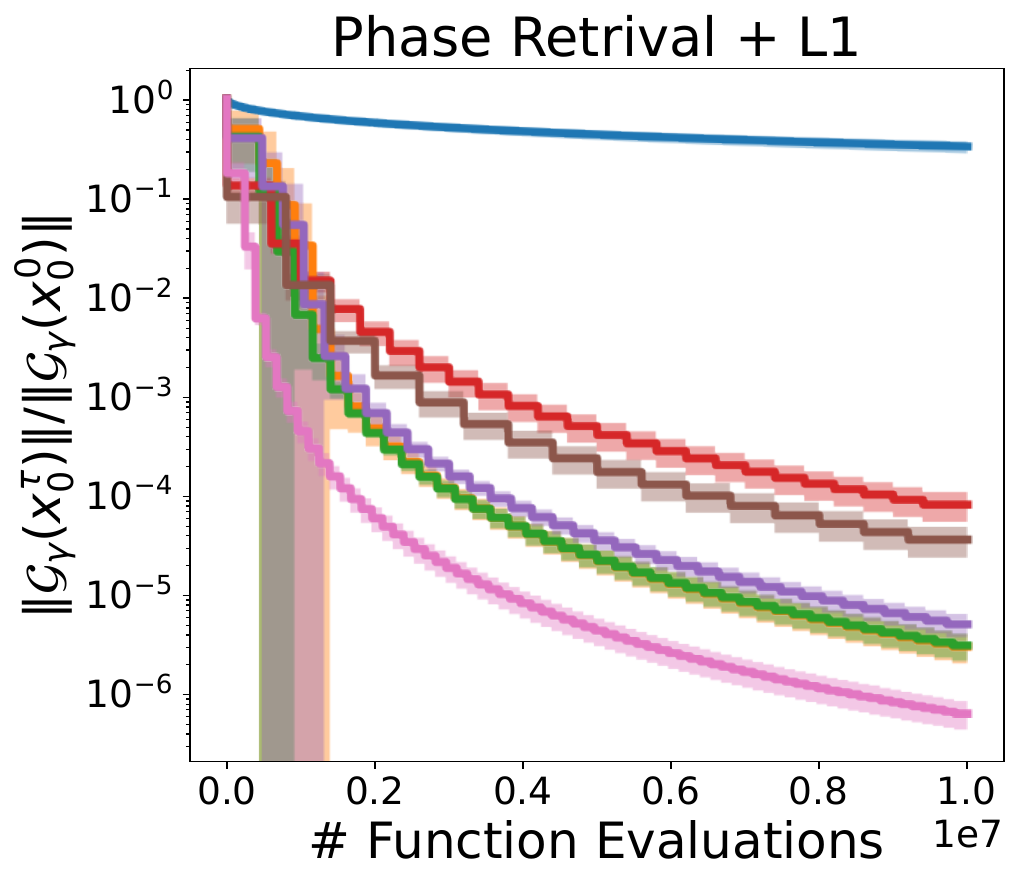}
\caption{Average normalized norm of the generalized gradient mapping as a function of the number of (stochastic) function evaluations for the least-squares, logistic, and phase retrieval problems.}
\label{fig:syn_comp}
\end{figure}
\noindent Figure~\ref{fig:syn_comp} reports the average normalized norm of the generalized gradient mapping as a function of the number of function evaluations. We observe that methods that incorporate variance reduction perform consistently better than RSPGF, which does not employ any variance reduction mechanism. Among the variance-reduced methods, VR-SZD consistently achieves the best performance on all three benchmark problems. This behavior can be explained by several reasons. First, unlike competing variance-reduced methods, VR-SZD provides two independent mechanisms to control the accuracy of the gradient surrogates, namely the mini-batch size $b$ and the number of sampled directions $\ell$. This additional degree of freedom allows for a more flexible allocation of the available function-evaluation budget between gradient estimation and optimization. Moreover, even when fixing $\ell=1$ (which corresponds to the design choice adopted by existing variance-reduced methods based on unstructured random directions), VR-SZD still requires fewer function evaluations to compute each outer-loop full-gradient approximation. This is a consequence of employing a forward finite-difference estimator instead of the central finite-difference estimator used by competing approaches, resulting in a lower cost per outer iteration - see Table \ref{tab:algo_cost}. Furthermore, when a small mini-batch size is employed (e.g., $b=1$), VR-SZD can still improve the quality of the stochastic gradient surrogate by increasing the number of sampled directions $\ell$. In contrast, competing methods can only improve the gradient approximation by increasing the mini-batch size. Notice that the function evaluation cost of VR-SZD scales more favorably with $\ell$ than the cost scaling in $b$ of competing methods. In addition, the stochastic gradient estimators in VR-SZD are constructed using structured directions rather than unstructured random directions. As shown in previous work on finite-difference methods~\cite{berahas2022theoretical,rando2025structuredtouroptimizationfinite}, structured directions typically provide more accurate gradient approximations using fewer directions. %
Consequently, VR-SZD can dedicate a larger fraction of the budget to optimization iterations, allowing to get more accurate solutions. We also observe that VR-SZD achieves better performance than ZO-PSVRG+ (CoordSGE) and ZO-PSpider+ (CoordSGE), although these methods employ coordinate directions for stochastic gradient estimation. While coordinate-based estimators provide highly accurate gradient approximations, each stochastic gradient estimate requires $\mathcal{O}(d)$ function evaluations, leading to a substantially higher per-iteration cost and a less favorable scaling with the mini-batch size. As a result, these methods exhaust the available computational budget after relatively few optimization iterations. In contrast, the additional flexibility provided by the independent tuning of $b$ and $\ell$, together with the use of structured directions, allows VR-SZD to better control the computational cost while maintaining accurate gradient estimates, resulting in improved practical performance.

\paragraph*{Universal Adversarial Perturbation.} Here, we evaluate our method against state-of-the-art approaches on generating universal adversarial perturbations \cite{moosavidezfooli2017universal} for a multiclass MNIST classifier \cite{lecun2010mnist}, trained with defensive distillation \cite{papernot2016distillation} to increase its robustness to adversarial examples. The goal of the problem is to find a single perturbation $\delta$, shared across all images belonging to a given class, that maximizes the misclassification rate of the classifier while keeping the magnitude $\|\delta\|$ small, the latter being enforced as a soft constraint through a penalty parameter $\lambda= 0.5$ - see Appendix \ref{app:details_adv} for further details. %
We fix a budget of $10^7$ function evaluations and run every algorithm with parameters tuned via gridsearch. We consider finding universal perturbations for images labeled as $4$, $9$, $1$. Figure \ref{fig:adv_exp} reports the average normalized function values (top row) and the classification accuracy (bottom row) as functions of the number of function evaluations.

\begin{figure}[H]
\centering
\includegraphics[width=0.32\linewidth]{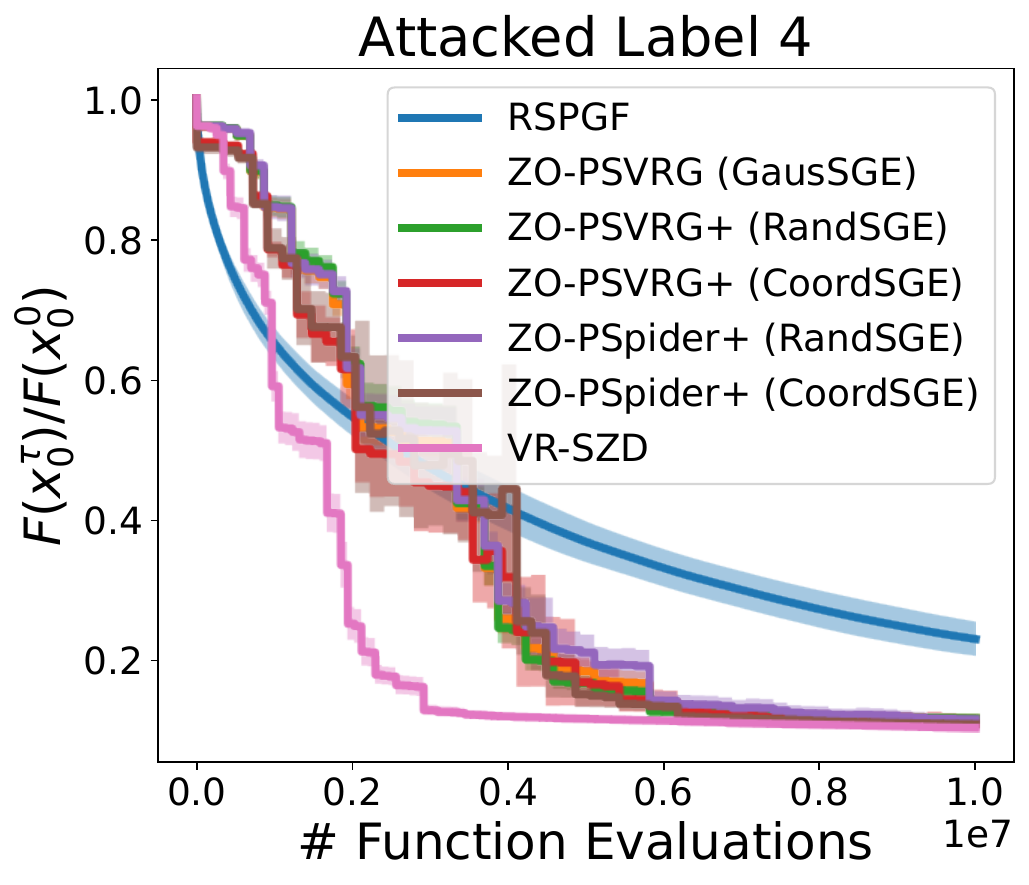}
\includegraphics[width=0.32\linewidth]{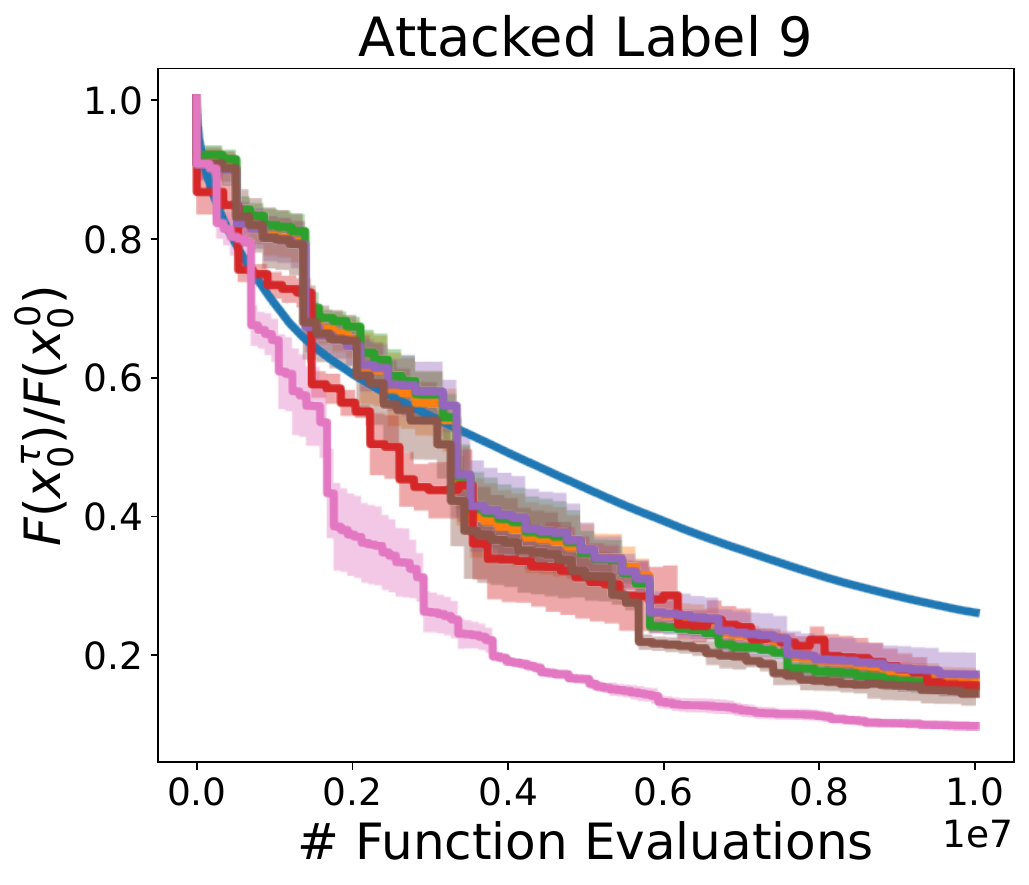}
\includegraphics[width=0.32\linewidth]{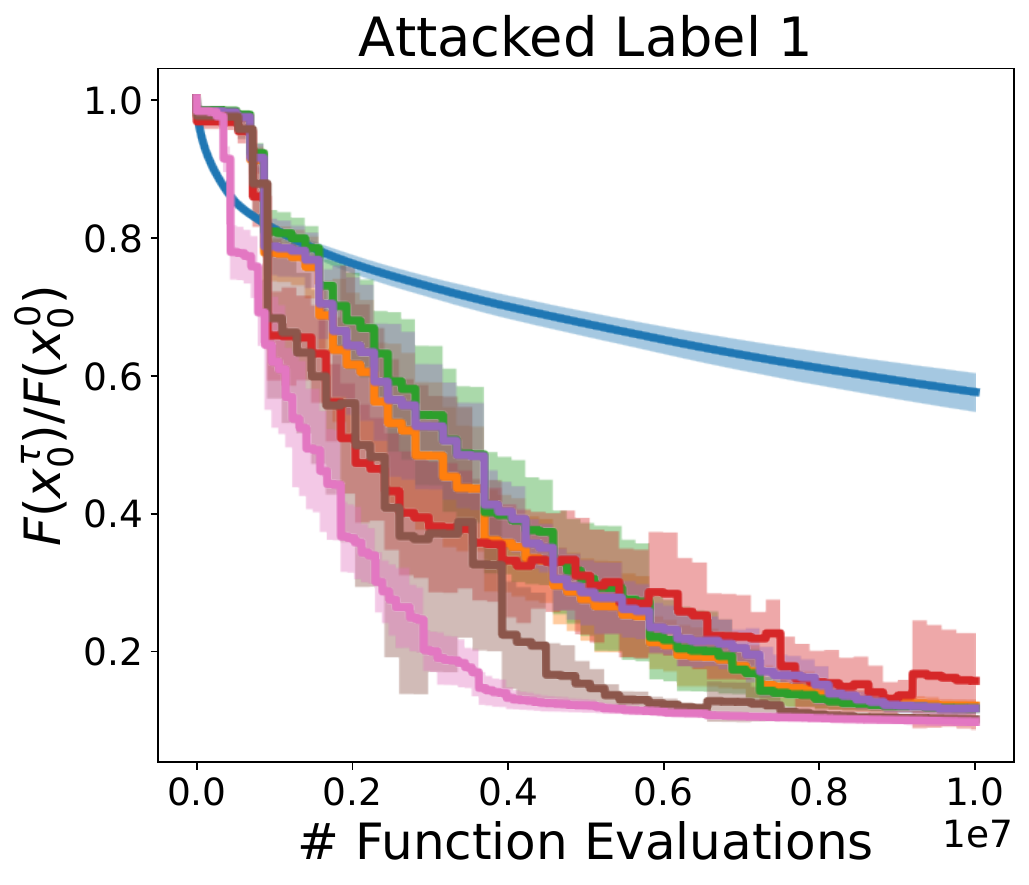}\\
\includegraphics[width=0.32\linewidth]{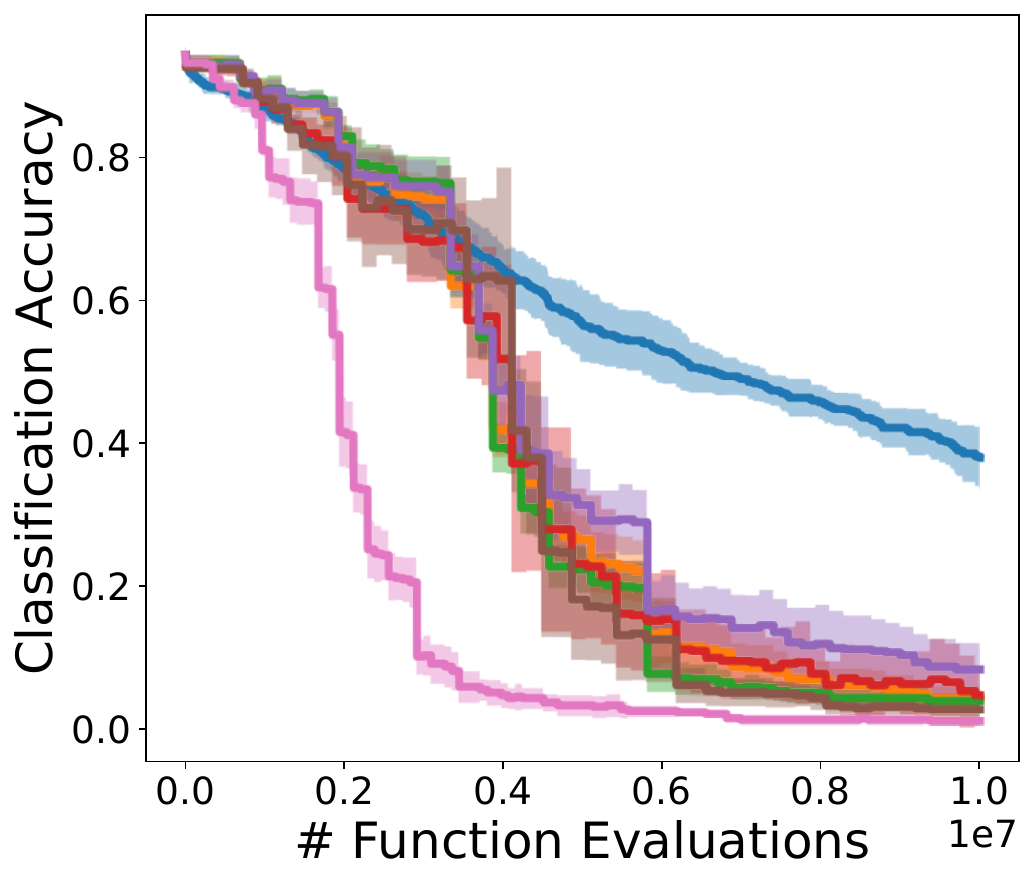}
\includegraphics[width=0.32\linewidth]{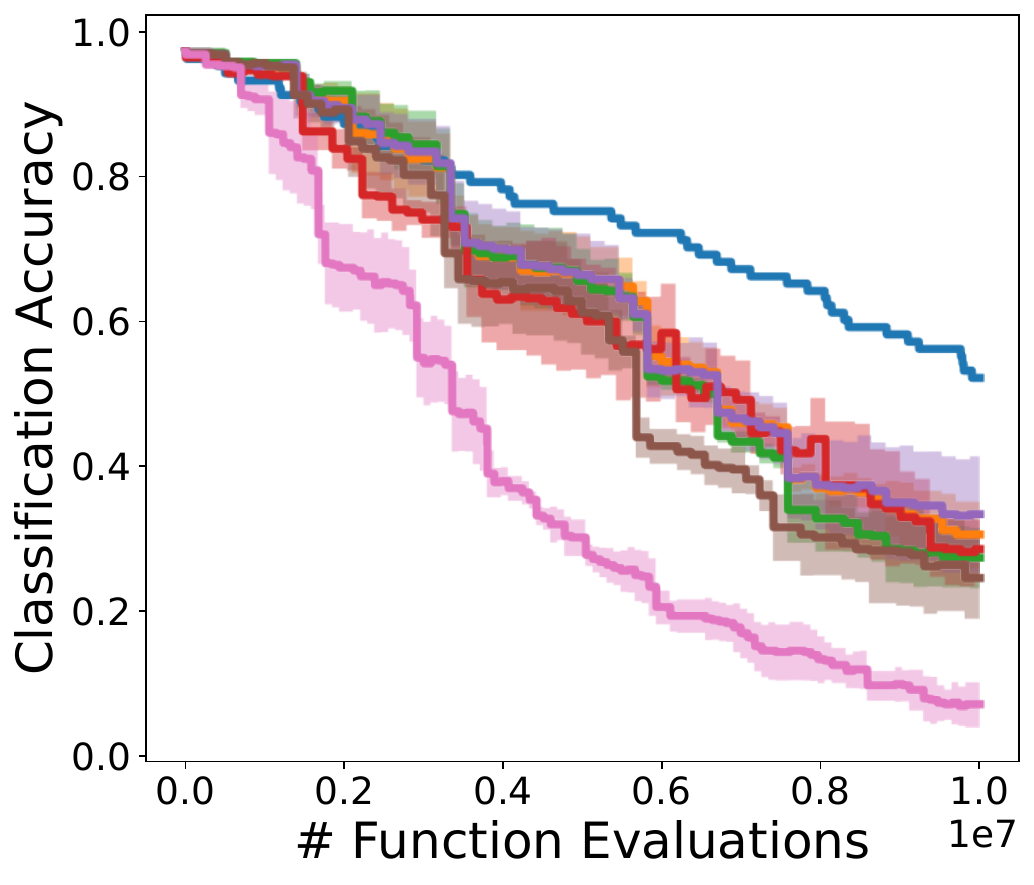}
\includegraphics[width=0.32\linewidth]{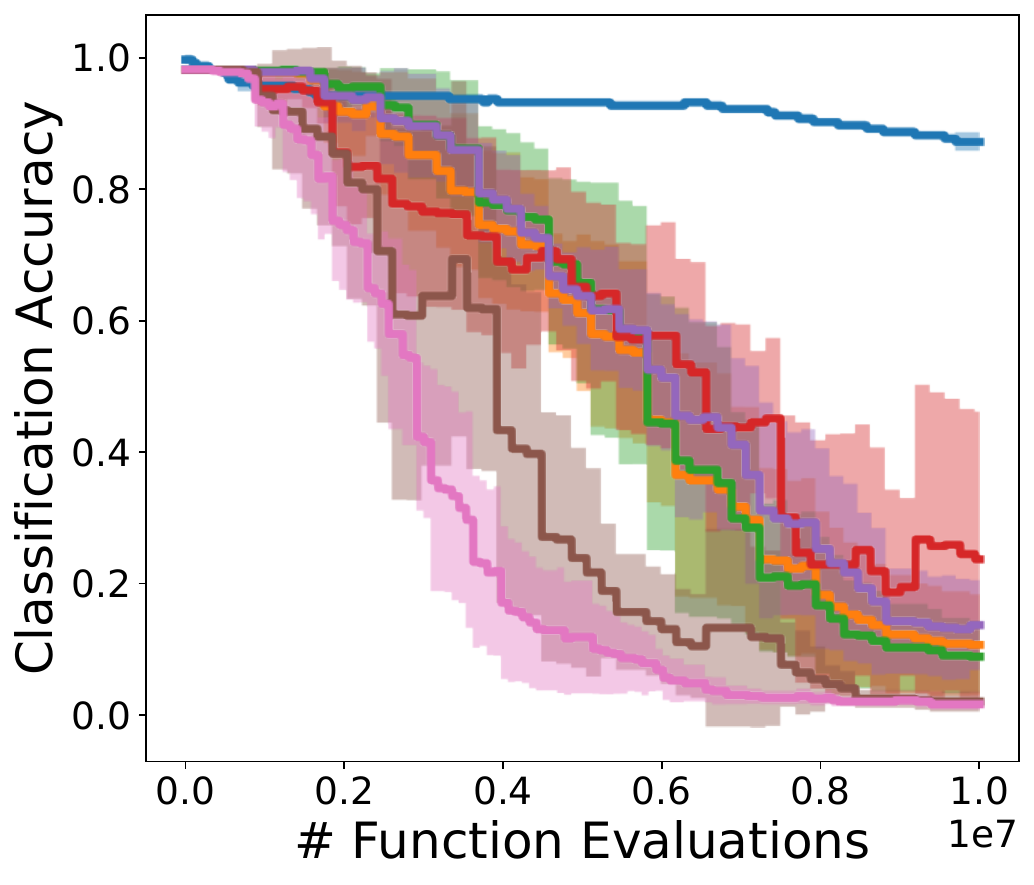}
\caption{Normalized function values $F(x^\tau_0)/F(x_0^0)$ (top row) and of the resulting classification accuracy on the attacked images (bottom row) as a function of the number of function evaluations, for the labels $4$, $9$ and $1$ (left to right).}%
\label{fig:adv_exp}
\end{figure}
 \noindent As expected, RSPGF performs worse than the variance-reduced algorithms, since it does not benefit from any variance-reduction mechanism. We observe that, for all three attacked labels, VR-SZD performs better than the competing variance-reduced methods. This is because structured directions allow it to use fewer evaluations to estimate stochastic gradients while still keeping good accuracy. This allows to choose a larger stepsize while keeping the cost per iteration contained and thus get more refinements. In the bottom row of Figure \ref{fig:adv_exp}, we observe that our method provides a faster decrease of the classification accuracy, also coherently with the results in function values.

\paragraph*{Black-box Binary Classification.} Now, we compare the different algorithms in solving a binary classification problem - see Appendix \ref{app:details_bbclass} for details on the objective function. %
We run the different algorithms fixing a budget of $10^7$ function evaluations and tuning the parameters of every algorithm via grid-search. %
In Figure \ref{fig:bbclass}, ignoring RSPGF, which does not exploit any variance reduction technique, we observe that methods using all coordinate directions perform worse than other variance-reduced algorithms. This is because the per-iteration cost is very high, resulting in fewer iterations being performed and, consequently, less accurate results, despite relying on better (stochastic) gradient approximations. We also observe that our algorithm performs better than the other methods. This can again be attributed to the fact that the iterations of our algorithm require fewer function evaluations and that the (stochastic) gradient approximations are constructed using structured directions.
\begin{figure}[H]
    \centering
    \includegraphics[width=0.45\linewidth]{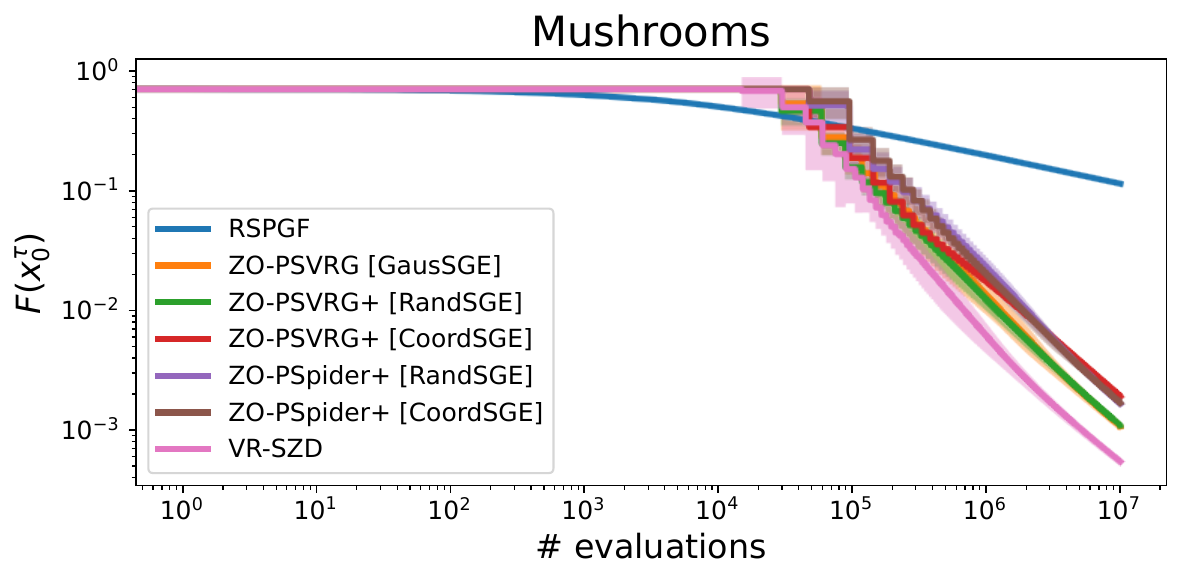}
    \includegraphics[width=0.45\linewidth]{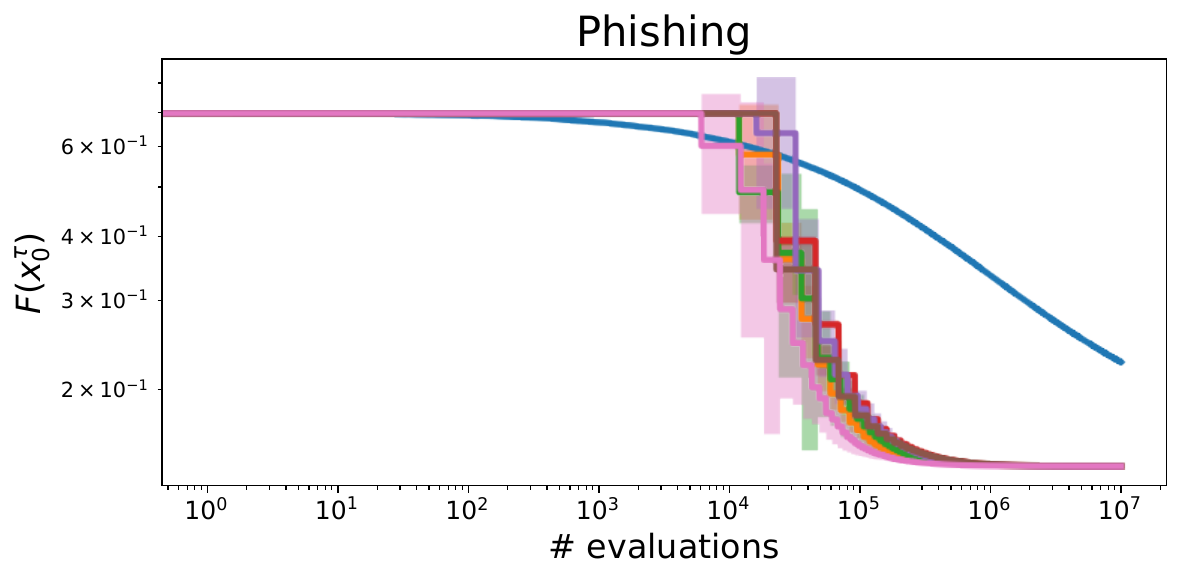}
    \includegraphics[width=0.45\linewidth]{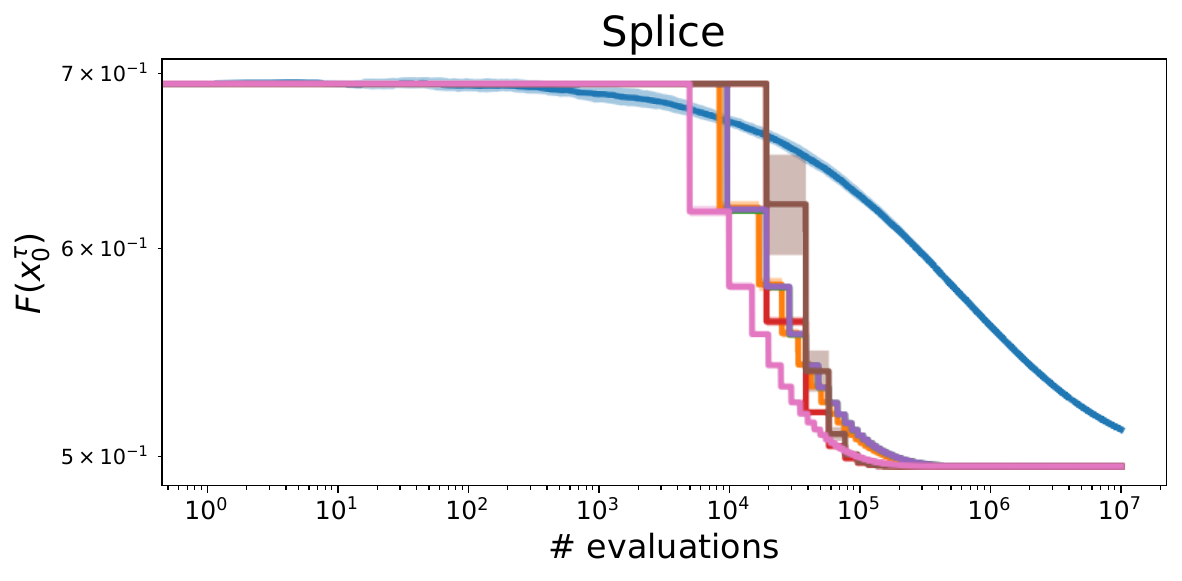}
    \includegraphics[width=0.45\linewidth]{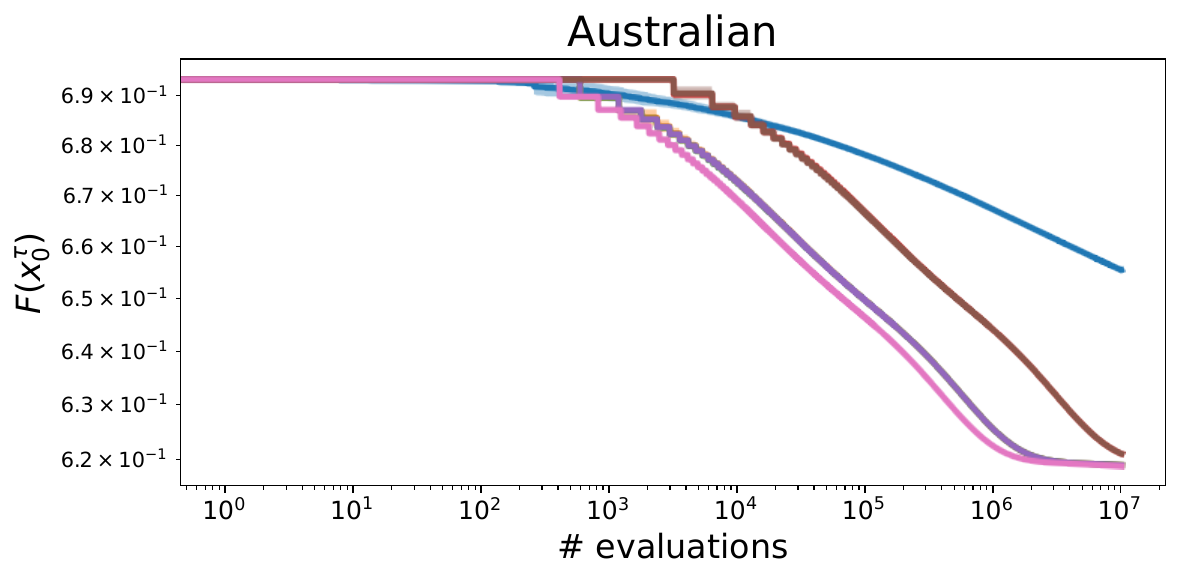}
    \caption{Comparison of different algorithms in solving the black-box binary classification problem.}%
    \label{fig:bbclass}
\end{figure}

\section{Conclusions}\label{sec:conclusion}
We introduced and analyzed VR-SZD, a structured variance-reduced finite-difference algorithm for zeroth-order optimization of non-convex finite-sum targets. We %
derived convergence rates for both non-convex and PL non-convex functions. This work opens several promising research directions. For instance, an interesting one would be the development of adaptive procedures to select the parameters or studying methods to replace the full-gradient approximation with a more cost-effective surrogate. Another potential research direction consists of extending the framework to a more general setting where the non-smooth function $h$ is a finite sum of non-smooth functions, utilizing stochastic proximal operators \cite{Traoré2024}.

\paragraph*{Acknowledgements \& Fundings.}
L. R. acknowledges the financial support of the European Commission (Horizon Europe grant ELIAS 101120237), the Ministry of Education, University and Research (FARE grant ML4IP R205T7J2KP) and the Center for Brains, Minds and Machines (CBMM), funded by NSF STC award CCF-1231216. S. V. acknowledges the support of the European Commission (grant TraDE-OPT 861137). The research by S. V., C. M. and C. T. has been supported by the MUR Excellence Department Project awarded to Dipartimento di Matematica, Universita di Genova, CUP D33C23001110001. L. R., S. V., C. M. acknowledge the financial support of the European Research Council (grant SLING 819789). L. R., S. V., C. M., C. T. acknowledge the support of the US Air Force Office of Scientific Research (FA8655-22-1-7034), the Ministry of Education, University and Research (grant BAC FAIR PE00000013 funded by the EU - NGEU) and MIUR (PRIN 202244A7YL). C. M. and S. V. are members of the Gruppo Nazionale per l’Analisi Matematica, la Probabilità e le loro Applicazioni (GNAMPA) of the Istituto Nazionale di Alta Matematica (INdAM). C. M. was also supported by the Programma Operativo Nazionale (PON) “Ricerca 
e Innovazione” 2014–2020. M.R. was also supported by a France 2030 support managed by the Agence Nationale de la Recherche, under the reference ANR-23-PEIA-0004 (PDE-AI project). This work represents only the view of the authors. The European Commission and the other organizations are not responsible for any use that may be made of the information it contains.\\

\paragraph*{Data Availability.} Datasets used in the binary classification experiments in Section \ref{sec:experiments}, can be downloaded from LIBSVM \cite{libsvm} website: \url{https://www.csie.ntu.edu.tw/~cjlin/libsvmtools/datasets/binary.html} - see Appendix \ref{app:exp_details} for further details.\\

\paragraph*{Code Availability.} The code used to perform the experiments is provided in the following GitHub repository: \url{https://github.com/MarcoRando/vr_szd}.

\bibliographystyle{plain}
\bibliography{refs}

@misc{rando2026zobaefficientsingleloopzerothorder,
      title={ZOBA: An Efficient Single-loop Zeroth-order Bilevel Optimization Algorithm}, 
      author={Marco Rando and Samuel Vaiter},
      year={2026},
      eprint={2601.21836},
      archivePrefix={arXiv},
      primaryClass={math.OC},
      url={https://arxiv.org/abs/2601.21836}, 
}

@book{mattila_1995, place={Cambridge}, 
series={Cambridge Studies in Advanced Mathematics}, title={Geometry of Sets and Measures in Euclidean Spaces: Fractals and Rectifiability}, 
DOI={10.1017/CBO9780511623813}, 
publisher={Cambridge University Press}, 
address ={Cambridge, United Kingdom},
author={Mattila, P.}, year={1995}, collection={Cambridge Studies in Advanced Mathematics}}

@article{nesterov2017random,
  title={Random gradient-free minimization of convex functions},
  author={Nesterov, Y. and Spokoiny, V.},
  journal={Foundations of Computational Mathematics},
  volume={17},
  pages={527--566},
  year={2017},
  publisher={Springer},
  doi = {10.1007/s10208-015-9296-2}
}

@article{lecun2010mnist,
  title={MNIST handwritten digit database},
  author={LeCun, Yann and Cortes, Corinna and Burges, CJ},
  journal={ATT Labs [Online]. Available: http://yann.lecun.com/exdb/mnist},
  volume={2},
  year={2010}
}

@inproceedings{moosavidezfooli2017universal,
  author    = {Moosavi-Dezfooli, Seyed-Mohsen and Fawzi, Alhussein and Fawzi, Omar and Frossard, Pascal},
  title     = {Universal Adversarial Perturbations},
  booktitle = {2017 IEEE Conference on Computer Vision and Pattern Recognition (CVPR)},
  pages     = {86--94},
  year      = {2017},
  publisher = {IEEE},
  doi       = {10.1109/CVPR.2017.17}
}

@inproceedings{papernot2016distillation,
  title={Distillation as a defense to adversarial perturbations against deep neural networks},
  author={Papernot, Nicolas and McDaniel, Patrick and Wu, Xi and Jha, Somesh and Swami, Ananthram},
  booktitle={2016 IEEE symposium on security and privacy (SP)},
  pages={582--597},
  year={2016},
  organization={IEEE}
}

@inproceedings{gfm_lin_zheng_jordan,
 author = {Lin, T. and Zheng, Z. and Jordan, M.},
 booktitle = {Advances in Neural Information Processing Systems},
 editor = {Koyejo, S. and Mohamed, S. and Agarwal, A.  and Belgrave, D. and Cho, K. and Oh, A.},
 pages = {26160--26175},
 publisher = {Curran Associates, Inc.},
 address   = {Red Hook, NY, USA},
 title = {Gradient-Free Methods for Deterministic and Stochastic Nonsmooth Nonconvex Optimization},
 url = {\url{https://proceedings.neurips.cc/paper_files/paper/2022/file/a78f142aec481e68c75276756e0a0d91-Paper-Conference.pdf}},
 volume = {35},
 year = {2022}
}

@inproceedings{saga,
 author = {Defazio, A. and Bach, F. and Lacoste-Julien, S.},
 booktitle = {Advances in Neural Information Processing Systems},
 editor = {Z. Ghahramani and M. Welling and C. Cortes and N. Lawrence and K.Q. Weinberger},
 pages = {},
 title = {SAGA: A Fast Incremental Gradient Method With Support for Non-Strongly Convex Composite Objectives},
url = {\url{https://proceedings.neurips.cc/paper_files/paper/2014/hash/937964195d6fb3a55cd7cc578165f058-Abstract.html}},
 volume = {27},
 year = {2014}
}

@article{Huang_Gu_Huo_Chen_Huang_2019, title={Faster Gradient-Free Proximal Stochastic Methods for Nonconvex Nonsmooth Optimization}, volume={33}, url={https://ojs.aaai.org/index.php/AAAI/article/view/3963}, DOI={10.1609/aaai.v33i01.33011503}, abstractNote={&lt;p&gt;Proximal gradient method has been playing an important role to solve many machine learning tasks, especially for the nonsmooth problems. However, in some machine learning problems such as the bandit model and the black-box learning problem, proximal gradient method could fail because the explicit gradients of these problems are difficult or infeasible to obtain. The gradient-free (zeroth-order) method can address these problems because only the objective function values are required in the optimization. Recently, the first zeroth-order proximal stochastic algorithm was proposed to solve the nonconvex nonsmooth problems. However, its convergence rate is &lt;em&gt;O&lt;/em&gt;(1/√T) for the nonconvex problems, which is significantly slower than the best convergence rate &lt;em&gt;O&lt;/em&gt;(&lt;sub&gt;&lt;em&gt;T&lt;/em&gt;&lt;/sub&gt;&lt;sup&gt;1&lt;/sup&gt;) of the zerothorder stochastic algorithm, where &lt;em&gt;T&lt;/em&gt; is the iteration number. To fill this gap, in the paper, we propose a class of faster zeroth-order proximal stochastic methods with the variance reduction techniques of SVRG and SAGA, which are denoted as ZO-ProxSVRG and ZO-ProxSAGA, respectively. In theoretical analysis, we address the main challenge that an unbiased estimate of the true gradient does not hold in the zerothorder case, which was required in previous theoretical analysis of both SVRG and SAGA. Moreover, we prove that both ZO-ProxSVRG and ZO-ProxSAGA algorithms have &lt;em&gt;O&lt;/em&gt;(&lt;sub&gt;&lt;em&gt;T&lt;/em&gt;&lt;/sub&gt;&lt;sup&gt;1&lt;/sup&gt;) convergence rates. Finally, the experimental results verify that our algorithms have a faster convergence rate than the existing zeroth-order proximal stochastic algorithm.&lt;/p&gt;}, number={01}, journal={Proceedings of the AAAI Conference on Artificial Intelligence}, author={Huang, F. and Gu, B. and Huo, Z. and Chen, S. and Huang, H.}, year={2019}, month={Jul.}, pages={1503-1510} }

@article{zoro,
		author = {Cai, H. and McKenzie, D. and Yin, W. and Zhang, Z.},
		title = {Zeroth-Order Regularized Optimization (ZORO): Approximately Sparse Gradients and Adaptive Sampling},
		journal = {SIAM Journal on Optimization},
		volume = {32},
		number = {2},
		pages = {687-714},
		year = {2022},
		doi = {10.1137/21M1392966},
		
		URL = { 
			https://doi.org/10.1137/21M1392966
			
		}
	}

@article{lojasiewicz1963topological,
  title={A topological property of real analytic subsets},
  author={Lojasiewicz, S.},
  journal={Coll. du CNRS, Les {\'e}quations aux d{\'e}riv{\'e}es partielles},
  volume={117},
  number={87-89},
  pages={2},
  year         = {1963},
  publisher    = {Centre national de la recherche scientifique},
  address      = {Paris, France}
}

@InProceedings{ZO-BCD,
	title = 	 {A Zeroth-Order Block Coordinate Descent Algorithm for Huge-Scale Black-Box Optimization},
	author =       {Cai, H. and Lou, Y. and Mckenzie, D. and Yin, W.},
	booktitle = 	 {Proceedings of the 38th International Conference on Machine Learning},
	pages = 	 {1193--1203},
	year = 	 {2021},
	volume = 	 {139},
	series = 	 {Proceedings of Machine Learning Research},
	month = 	 {18--24 Jul},
	url = 	 {\url{https://proceedings.mlr.press/v139/cai21d.html}}
	}

@Article{Ghadimi2016,
author={Ghadimi, S. and Lan, G. and Zhang, H.},
title={Mini-batch stochastic approximation methods for nonconvex stochastic composite optimization},
journal={Mathematical Programming},
year={2016},
month={Jan},
day={01},
volume={155},
number={1},
pages={267-305},
issn={1436-4646},
doi={10.1007/s10107-014-0846-1},
url={https://doi.org/10.1007/s10107-014-0846-1}
}

@misc{sartore2024automaticgaintuninghumanoid,
author={Sartore, C. and Rando, M. and Romualdi, G. and Molinari, C. and Rosasco, L. and Pucci, D.},
booktitle={2024 IEEE-RAS 23rd International Conference on Humanoid Robots (Humanoids)}, 
title={Automatic Gain Tuning for Humanoid Robots Walking Architectures Using Gradient-Free Optimization Techniques}, 
year={2024},
volume={},
number={},
pages={996-1003},
doi={10.1109/Humanoids58906.2024.10769876}
}

@book{intro_schein,
author = {Conn, A. R. and Scheinberg, K. and Vicente, L. N.},
title = {Introduction to Derivative-Free Optimization},
publisher = {Society for Industrial and Applied Mathematics},
year = {2009},
doi = {10.1137/1.9780898718768},
address = {Philadelphia, PA},
edition   = {},
eprint = {10.1137/1.9780898718768}
}

@InProceedings{pmlr-v80-ilyas18a,
  title = 	 {Black-box Adversarial Attacks with Limited Queries and Information},
  author =       {Ilyas, A. and Engstrom, L. and Athalye, A. and Lin, J.},
  booktitle = 	 {Proceedings of the 35th International Conference on Machine Learning},
  pages = 	 {2137--2146},
  year = 	 {2018},
  editor = 	 {Dy, Jennifer and Krause, Andreas},
  volume = 	 {80},
  series = 	 {Proceedings of Machine Learning Research},
  month = 	 {10--15 Jul},
  url = 	 {\url{https://proceedings.mlr.press/v80/ilyas18a.html}}
}

@InProceedings{pmlr-v97-guo19a,
  title = 	 {Simple Black-box Adversarial Attacks},
  author =       {Guo, C. and Gardner, J. and You, Y. and Wilson, A. G. and Weinberger, K.},
  booktitle = 	 {Proceedings of the 36th International Conference on Machine Learning},
  pages = 	 {2484--2493},
  year = 	 {2019},
  editor = 	 {Chaudhuri, Kamalika and Salakhutdinov, Ruslan},
  volume = 	 {97},
  series = 	 {Proceedings of Machine Learning Research},
  month = 	 {09--15 Jun},
  publisher =    {PMLR},
  address   = {Brooklyn, NY},
  url = 	 {\url{https://proceedings.mlr.press/v97/guo19a.html}}
}

@inbook{tutorial_BO,
author = {Frazier, P. I.},
title = {Bayesian Optimization},
booktitle = {Recent Advances in Optimization and Modeling of Contemporary Problems},
chapter = {11},
pages = {255-278},
year={2018},
publisher ={INFORMS},
address   = {Catonsville, MD},
doi = {10.1287/educ.2018.0188},
eprint = {https://pubsonline.informs.org/doi/pdf/10.1287/educ.2018.0188},
}

@article{LEWIS2000191,
title = {Direct search methods: then and now},
journal = {Journal of Computational and Applied Mathematics},
volume = {124},
number = {1},
pages = {191-207},
year = {2000},
note = {Numerical Analysis 2000. Vol. IV: Optimization and Nonlinear Equations},
issn = {0377-0427},
doi = {https://doi.org/10.1016/S0377-0427(00)00423-4},
url = {https://www.sciencedirect.com/science/article/pii/S0377042700004234},
author = {Lewis, R. M. and 
          Torczon, V. and 
          Trosset, M. W.},

}

@Inbook{Totzeck2022,
author="Totzeck, C.",
title="Trends in Consensus-Based Optimization",
bookTitle="Active Particles, Volume 3: Advances in Theory, Models, and Applications",
year="2022",
publisher="Springer International Publishing",
address="Cham",
pages="201--226",
isbn="978-3-030-93302-9",
doi="10.1007/978-3-030-93302-9_6"
}

@InProceedings{pmlr-v151-rando22a,
  title = 	 { Ada-BKB: Scalable Gaussian Process Optimization on Continuous Domains by Adaptive Discretization },
  author =       {Rando, M. and Carratino, L. and Villa, S. and Rosasco, L.},
booktitle =  {Proceedings of The 25th International Conference on Artificial Intelligence and Statistics},  
pages =  {7320--7348},  
year =  {2022},  
editor =  {Camps-Valls, G. and Ruiz, F. J. R. and Valera, I.},  
volume =  {151},  
series =  {Proceedings of Machine Learning Research},  
month =  {28--30 Mar},  
publisher =    {PMLR},  
address = {Brooklyn, NY},  
url =  {\url{https://proceedings.mlr.press/v151/rando22a.html}}

}

@article{demetrio2021functionality,
  author={Demetrio, L. and Biggio, B. and Lagorio, G. and Roli, F. and Armando, A.},
  journal={IEEE Transactions on Information Forensics and Security}, 
  title={Functionality-Preserving Black-Box Optimization of Adversarial Windows Malware}, 
  year={2021},
  volume={16},
  number={},
  pages={3469-3478},
  doi={10.1109/TIFS.2021.3082330}
}

@INPROCEEDINGS{carlini_wagner,
  author={Carlini, Nicholas and Wagner, David},
  booktitle={2017 IEEE Symposium on Security and Privacy (SP)}, 
  title={Towards Evaluating the Robustness of Neural Networks}, 
  year={2017},
  volume={},
  number={},
  pages={39-57},
  doi={10.1109/SP.2017.49}}

@inproceedings{zoo,
author = {P.Y. Chen and H. Zhang and Y. Sharma and J. Yi and C.J. Hsi},
title = {ZOO: Zeroth Order Optimization Based Black-box Attacks to Deep Neural Networks without Training Substitute Models},
year = {2017},
isbn = {9781450352024},
publisher = {Association for Computing Machinery},
address = {New York, NY, USA},
url = {https://doi.org/10.1145/3128572.3140448},
doi = {10.1145/3128572.3140448},
booktitle = {Proceedings of the 10th ACM Workshop on Artificial Intelligence and Security},
pages = {15--26},
numpages = {12},
location = {Dallas, Texas, USA},
series = {AISec '17}
}

@misc{salimans2017evolution,
      title={Evolution Strategies as a Scalable Alternative to Reinforcement Learning}, 
      author={Salimans, T. and 
              Ho, J. and  
              Chen, X. and 
              Sidor, S. and 
              Sutskever, I.},

      year={2017},
      eprint={1703.03864},
      archivePrefix={arXiv},
      primaryClass={stat.ML},
      url={https://arxiv.org/abs/1703.03864}, 

}

@ARTICLE{rando2024newformulationzerothorderoptimization,
  author={Rando, Marco and Demetrio, Luca and Rosasco, Lorenzo and Roli, Fabio},
  journal={IEEE Transactions on Information Forensics and Security}, 
  title={A New Formulation for Zeroth-Order Optimization of Adversarial EXEmples in Malware Detection}, 
  year={2026},
  volume={21},
  number={},
  pages={506-515},
  doi={10.1109/TIFS.2025.3648867}}

@article{libsvm,
author = {Chang, C. C. and Lin, C. J.},
title = {LIBSVM: A library for support vector machines},
year = {2011},
issue_date = {April 2011},
publisher = {Association for Computing Machinery},
address = {New York, NY, USA},
volume = {2},
number = {3},
issn = {2157-6904},
url = {https://doi.org/10.1145/1961189.1961199},
doi = {10.1145/1961189.1961199},
journal = {ACM Trans. Intell. Syst. Technol.},
month = may,
articleno = {27},
numpages = {27}
}

@article{ghadimi_lan,
author = {Ghadimi, S. and Lan, G.},
title = {Stochastic First- and Zeroth-Order Methods for Nonconvex Stochastic Programming},
journal = {SIAM Journal on Optimization},
volume = {23},
number = {4},
pages = {2341-2368},
year = {2013},
doi = {10.1137/120880811},

URL = { 
    
        https://doi.org/10.1137/120880811
    
    

},

}

@Article{Bertsekas1973,
    author={Bertsekas, D. P.},
    title={Stochastic optimization problems with nondifferentiable cost functionals},
    journal={Journal of Optimization Theory and Applications},
    year={1973},
    month={Aug},
    day={01},
    volume={12},
    number={2},
    pages={218-231},
    issn={1573-2878},
    doi={10.1007/BF00934819},
    url={https://doi.org/10.1007/BF00934819}
    }

@InProceedings{ji_improved_vr,  
title =  {Improved Zeroth-Order Variance Reduced Algorithms and Analysis for Nonconvex Optimization},  author =       {Ji, K. and Wang, Z. and Zhou, Y. and Liang, Y.},  
booktitle =  {Proceedings of the 36th International Conference on Machine Learning},  
pages =  {3100--3109},  
year =  {2019},  
editor =  {Chaudhuri, Kamalika and Salakhutdinov, Ruslan},  
volume =  {97},  
series =  {Proceedings of Machine Learning Research},  
month =  {09--15 Jun},    
publisher =    {PMLR},  
address   = {Brooklyn, NY},  
url =  {\url{https://proceedings.mlr.press/v97/ji19a.html}}
}

@inproceedings{rando2023optimal,
 author = {Rando, M. and Molinari, C. and Rosasco, L. and Villa, S.},
 booktitle = {Advances in Neural Information Processing Systems},
 editor = {A. Oh and T. Naumann and A. Globerson and K. Saenko and M. Hardt and S. Levine},
 pages = {36738--36767},
 title = {An Optimal Structured Zeroth-order Algorithm for Non-smooth Optimization},
 url = {\url{https://proceedings.neurips.cc/paper_files/paper/2023/hash/7429f4c1b267cf619f28c4d4f1532f99-Abstract-Conference.html}},
 volume = {36},
 year = {2023}
}

@article{polyak1987introduction,
  author={Polyak, B. T.},
  title     = {Introduction to Optimization},
  series    = {Translations in Mathematics and Engineering},
  publisher = {Optimization Software, Publications Division},
  address   = {New York, NY, USA},
  year      = {1987},
  edition   = {1st English ed.},
  pages     = {438},
  isbn      = {0-911575-14-6},
  note      = {Originally published in Russian (1983). \url{https://www.researchgate.net/publication/342978480_Introduction_to_Optimization}}
}

@Article{Rando2024,
author={Rando, M. and Molinari, C. and Villa, S. and Rosasco, L.},
title={Stochastic zeroth order descent with structured directions},
journal={Computational Optimization and Applications},
year={2024},
month={Oct},
day={18},
issn={1573-2894},
doi={10.1007/s10589-024-00616-1},
url={https://doi.org/10.1007/s10589-024-00616-1}
}

@inproceedings{zo_svrg_avg,
 author = {Liu, S. and Kailkhura, B. and Chen, P. and Ting, P. and Chang, S. and Amini, L.},
 booktitle = {Advances in Neural Information Processing Systems},
 editor = {S. Bengio and H. Wallach and H. Larochelle and K. Grauman and N. Cesa-Bianchi and R. Garnett},
 pages = {},
 publisher = {Curran Associates, Inc.},
 address   = {Red Hook, NY, USA},
 title = {Zeroth-Order Stochastic Variance Reduction for Nonconvex Optimization},
 url = {\url{https://proceedings.neurips.cc/paper_files/paper/2018/hash/ba9a56ce0a9bfa26e8ed9e10b2cc8f46-Abstract.html}},
 volume = {31},
 year = {2018}
}

@misc{liu2018stochastic,
  title={Stochastic zeroth-order optimization via variance reduction method},
  author={Liu, L. and Cheng, M. and Hsieh, C. J. and Tao, D.},
  year={2018},
  eprint = {1805.11811},
  archivePrefix={arXiv},
  url = {https://arxiv.org/abs/1805.11811},
}

@article{kozak2021zeroth,
author={Kozak, D. and Molinari, C. and Rosasco, L. and Tenorio, L. and Villa, S.},
title={Zeroth-order optimization with orthogonal random directions},
journal={Mathematical Programming},
year={2023},
month={May},
day={01},
volume={199},
number={1},
pages={1179-1219},
issn={1436-4646},
doi={10.1007/s10107-022-01866-9},
publisher={Springer},
address = {233 Spring Street, New York, NY 10013, USA},
}

@article{berahas2022theoretical,
  title={A theoretical and empirical comparison of gradient approximations in derivative-free optimization},
  author={Berahas, A. S. and Cao, L. and Choromanski, K. and Scheinberg, K.},
  journal={Foundations of Computational Mathematics},
  volume={22},
  number={2},
  pages={507--560},
  year={2022},
  doi = {10.1007/s10208-021-09513-z},
  publisher={Springer},
  address = {233 Spring Street, New York, NY 10013, USA},
}

@Article{stief_zeroth,
author={Wang, T. and Feng, Y.},
title={Convergence Rates of Zeroth Order Gradient Descent for {\L}ojasiewicz Functions},
journal={INFORMS Journal on Computing},
year={2024},
month={Mar},
day={13},
publisher={INFORMS},
issn={1091-9856},
doi={10.1287/ijoc.2023.0247},
url={https://doi.org/10.1287/ijoc.2023.0247}
}

@misc{mu2024variancereducedgradientestimatornonconvex,
      title={Variance-Reduced Gradient Estimator for Nonconvex Zeroth-Order Distributed Optimization}, 
      author={Mu, H. and Tang, Y. and Li, Z.},
      year={2024},
      eprint={2409.19567},
      archivePrefix={arXiv},
      primaryClass={math.OC},
      url={https://arxiv.org/abs/2409.19567}, 
}

@InProceedings{pmlr-v119-huang20j,  title =  {Accelerated Stochastic Gradient-free and Projection-free Methods},  author =       {Huang, F. and Tao, L. and Chen, S.},  booktitle =  {Proceedings of the 37th International Conference on Machine Learning},  pages =  {4519--4530},  year =  {2020},  editor =  {III, Hal Daumé and Singh, Aarti},  volume =  {119},  series =  {Proceedings of Machine Learning Research},  month =  {13--18 Jul},  url =  {\url{https://proceedings.mlr.press/v119/huang20j.html}}}

@inproceedings{svrg_johnson,
 author = {Johnson, R. and Zhang, T.},
 booktitle = {Advances in Neural Information Processing Systems},
 editor = {C.J. Burges and L. Bottou and M. Welling and Z. Ghahramani and K.Q. Weinberger},
 pages = {},
 title = {Accelerating Stochastic Gradient Descent using Predictive Variance Reduction},
url = {\url{https://proceedings.neurips.cc/paper_files/paper/2013/hash/ac1dd209cbcc5e5d1c6e28598e8cbbe8-Abstract.html}},
 volume = {26},
 year = {2013}
}

@InProceedings{Moosavi-Dezfooli_2017_CVPR,
author = {Moosavi-Dezfooli, S. and Fawzi, A. and Fawzi, O. and Frossard, P.},
title = {Universal Adversarial Perturbations},
booktitle = {Proceedings of the IEEE Conference on Computer Vision and Pattern Recognition (CVPR)},
month = {July},
year = {2017},
doi={10.1109/CVPR.2017.17}
}

@inproceedings{mezo,
 author = {Malladi, S. and Gao, T. and Nichani, E. and Damian, A. and Lee, J. D. and Chen, D. and Arora, S.},
 booktitle = {Advances in Neural Information Processing Systems},
 editor = {A. Oh and T. Naumann and A. Globerson and K. Saenko and M. Hardt and S. Levine},
 pages = {53038--53075},
 publisher = {Curran Associates, Inc.},
 address   = {Red Hook, NY 12571, USA},
 title = {Fine-Tuning Language Models with Just Forward Passes},
 url = {\url{https://proceedings.neurips.cc/paper_files/paper/2023/hash/a627810151be4d13f907ac898ff7e948-Abstract-Conference.html}},
 volume = {36},
 year = {2023}
}

@InProceedings{zo_llm_benchmark,  
title =  {Revisiting Zeroth-Order Optimization for Memory-Efficient {LLM} Fine-Tuning: A Benchmark},  author =       {Zhang, Y. and Li, P. and Hong, J. and Li, J. and Zhang, Y. and Zheng, W. and Chen, P. and Lee, J. D. and Yin, W. and Hong, M. and Wang, Z. and Liu, S. and Chen, T.},  
booktitle =  {Proceedings of the 41st International Conference on Machine Learning},  
pages =  {59173--59190},  
year =  {2024},  
editor =  {Salakhutdinov, R. and Kolter, Z. and Heller, K. and Weller, A. and Oliver, N. and Scarlett, J. and Berkenkamp, F.},  
volume =  {235},  
series =  {Proceedings of Machine Learning Research},  
month =  {21--27 Jul},  
publisher =    {PMLR},  
address = {Cambridge, MA},  
url =  {\url{https://proceedings.mlr.press/v235/zhang24ad.html}}

}

@article{raffel_llms,
author = {Raffel, C. and Shazeer, N. and Roberts, A. and Lee, K. and Narang, S. and Matena, M. and Zhou, Y. and Li, W. and Liu, P. J.},
title   = {Exploring the Limits of Transfer Learning with a Unified Text-to-Text Transformer},
journal = {Journal of Machine Learning Research},
year    = {2020},
volume  = {21},
note     = {\url{http://jmlr.org/papers/v21/20-074.html}},
number  = {140},
pages   = {1--67},
}

@misc{rando2025structuredtouroptimizationfinite,
      title={A Structured Tour of Optimization with Finite Differences}, 
      author={Rando, M. and Molinari, C. and Rosasco, L. and Villa, S.},
      year={2025},
      eprint={2505.19720},
      archivePrefix={arXiv},
      primaryClass={math.OC},
      url={https://arxiv.org/abs/2505.19720}, 
}

@INPROCEEDINGS{bb_univ_pert,

  author={Wang, S. and Shi, Y. and Han, Y.},

  booktitle={2018 24th International Conference on Pattern Recognition (ICPR)}, 

  title={Universal Perturbation Generation for Black-box Attack Using Evolutionary Algorithms}, 

  year={2018},

  volume={},

  number={},

  pages={1277-1282},

  doi={10.1109/ICPR.2018.8546023}}

@article{shamir2017optimal,
  author={Shamir, O.},
  title   = {An Optimal Algorithm for Bandit and Zero-Order Convex Optimization with Two-Point Feedback},
  journal = {Journal of Machine Learning Research},
  year    = {2017},
  volume  = {18},
  number  = {52},
  pages   = {1-11},
  note     = {\url{http://jmlr.org/papers/v18/16-632.html}}
}

@inproceedings{flaxman2005online,
author = {Flaxman, A. and Kalai, A. Tauman and McMahan, B.},
title = {Online Convex Optimization in the Bandit Setting: Gradient Descent Without a Gradient},
booktitle = {SODA '05 Proceedings of the sixteenth annual ACM-SIAM symposium on Discrete algorithms},
year = {2005},
month = {January},
url = {\url{https://www.microsoft.com/en-us/research/publication/online-convex-optimization-bandit-setting-gradient-descent-without-gradient/}},
pages = {385-394},
isbn = {0-89871-585-7},
edition = {SODA '05 Proceedings of the sixteenth annual ACM-SIAM symposium on Discrete algorithms},
}

@Article{matplotlib,
  Author    = {Hunter, J. D.},
  Title     = {Matplotlib: A 2D graphics environment},
  Journal   = {Computing in Science \& Engineering},
  Volume    = {9},
  Number    = {3},
  Pages     = {90--95},
  publisher = {IEEE COMPUTER SOC},
  doi       = {10.1109/MCSE.2007.55},
  year      = 2007
}

@incollection{pytorch,
		title = {Py{{T}}orch: An Imperative Style, High-Performance Deep Learning Library},
		author = {Paszke, A. and Gross, S. and Massa, F. and Lerer, A. and Bradbury, J. and Chanan, G. and Killeen, T. and Lin, Z. and Gimelshein, N. and Antiga, L. and Desmaison, A. and Kopf, A. and Yang, E. and DeVito, Z. and Raison, M. and Tejani, A. and Chilamkurthy, S. and Steiner, B. and Fang, L. and Bai, J. and Chintala, S.},
		booktitle = {Advances in Neural Information Processing Systems 32},
		pages = {8024--8035},
		year = {2019},
		publisher = {Curran Associates, Inc.},
		address = {Red Hook, NY, USA},
        url = {\url{https://proceedings.neurips.cc/paper_files/paper/2019/hash/bdbca288fee7f92f2bfa9f7012727740-Abstract.html}}
	}

@Article{         numpy,
 title         = {Array programming with {NumPy}},
 author        = {Harris, C. R. and 
 Millman, K. J. and 
 van der Walt, S. J. and 
 Gommers, R. and 
 Virtanen, P. and 
 Cournapeau, D. and 
 Wieser, E. and 
 Taylor, J. and 
 Berg, S. and 
 Smith, N. J. and 
 Kern, R. and 
 Picus, M. and 
 Hoyer, S. and 
 van Kerkwijk, M. H. and 
 Brett, M. and 
 Haldane, A. and 
Fern{\'{a}}ndez del R{\'{i}}o, J. and 
  Wiebe, M. and 
  Peterson, P. and 
  G{\'{e}}rard-Marchant, P. and
  Sheppard, K. and 
  Reddy, T. and
  Weckesser, W. and 
  Abbasi, H. and 
  Gohlke, C. and 
  Oliphant, T. E.},
 year          = {2020},
 month         = sep,
 journal       = {Nature},
 volume        = {585},
 number        = {7825},
 pages         = {357--362},
 doi           = {10.1038/s41586-020-2649-2},
 publisher     = {Springer Science and Business Media {LLC}},
 url           = {https://doi.org/10.1038/s41586-020-2649-2}
}

@inproceedings{prox_svrg_nconv,
 author = {J. Reddi, S. and Sra, S. and Poczos, B. and Smola, A. J.},
 booktitle = {Advances in Neural Information Processing Systems},
 editor = {D. Lee and M. Sugiyama and U. Luxburg and I. Guyon and R. Garnett},
 pages = {},
 publisher = {Curran Associates, Inc.},
 address   = {Red Hook, NY, USA},
 title = {Proximal Stochastic Methods for Nonsmooth Nonconvex Finite-Sum Optimization},
 url = {\url{https://proceedings.neurips.cc/paper_files/paper/2016/hash/291597a100aadd814d197af4f4bab3a7-Abstract.html}},
volume = {29},
 year = {2016}
}

@article{mezzadri2006generate,
  author    = {Mezzadri, F.},
  title     = {How to generate random matrices from the classical compact groups},
  journal   = {Notices of the American Mathematical Society},
  volume    = {54},
  number    = {5},
  pages     = {592--604},
  month     = may,
  year      = {2007},
  publisher = {American Mathematical Society},
  address   = {Providence, RI},
  note = {\url{https://www.ams.org/notices/200705/fea-mezzadri-web.pdf}}
}

@article{Zhu_Zhang_2023, title={An Adaptive Variance Reduction Zeroth-Order Algorithm for Finite-Sum Optimization}, volume={3}, url={https://drpress.org/ojs/index.php/fcis/article/view/8568}, DOI={10.54097/fcis.v3i3.8568}, abstractNote={
The unconstrained finite-sum optimization problem is a common type of problem in the field of optimization, and there is currently limited research on zeroth-order optimization algorithms. To solve unconstrained finite-sum optimization problems for non-convex function, we propose a zeroth-order optimization algorithm with adaptive variance reduction, called ZO-AdaSPIDER for short. Then, we analyze the convergence performance of the algorithm. The theoretical results show that ZO-AdaSPIDER algorithm can converge to -stationary point when facing non-convex function, and its convergence rate is .
}, number={3}, journal={Frontiers in Computing and Intelligent Systems}, author={Zhu, Y. and Zhang, M.}, year={2023}, month={May}, pages={66–70} }

@InProceedings{var_red_llm,  title =  {Variance-reduced Zeroth-Order Methods for Fine-Tuning Language Models},  author =       {Gautam, T. and Park, Y. and Zhou, H. and Raman, P. and Ha, W.},  booktitle =  {Proceedings of the 41st International Conference on Machine Learning},  pages =  {15180--15208},  year =  {2024},  editor =  {Salakhutdinov, Ruslan and Kolter, Zico and Heller, Katherine and Weller, Adrian and Oliver, Nuria and Scarlett, Jonathan and Berkenkamp, Felix},  volume =  {235},  series =  {Proceedings of Machine Learning Research},  month =  {21--27 Jul},  url =  {\url{https://proceedings.mlr.press/v235/gautam24a.html}}}

@Article{Traoré2024,
author={Traor{\'e}, C. and Apidopoulos, V. and Salzo, S. and Villa, S.},
title={Variance Reduction Techniques for Stochastic Proximal Point Algorithms},
journal={Journal of Optimization Theory and Applications},
year={2024},
month={Nov},
day={01},
volume={203},
number={2},
pages={1910-1939},
issn={1573-2878},
doi={10.1007/s10957-024-02502-6},
url={https://doi.org/10.1007/s10957-024-02502-6}
}

@inproceedings{prox_svrg_li,
 author = {Li, Z. and Li, J.},
 booktitle = {Advances in Neural Information Processing Systems},
 editor = {S. Bengio and H. Wallach and H. Larochelle and K. Grauman and N. Cesa-Bianchi and R. Garnett},
 pages = {},
 publisher = {Curran Associates, Inc.},
 address   = {Red Hook, NY, USA},
 title = {A Simple Proximal Stochastic Gradient Method for Nonsmooth Nonconvex Optimization},
 url = {\url{https://proceedings.neurips.cc/paper_files/paper/2018/file/e727fa59ddefcefb5d39501167623132-Paper.pdf}},
 volume = {31},
 year = {2018}
}

@book{lin_alg_book,
author = {Lankham, I. and Nachtergaele, B. and Schilling, A.},
title = {Linear Algebra as an Introduction to Abstract Mathematics},
publisher = {World Scientific Publishing Co. Pte. Ltd.},
address   = {Singapore},
year      = {2015},
isbn      = {9789814723770},
URL = {\url{https://www.worldscientific.com/doi/abs/10.1142/9808}},
eprint = {https://www.worldscientific.com/doi/pdf/10.1142/9808}
}

@InProceedings{redd_svrg_nonconv,  
title =  {Stochastic Variance Reduction for Nonconvex Optimization},  
author =  {Reddi, S. J. and Hefny, A. and Sra, S. and Poczos, B. and Smola, A.},  
booktitle =  {Proceedings of The 33rd International Conference on Machine Learning},  
pages =  {314--323},  year =  {2016},  
editor =  {Balcan, Maria Florina and Weinberger, Kilian Q.},  
volume =  {48},  series =  {Proceedings of Machine Learning Research},  address =  {New York, New York, USA},  month =  {20--22 Jun},  publisher =    {PMLR},  url =  {\url{https://proceedings.mlr.press/v48/reddi16.html}}}

@book{chikuse2012statistics,
  title={Statistics on special manifolds},
  author={Chikuse, Y.},
  series    = {Lecture Notes in Statistics},
  volume    = {174},
  publisher = {Springer},
  address   = {New York, NY, USA},
  year      = {2003},
  isbn      = {978-0-387-00160-9},
  doi       = {10.1007/978-0-387-21540-2}
}

@Article{Kazemi2024,
author={Kazemi, E. and Wang, L.},
title={Efficient zeroth-order proximal stochastic method for nonconvex nonsmooth black-box problems},
journal={Machine Learning},
year={2024},
month={Jan},
day={01},
volume={113},
number={1},
pages={97-120},
issn={1573-0565},
doi={10.1007/s10994-023-06409-7},
url={https://doi.org/10.1007/s10994-023-06409-7}
}

@inproceedings{fang_spider,
 author = {Fang, C. and Li, C. J. and Lin, Z. and Zhang, T.},
 booktitle = {Advances in Neural Information Processing Systems},
 editor = {S. Bengio and H. Wallach and H. Larochelle and K. Grauman and N. Cesa-Bianchi and R. Garnett},
 pages = {},
 title = {SPIDER: Near-Optimal Non-Convex Optimization via Stochastic Path-Integrated Differential Estimator},
 url = {\url{https://proceedings.neurips.cc/paper_files/paper/2018/hash/1543843a4723ed2ab08e18053ae6dc5b-Abstract.html}},
volume = {31},
 year = {2018}
}

@InProceedings{str_zo_applied,  title =  {Structured Evolution with Compact Architectures for Scalable Policy Optimization},  
author =       {Choromanski, K. and Rowland, M. and Sindhwani, V. and Turner, R. and Weller, A.},  booktitle =  {Proceedings of the 35th International Conference on Machine Learning},  pages =  {970--978},  year =  {2018},  editor =  {Dy, Jennifer and Krause, Andreas},  volume =  {80},  series =  {Proceedings of Machine Learning Research},  month =  {10--15 Jul},  url =  {\url{https://proceedings.mlr.press/v80/choromanski18a.html}}}

@InProceedings{gasnikov_sph,  title =  {The power of first-order smooth optimization for black-box non-smooth problems},  author =       {Gasnikov, A. and Novitskii, A. and Novitskii, V. and Abdukhakimov, F. and Kamzolov, D. and Beznosikov, A. and Takac, M. and Dvurechensky, P. and Gu, B.},  booktitle =  {Proceedings of the 39th International Conference on Machine Learning},  pages =  {7241--7265},  year =  {2022},  editor =  {Chaudhuri, Kamalika and Jegelka, Stefanie and Song, Le and Szepesvari, Csaba and Niu, Gang and Sabato, Sivan},  volume =  {162},  series =  {Proceedings of Machine Learning Research},  month =  {17--23 Jul},  publisher =    {PMLR}, address={Baltimore, Maryland, USA},  url =  {\url{https://proceedings.mlr.press/v162/gasnikov22a.html}},  }

@misc{chen2015randomized,
  title={Randomized derivative-free optimization of noisy convex functions},
  author={Chen, R. and Wild, S.},
  journal={arXiv},
  year={2015},
  eprint = {1507.03332},
  url = {https://arxiv.org/abs/1507.03332}
}

@ARTICLE{duchi_power_of_two,
  author={Duchi, J. C. and 
  Jordan, M. I. and 
  Wainwright, M. J. and 
  Wibisono, A.},
  journal={IEEE Transactions on Information Theory}, 
  title={Optimal Rates for Zero-Order Convex Optimization: The Power of Two Function Evaluations}, 
  year={2015},
  volume={61},

  number={5},

  pages={2788-2806},

  doi={10.1109/TIT.2015.2409256}}

@article{bolte,
    author = {Bolte, J. and Daniilidis, A. and Ley, O. and Mazet, L.},
    year = {2009},
    month = {06},
    pages = {3319-3363},
    title = {Characterizations of Łojasiewicz inequalities: Subgradient flows, talweg, convexity},
    volume = {362},
    journal = {Transactions of The American Mathematical Society - TRANS AMER MATH SOC},
    doi = {10.1090/S0002-9947-09-05048-X}
}

@Article{Gao2018,
author={Gao, X.
and Jiang, B.
and Zhang, S.},
title={On the Information-Adaptive Variants of the ADMM: An Iteration Complexity Perspective},
journal={Journal of Scientific Computing},
year={2018},
month={Jul},
day={01},
volume={76},
number={1},
pages={327-363},
issn={1573-7691},
doi={10.1007/s10915-017-0621-6},
url={https://doi.org/10.1007/s10915-017-0621-6}
}

@article{szd_prox_weakly_conv,
author = {Pougkakiotis, S. and Kalogerias, D.},
title = {A Zeroth-Order Proximal Stochastic Gradient Method for Weakly Convex Stochastic Optimization},
journal = {SIAM Journal on Scientific Computing},
volume = {45},
number = {5},
pages = {A2679-A2702},
year = {2023},
doi = {10.1137/22M1494270},

URL = { 
    
        https://doi.org/10.1137/22M1494270
    
    

},
eprint = { 
    
        https://doi.org/10.1137/22M1494270
    
    

}
}

\newpage

\appendix
\section{Experimental Details}\label{app:exp_details}
In this appendix, we provide the details of the experiments presented in Section \ref{sec:experiments}. Every script has been implemented in Python (version 3.10.14) using PyTorch \cite{pytorch} v2.0.1, NumPy \cite{numpy} v1.26.4, Matplotlib \cite{matplotlib}  v3.9.2  libraries. %
The experiments have been performed on a machine with the following specifications: \texttt{CPU: 64 x AMD EPYC 7301 16-Core}, \texttt{GPU: 1 x NVIDIA Quadro RTX 6000, RAM: 256GB} \footnote{Further details on the GPU can be found at the following link: \url{https://resources.nvidia.com/en-us-general-rtx/quadro-rtx-6000-data-sheet?xs=105058}}. In the following, we provide details of experiments performed in Section \ref{sec:experiments}.

\paragraph*{Construction of Structured Directions}
Several methods for constructing structured direction matrices have been proposed in the literature - see e.g. \cite{rando2023optimal,kozak2021zeroth,rando2025structuredtouroptimizationfinite}. In this work,  to construct them for Algorithm \ref{algo:osvrz}, we adopted the QR factorization with sign correction technique proposed in previous works \cite{rando2023optimal,kozak2021zeroth,mezzadri2006generate}. For every outer iteration $\tau \in \mathbb{N}$, at each inner iteration $k \in [m]$, every structured direction matrix $G_{i,k}^\tau$ with $i =1,\dots,b$ is constructed as follows: a matrix $A_{i,k}^{\tau} \in \mathbb{R}^{d \times d}$  with i.i.d. entries sampled from a Gaussian distribution $\mathcal{N}(0, I)$ is constructed. Then, we compute its QR decomposition $A_{i,k}^{\tau} = Q_{i,k}^\tau R_{i,k}^\tau$. The signs of the columns of $Q_{i,k}^\tau$ are subsequently corrected according to the signs of the diagonal entries of $R_{i,k}^\tau$, ensuring the desired Haar-uniform distribution of $Q_{i,k}^\tau$ \cite{mezzadri2006generate}. The structured direction matrix $G_{i,k}^\tau$ is obtained by truncating $Q_{i,k}^\tau$, i.e.,let  $I_{d,\ell} \in \mathbb{R}^{d \times \ell}$ be the truncated identity matrix, then $G_{i,k}^\tau = Q_{i,k}^\tau I_{d,\ell}$. In the following we summarize the procedure to construct a direction matrix we used to build every $G_{i,k}^\tau$ of Algorithm \ref{algo:osvrz} in the experiments.
\begin{algorithm}[H]
\caption{Generate Structured directions}%
\begin{algorithmic}[1]
\State \textbf{Input:} $d \in \mathbb{N}_+$, $\ell$ s.t. $ 1 \leq \ell \leq d$.
\State Build $A \in \mathbb{R}^{d \times d}$ such that $A_{i,j} \sim \mathcal{N}(0,1)$ for every $i,j \in [d]$.
\State Compute $A = QR$
\State Construct $D = \operatorname{diag}(\operatorname{sign}(R_{1,1}),\ldots,\operatorname{sign}(R_{d,d}))$
\State {\bf Output:} $G = (Q D) I_{d,\ell}$
\end{algorithmic}
\end{algorithm}

\paragraph*{Parameter selection procedure.}
For every experiment, the parameters of each competing algorithm are selected via a grid search over a fixed set of candidate values - the parameter grids are reported in Appendices~\ref{app:details_synthetic}, \ref{app:details_adv}, and \ref{app:details_bbclass}. The procedure is as follows. For each method and each parameter configuration in its corresponding grid, the method is run using a function-evaluation budget equal to the budget specified for that problem (i.e., $10^7$ for the synthetic experiments and the universal adversarial perturbation problem, and $10^6$ for the black-box classification problem). Each configuration is evaluated over $5$ independent runs (repetitions), using different initial conditions sampled i.i.d. from the corresponding initial-condition space (i.e., the hypercube $[-5,5]^d$ for the synthetic problems, $[-0.5,-0.3]^d$ for the universal adversarial perturbation problem, and the zero vector for the black-box classification problem). For each method, we select as best configuration, the parameter configuration that achieves the smallest normalized function value at the end of the prescribed function-evaluation budget. The function evaluations used during the parameter-selection procedure are not included in the function-evaluation budgets used in the experiments.

\subsection{Synthetic Experiments}\label{app:details_synthetic}

We set $d=100$ and we considered the minimization of three synthetic target functions $F_1,F_2,F_3:\mathbb{R}^d\to\mathbb{R}$. The first, Least Squares + L1 (Lasso), is given by $F_1(x) = \tfrac12\|Ax-y\|_2^2 + \lambda\|x\|_1$, where the matrix $A\in\mathbb{R}^{d\times d}$ is constructed to enforce a prescribed condition number on the smooth part as follows: we sample $\bar A \in \mathbb{R}^{d \times d}$ with i.i.d.\ $\mathcal N(0,1)$ entries, compute its SVD $\bar A = U \Sigma V^\top$, and replace the singular values in $\Sigma$ with values linearly spaced between $\sqrt{\mu}=1$ and $\sqrt{L}=\sqrt{10^5}$, yielding $A = U\Sigma V^\top$; the smooth term $\tfrac12\|Ax-y\|_2^2$ is therefore $1$-strongly convex and $10^5$-smooth. The target vector is set as $y = Ax^\star$ for $x^\star = [1,\dots,1]$. The second function (Logistic + L1), is a sparse logistic regression objective $F_2(x) = \tfrac1n\sum_{i=1}^n \log(1+e^{-b_i a_i^\top x}) + \lambda\|x\|_1$, where the rows $a_i\in\mathbb{R}^d$ of $A$ are i.i.d.\ $\mathcal N(0,1)/\sqrt{d}$, a sparse ground-truth $x_{\text{true}}$ is generated with only $d/10$ nonzero entries drawn from $\mathcal N(0,100)$, the labels are initially set to $b_i = \mathrm{sign}(a_i^\top x_{\text{true}})$, and $5\%$ of them are randomly flipped to simulate label noise. The third function, (Phase Retrieval + L1), is a non-convex quartic objective $F_3(x) = \tfrac1n\sum_{i=1}^n\big((a_i^\top x)^2 - c_i\big)^2 + \lambda\|x\|_1$, where the vectors $a_i$ are generated in the same way as for the logistic case, a sparse $x_{\text{true}}$ is drawn analogously (nonzero entries from $\mathcal N(0,1)$ this time), and $c_i = (a_i^\top x_{\text{true}})^2$, optionally corrupted by additive Gaussian noise (set to zero in our experiments). For all three problems, the regularization parameter is fixed to $\lambda = 10^{-3}$. Synthetic target function definitions are summarized in Table \ref{tab:syn_target_fun}.
 
\begin{table}[ht]
\centering
\caption{Target functions considered in the synthetic experiments of Section \ref{sec:experiments}.}
\label{tab:syn_target_fun}
\begin{tabular}{l l l l l}
\toprule
\textbf{Name} & \textbf{Function} & $d$ & $n$ & $\lambda$\\
\midrule
Least Squares + L1 & $F_1(x) = \dfrac{1}{2}\|Ax-y\|_2^2 + \lambda\|x\|_1$ & $100$ & $100$ &$10^{-3}$\\
Logistic + L1 & $F_2(x) = \dfrac{1}{n}\displaystyle\sum_{i=1}^n \log\!\left(1+e^{-b_i a_i^\top x}\right) + \lambda\|x\|_1$ & $100$ & $1000$ &$10^{-3}$\\
Phase Retrieval + L1 & $F_3(x) = \dfrac{1}{n}\displaystyle\sum_{i=1}^n \left((a_i^\top x)^2 - c_i\right)^2 + \lambda\|x\|_1$ & $100$ & $1000$ &$10^{-3}$ \\
\bottomrule
\end{tabular}
\end{table}

\noindent For all experiments, the hyperparameters of each algorithm are selected via grid search. Specifically, we tune the stepsize, the number of sampled directions, the mini-batch size, and, for variance-reduced methods, the number of inner iterations. The search space is given by the Cartesian product of the following candidate values. The stepsize grid is $\{10^{-9}, 10^{-8}, 10^{-7}, 10^{-6}\}$ for Least-Squares + L1 and $\{10^{-3}, 10^{-2}, 10^{-1}, 1\}$ for Logistic + L1 and Phase Retrieval + L1. The number of directions and the batch-size are chosen from $\{1, 10, 25, 50, 100\}$. The number of inner iterations is selected from $\{50, 100, 200, 500,1000\}$. For RSPGF, the stepsize sequence is given by $\gamma_\tau = \gamma/\sqrt{\tau+1}$, where the initial stepsize parameter $\gamma$ is selected by grid search using the same candidate values described above. For all algorithms, the smoothing parameter is kept constant throughout the optimization, i.e., $\beta_\tau = 10^{-5}$ for every $\tau \in \mathbb{N}$. Each experiment is repeated $10$ times. In every trial, the initialization is sampled independently and uniformly from the hypercube $[-5,5]^d$. Candidate set of values are summarized in Table \ref{tab:syn_params}.

\begin{table}[ht]
    \centering
    \caption{Candidate parameter values considered in the numerical experiments. The \emph{Algorithm(s)} column specifies the methods for which each parameter is applicable. Different stepsize grids are used for $F_1$ and for $F_2$-$F_3$ defined in Table \ref{tab:syn_target_fun}.}
    \label{tab:syn_params}
    \begin{tabular}{lll}
        \toprule
        {\bf Parameter}  &  {\bf Algorithm(s)} & {\bf Candidate Values} \\
        \midrule
        Stepsize $\gamma$ & All algorithms & \begin{tabular}{l}
                $F_1:\ \{10^{-9},\,10^{-8},\,10^{-7},\,10^{-6}\}$ \\
                $F_2,F_3:\ \{10^{-3},\,10^{-2},\,10^{-1},\,1\}$
                \end{tabular}\\
        Inner iterations $m$ & Only variance-reduced methods & $\{50, 100, 200, 500,1000\}$\\
        Number of directions $\ell$ & VR-SZD & $\{1, 10, 25, 50, 100\}$\\
        Batch-size $b$ & All algorithms & $\{1, 10, 25, 50, 100\}$\\ 
     \bottomrule
    \end{tabular}
\end{table}

\subsection{Universal Adversarial Perturbation}\label{app:details_adv}
In this appendix, we report the experimental details of the universal adversarial attack experiment presented in Section~\ref{sec:experiments}.

\paragraph*{Problem Description.}

Let $\mathcal{X} \subset \mathbb{R}^d$ be a closed, non-empty convex set. Given a classifier $C:\mathcal{X}\to\{1,\dots,K\}$ and a target class $c\in\{1,\dots,K\}$, a universal adversarial perturbation \cite{moosavidezfooli2017universal} is a single perturbation vector $\delta\in\mathcal{X}$, shared across a whole collection of inputs of class $c$, that causes the classifier to misclassify most of them once added to the corresponding image. Notice that in real-world setting an attacker have no access to the classifier structure (i.e. he has no knowledge on the machine learning model used) and thus gradients cannot be computed making zeroth-order methods necessary to tackle such a problem.  Formally, let $\{x_i\}_{i=1}^{n}$ be a set of points labeled as $c$. The goal is to find $\delta \in \mathcal{X}$ which maximizes $\frac{1}{n} \sum\limits_{i=1}^n 1_{C(x_i + \delta) \neq c}$ while keeping $\|\delta\|\leq \epsilon$ for some $\epsilon>0$. However, since the indicator of misclassification is not amenable to optimization, this is often relaxed by searching for the perturbation $\delta$ that minimizes the confidence of classifying correctly the input. In our experiments, we model the confidence criterion with the black-box untargeted loss, as proposed in \cite{zoo,carlini_wagner}, and we replace the constraint with a penalty term $\lambda \| \delta\|^2$ where $\lambda > 0$ is a regularization parameter. Let $[ f( x )]_j$ denote the probability that the classifier assigns to $x$ the label $j$. Therefore, the optimization problem we tackle is the following
\begin{equation*}
    \argmin\limits_{\delta \in \mathcal{X}} \frac{1}{n} \sum\limits_{i=1}^n \max \left( \log([ f( \Pi_\mathcal{X} [x_i + \delta])]_c) - \max\limits_{j \neq c} \log([ f( \Pi_\mathcal{X} [x_i + \delta] )]_j), -\kappa\right) + \lambda \| \delta\|^2
\end{equation*}
where $\Pi_\mathcal{X}[x]$ denotes the projection of $x$ onto the space $\mathcal{X}$. Intuitively, for each perturbed image the loss compares the log-confidence of the best wrong class, $\max_{j\neq c}\log([f(\Pi_\mathcal{X} [x_i + \delta])]_j)$, against the log-confidence of the true class $c$, $\log([f(\Pi_\mathcal{X} [x_i + \delta])]_c)$. The projection $\Pi_\mathcal{X}$ is included to ensure that the perturbed input remains admissible. Minimizing this gap pushes the classifier toward preferring some other label over the correct label $c$. The floor at $-\kappa$ caps the loss once misclassification is already achieved with margin $\kappa$, so that the optimizer does not keep spending evaluations on images that are already successfully fooled, and instead focuses the budget on the harder ones. In our experiments, we fixed $\lambda = 0.5$ and $\kappa = 1.0$, and we considered a classifier trained on the MNIST dataset \cite{lecun2010mnist}, where every image pixel has been normalized to $[-0.5, 0.5]$ as suggested in~\cite{prox_svrg_li,rando2025structuredtouroptimizationfinite,rando2026zobaefficientsingleloopzerothorder}, and thus $\mathcal{X} = [-0.5, 0.5]^{d} \subset \mathbb{R}^{d}$ with $d = 784$ (since images are $28 \times 28$ matrices, we flatten them). The experiments were performed to find universal perturbations for subsets of images labeled $c=1,4$ and $9$.

\paragraph*{Data Preprocessing.} We downloaded the training and test sets of MNIST \cite{lecun2010mnist} dataset \footnote{We used the version of MNIST available through the torchvision library.} and we normalized the images to be in $[-0.5, 0.5]^{d}$ as suggested in previous works \cite{prox_svrg_li,rando2025structuredtouroptimizationfinite,rando2026zobaefficientsingleloopzerothorder}.

\paragraph*{The multiclass classifier.} The multiclass classifier is implemented as a convolutional neural network (CNN) following the architecture proposed in~\cite{rando2025structuredtouroptimizationfinite,rando2026zobaefficientsingleloopzerothorder}. The network comprises four convolutional layers with ReLU activations, interleaved with two max-pooling layers, followed by two fully connected hidden layers and an output layer. The first two convolutional layers use $3 \times 3$ kernels. The first layer maps the single-channel input image to $32$ feature maps, while the second maps the $32$ input channels to $64$ output channels. These layers are followed by a $2 \times 2$ max-pooling layer with stride $2$. The third and fourth convolutional layers both use $3 \times 3$ kernels and have $64$ input and $64$ output channels. A second $2 \times 2$ max-pooling layer with stride $2$ is then applied. The resulting feature maps are passed to two fully connected hidden layers, each containing $200$ units with ReLU activation. The final layer consists of $10$ output units with a softmax activation, corresponding to the classes of the MNIST dataset. A summary of the network architecture is provided in Table~\ref{tab:net_model}.

\begin{table}[ht]
\centering
\caption{Architecture of the multiclass classifier.}\label{tab:net_model}
\begin{tabular}{llll}
\toprule
Layer & Type & Parameters & Activation \\
\midrule
1 & Convolutional & in=$1$, out=$32$, kernel=$3\times3$ & ReLU \\
2 & Convolutional & in=$32$, out=$64$, kernel=$3\times 3$ & ReLU  \\
3 & Max Pooling & kernel=$2 \times 2$, stride=$2$ & --  \\
4 & Convolutional & in=$64$, out=$64$, kernel=$3 \times 3$ & ReLU  \\
5 & Convolutional & in=$64$, out=$64$, kernel=$3 \times 3$ & ReLU  \\
6 & Max Pooling & kernel=$2 \times 2$, stride=$2$ & --  \\
7 & Dense & in=$1024$, out=$200$ & ReLU  \\
8 & Dense & in=$200$, out=$200$ & ReLU  \\
9 & Dense & in=$200$, out=$10$ & Softmax  \\
\bottomrule
\end{tabular}
\end{table}

\paragraph*{Training Details.} In order to make the problem more challenging, we trained the multiclass classifier using the defensive distillation technique \cite{papernot2016distillation} to improve robustness against adversarial perturbations. Defensive distillation is a training procedure based on a teacher-student framework. Let $\mathcal{D} = \{(\hat{x}_i, y_i)\}_{i=1}^n$ denote the labeled training dataset, where $\hat{x}_i \in \mathbb{R}^{d_{\text{img}}}$ and $y_i \in \{1, \dots, C\}$ are the data points and labels respectively, with $C > 1$ denoting the number of classes. A teacher network $\hat{f}_{\text{teacher}}: \mathbb{R}^{d_{\text{img}}} \rightarrow \mathbb{R}^C$ is trained using a temperature-scaled softmax with temperature $T > 1$. More precisely, given the logits $x^{\text{log}} \in \mathbb{R}^C$ produced by the network for an input $\hat{x}$, the temperature-scaled softmax is a vector where every entry is
\begin{equation*}
[\text{softmax}_T(x^{\text{log}})]_i = \frac{\exp(x^{\text{log}}_i / T)}{\sum_{j=1}^C \exp(x^{\text{log}}_j / T)}, \quad i = 1, \dots, C.    
\end{equation*}
The teacher model is trained by minimizing the standard cross-entropy loss between the softened predictions and the ground-truth labels. After this preliminary training phase, the trained teacher network is used to generate soft labels $\bar{y}_i = \text{softmax}_T(\hat{f}_{\text{teacher}}(\hat{x}_i))$ for each input $\{\hat{x}_i\}_{i=1}^n$ and a student network $\hat{f}_{\text{student}}$, with the same architecture as the teacher, is then trained to mimic such predictions by minimizing the cross-entropy loss with respect to these instead of the true labels $y_i$. After training, the temperature is reset to $T = 1$ and the student model is used for inference. In our experiments, the temperature was set to $T = 100$ during both the teacher and student training, as suggested in~\cite{papernot2016distillation}. Both the student and teacher networks are initialized from the weights obtained after one SGD step with temperature $T=1$, starting from random initialization for weights and biases set to zero. The training dataset is split into two disjoint subsets, with $80\%$ of the data used for training and $20\%$ for validation and the models have been trained on training part using the SGD optimizer with learning rate $0.01$, momentum $0.9$, and batch size $32$. Early stopping based on the validation loss was employed, with a patience of $5$ epochs, as suggested in \cite{rando2026zobaefficientsingleloopzerothorder,rando2025structuredtouroptimizationfinite}. Training is performed for up to $100$ epochs. Additionally, dropout with rate $0.8$ was applied to the last hidden layer before the output layer. A summary of the hyperparameters is provided in Table~\ref{tab:training_params}. After training, we assessed the performance of the student model by evaluating its training and test accuracy, achieving a training accuracy of $99.12\%$ and a test accuracy of $98.66\%$.%
\begin{table}[ht]
    \centering
    \caption{Training parameters for classifier}\label{tab:training_params}
    \begin{tabular}{ll}
    \toprule
    {\bf Parameter} &   \\
    \midrule
    Optimizer & SGD with momentum \\
    Learning rate & $0.01$ \\
    Momentum & $0.9$ \\
    Batch size & $32$ \\
    Temperature & $100$ \\
    Dropout & $0.8$  \\
    Max epochs & $100$ \\
    Early stopping & Patience of $5$ epochs (on validation loss) \\
    \bottomrule
    \end{tabular}

\end{table}

\paragraph*{Parameter tuning.} The hyperparameters of each algorithm for each adversarial experiment have been selected via grid search. The stepsize grid is $\{10^{-6}, 10^{-5},10^{-4},10^{-3}\}$. The number of directions and the batch-size are chosen from $\{1, 10,  50, 100\}$. The number of inner iterations is selected from $\{1, 10, 50, 100\}$. As for the synthetic experiments, the stepsize sequence of RSPGF is set to $\gamma_\tau = \gamma/\sqrt{\tau+1}$ where the parameter $\gamma$ is selected by grid search using the same candidate values described above. For all algorithms, the smoothing parameter is kept constant $\beta_\tau = 10^{-5}$ for every $\tau \in \mathbb{N}$. Each experiment is repeated $10$ times. In every trial, the initialization is sampled independently and uniformly from the hypercube $[-0.5,-0.3]^d$, i.e., we initialize the algorithms with a small perturbation since, after rescaling, a value of $-0.5$ corresponds to a pixel intensity of $0$. The search space is given by the Cartesian product of the candidate value sets which are summarized in Table \ref{tab:adv_params}.

\begin{table}[ht]
    \centering
    \caption{Candidate parameter values considered in the universal adversarial perturbation experiments. The Algorithm(s) column specifies the methods for which each parameter is applicable.}
    \label{tab:adv_params}
    \begin{tabular}{lll}
        \toprule
        {\bf Parameter}  &  {\bf Algorithm(s)} & {\bf Candidate Values} \\
        \midrule
        Stepsize $\gamma$ & All algorithms & $\{10^{-6}, 10^{-5},10^{-4},10^{-3}\}$\\
        Inner iterations $m$ & Only variance-reduced methods & $\{10, 50, 100\}$\\
        Number of directions $\ell$ & VR-SZD & $\{1, 10, 50, 100\}$\\
        Batch-size $b$ & All algorithms & $\{1, 10, 50, 100\}$\\ 
     \bottomrule
    \end{tabular}
\end{table}

\subsection{Black-box Classification}\label{app:details_bbclass}
In experiment, we compared our method with other state-of-the-art algorithm to solve a black-box classification problems. Given a dataset $S = \{(x_i, y_i)\}_{i = 1}^n$ with $x_i \in \mathbb{R}^d$ and $y_i \in \{0,1\}$, we consider the following minimization problem
\begin{equation*}
    w^* \in \argmin\limits_{w} F(w) := \frac{1}{n} \sum\limits_{i = 1}^n L(w^\intercal x_i, y_i) + \lambda \| w \|_1,
\end{equation*}
where $\lambda$ is a regularization parameter and it is fixed to $10^{-5}$ and $L(x, y)$ is the binary cross-entropy loss defined as
\begin{equation*}
L(x, y) := -(y \log x + (1 - y) \log(1 - x)).    
\end{equation*}
The dimension of the input space $d$ and the number of functions $n$ are reported in Table \ref{tab:dataset_details}. For this experiment, we run the algorithms with $\beta_\tau = 10^{-5}$ for every $\tau \in \mathbb{N}$. %
\begin{table}[ht]
\centering
\caption{Details on dataset used for the classification experiments}\label{tab:dataset_details}
\begin{tabular}{lll}
\toprule
Dataset & $n$ & $d$ \\
\midrule
Mushrooms & 8124 & 112 \\ 
Phishing & 11055 & 68 \\
Australian & 690 & 14 \\
Splice & 1000 & 60 \\
\bottomrule
\end{tabular}
\end{table}

\noindent Datasets can be downloaded from LIBSVM \cite{libsvm} website \footnote{ \url{https://www.csie.ntu.edu.tw/~cjlin/libsvmtools/datasets/binary.html}. Note that data are partially pre-processed}. After downloading the data, we standardize them to zero mean and unit standard deviation and labels are set to be $0$ and $1$. The hyperparameters of each algorithm for each experiment were selected via grid search. The stepsize set is $\{0.001, 0.01, 0.1, 1.0\}$. Since the datasets have different dimensions $d$, the number of directions and the batch size are chosen from $\{\tfrac{1}{10}d, \tfrac{1}{5}d, \tfrac{1}{2}d, d\}$. The number of inner iterations is selected from $\{10, 50, 100\}$. As for the synthetic and universal adversarial experiments, the stepsize sequence of RSPGF is set to $\gamma_\tau = \gamma / \sqrt{\tau + 1}$, where the parameter $\gamma$ is selected via grid search using the same candidate values described above. The initialization is fixed to the zero vector. The search space is given by the Cartesian product of the candidate value sets summarized in Table \ref{tab:bb_params}.

\begin{table}[ht]
    \centering
    \caption{Candidate parameter values considered in the black-box classification experiments. The Algorithm(s) column specifies the methods for which each parameter is applicable.}
    \label{tab:bb_params}
    \begin{tabular}{lll}
        \toprule %
        {\bf Parameter}  &  {\bf Algorithm(s)} & {\bf Candidate Values} \\
        \midrule
        Stepsize $\gamma$ & All algorithms & $\{0.001, 0.01, 0.1, 1.0\}$\\
        Inner iterations $m$ & Only variance-reduced methods & $\{10, 50, 100\}$\\
        Number of directions $\ell$ & VR-SZD & $\{\frac{1}{10}d, \frac{1}{5} d, \frac{1}{2}d, d\}$\\
        Batch-size $b$ & All algorithms & $\{\frac{1}{10}d, \frac{1}{5} d, \frac{1}{2}d, d\}$\\ 
     \bottomrule
    \end{tabular}
\end{table}

\subsection{ Parameters ablation study}\label{app:ablation}
In this appendix, we investigate the impact of the algorithmic parameters of Algorithm~\ref{algo:osvrz} on its empirical performance. To this end, we consider the following $\ell_1$-regularized least-squares (LASSO) problem.
\begin{equation}\label{eqn:app_lasso_eqn}
F(x) = \frac{1}{2}\|Ax - y\|_2^2 + \lambda \|x\|_1,
\end{equation}
where $A \in \mathbb{R}^{d \times d}$ is constructed with the same procedure described in Appendix \ref{app:details_synthetic} such that the smooth part of the objective is $1$-strongly convex with a $100$-Lipschitz continuous gradient, and $y = Ax^\ast$ with $x^* = [1,\dots,1]$. Every experiment in this appendix is repeated $10$ times and mean and standard deviation of the results are reported. Moreover, the budget is fixed to $10^6$ function evaluations and the smoothing sequence is fixed constant as $\beta_\tau = 10^{-5}$.

\paragraph*{Number of directions.} We first study the effect of the number of directions used to approximate the stochastic gradients in VR-SZD. We fix the mini-batch size $b=1$ and run Algorithm~\ref{algo:osvrz} for minimizing eq.~\eqref{eqn:app_lasso_eqn} for dimensions $d \in \{10, 50, 100\}$ and different parameter configurations. Figure~\ref{fig:abl_change_l} illustrates the average normalized generalized gradient mapping norm at the final iteration, $\|\mathcal{G}_\gamma(x_K^\tau)\|/\|\mathcal{G}_\gamma(x_0^0)\|$, as a function of the stepsize $\gamma$, for the variance-reduced structured zeroth-order method, across varying problem dimension $d$ (rows) and number of inner iterations $m$ (columns), for several choices of the number of directions $\ell$. For stepsize values that lead to divergence, the corresponding metric is clipped to one.
\begin{figure}[h]
    \centering
    \includegraphics[width=\linewidth]{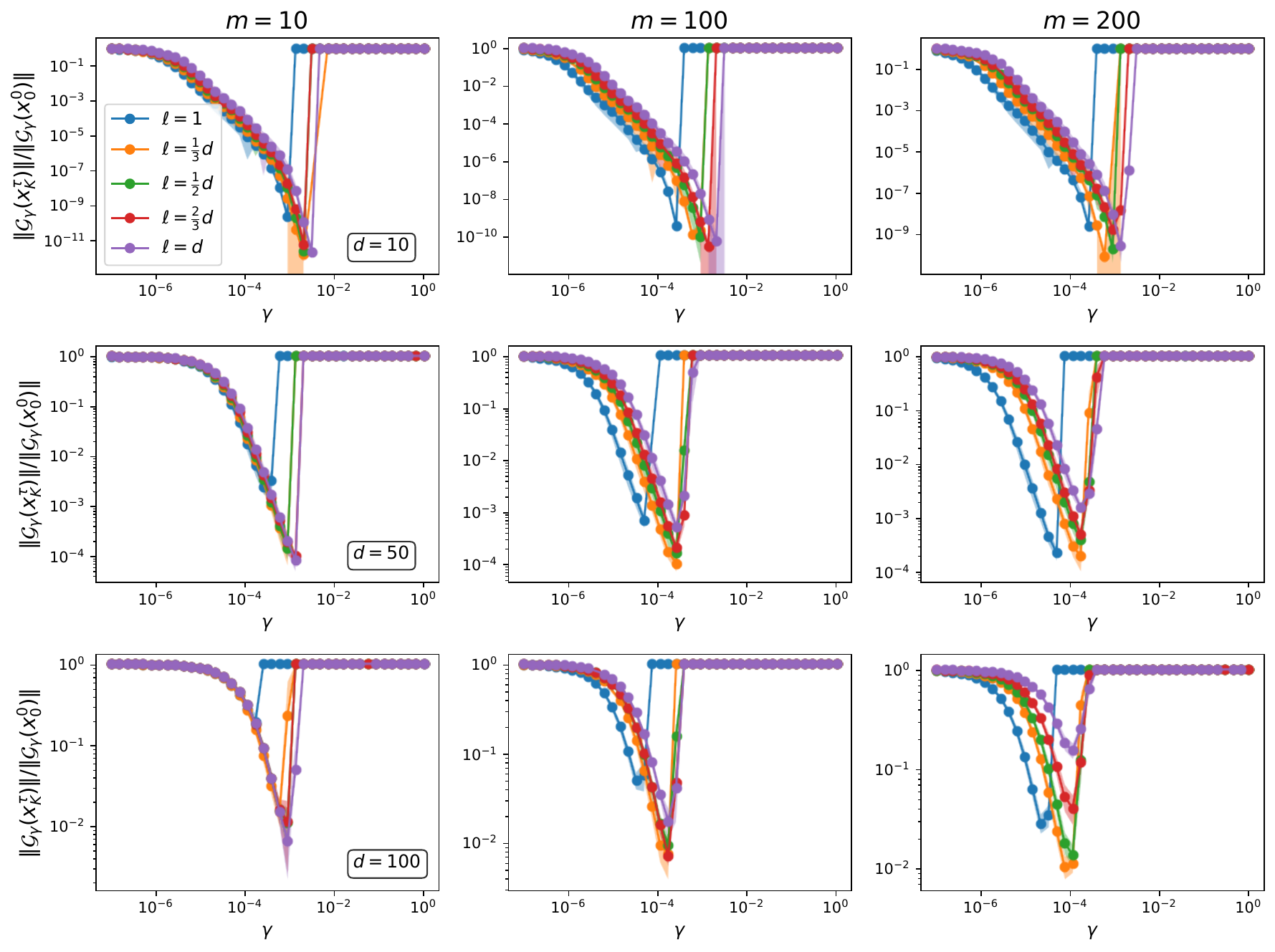}
    \caption{Normalized average generalized gradient mapping norm at the last iteration of Algorithm~\ref{algo:osvrz} as a function of the stepsize for different choices of number of directions $\ell$. For parameter configurations that lead to divergence, the corresponding values are clipped to one. Rows correspond to different dimensions $d \in \{10, 50, 100\}$, while columns correspond to different numbers of inner iterations $m \in \{10, 100, 200\}$.}
    \label{fig:abl_change_l}
\end{figure}
Across all settings, we observe that increasing $\ell$ enlarges the range of admissible stepsizes for which the method converges, and correspondingly improves the best attainable performance, consistently shifting the minimum of each curve toward larger values of $\gamma$ and to lower values of the normalized generalized gradient mapping norm. Notably, achieving good performance does not require taking $\ell = d$. Intermediate choices, such as $\ell = \tfrac{1}{2}d$ or even $\ell = \tfrac{1}{3}d$, already yield performance comparable to, and in several cases indistinguishable from, the full-direction case. This is because, for large values of $\ell$, the per-iteration cost increases and, given the finite budget of function evaluations, reduces the total number of iterations that can be performed. Comparing across columns, we observe that the enlargement of the admissible stepsize range due to larger $m$ becomes more pronounced as $\ell$ increases, and that the overall performance is generally superior for $\ell > 1$ relative to the single-direction case. This effect can be explained by the structure of Algorithm~\ref{algo:osvrz}: as $m$ grows, the inner iterations rely for a longer stretch on the full gradient approximation computed at the (increasingly stale) reference point, so that the accuracy of the directional gradient estimate becomes more critical to the stability of the iterates. Using a larger number of directions $\ell$ reduces the variance of this estimate, thereby partially compensating for the staleness of the reference information and allowing the method to remain stable, and hence to tolerate larger stepsizes. However, since the budget of function evaluations is finite, this improvement is not monotone: when both $\ell$ and $m$ are simultaneously large, the performance may deteriorate, an effect that is particularly visible in the high-dimensional regime (e.g., $d = 100$), where the higher per-iteration cost associated with large $\ell$ and $m$ reduces the number of outer iterations that can be performed within the fixed budget, thereby limiting the overall progress of the method.

\paragraph*{Batch size and impact of structured directions.}
We now study the effect of the batch size $b$ on performance. We set the number of directions $\ell=1$ and repeat the previous experiments, varying $b$ instead of $\ell$. Notice that with $\ell=1$, VR-SZD does not use structured directions but rather $b$ directions sampled i.i.d.\ uniformly over the sphere, so that the resulting stochastic gradient surrogates are unstructured. Figure~\ref{fig:abl_change_b} reports the normalized generalized gradient mapping norm as a function of $\gamma$, across problem dimension $d$ (rows) and number of inner iterations $m$ (columns), for this setting.
\begin{figure}[h]
    \centering
    \includegraphics[width=\linewidth]{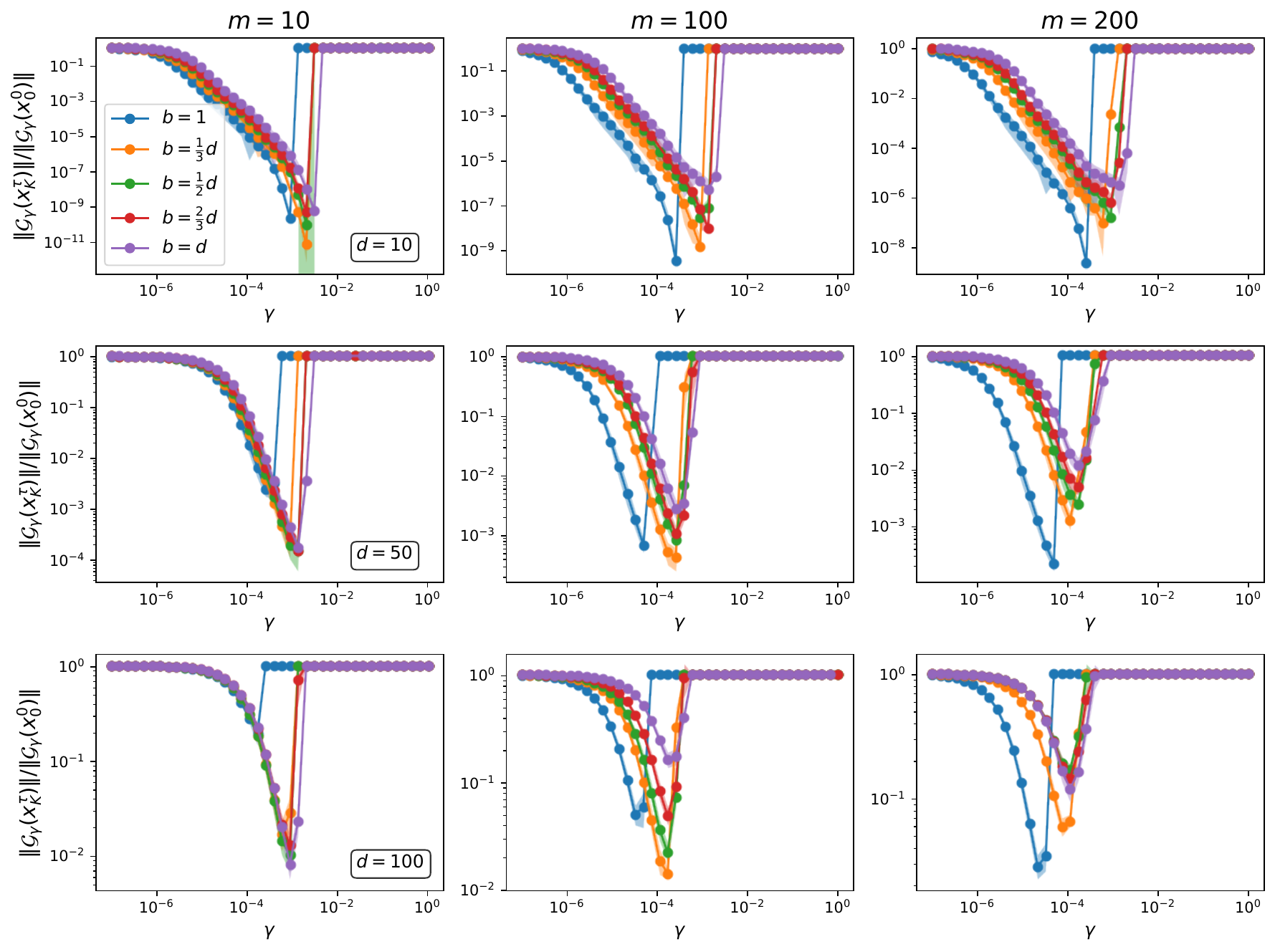}
    \caption{Normalized average generalized gradient mapping norm at the last iteration of Algorithm~\ref{algo:osvrz} as a function of the stepsize for different choices of batch size $b$. For parameter configurations that lead to divergence, the corresponding values are clipped to one. Rows correspond to different dimensions $d \in \{10, 50, 100\}$, while columns correspond to different numbers of inner iterations $m \in \{10, 100, 200\}$.}
    \label{fig:abl_change_b}
\end{figure}
As in the case of $\ell$, increasing $b$ enlarges the range of admissible stepsizes and shifts the minimum of the curves toward larger $\gamma$, with the effect becoming more pronounced as $m$ increases for similar reasons. For larger $m$, the method relies for longer on the full-gradient approximation computed at the beginning of the inner loop, so that a more accurate gradient estimate, obtained through a larger batch size, is needed to preserve stability and to tolerate larger stepsizes. However, unlike the case of $\ell$, increasing the batch size does not always translate into a comparable improvement in the best attainable performance. In particular, for large $m$ and small dimension (e.g., $d = 10$), $b = 1$ achieves a lower final normalized generalized gradient mapping norm than larger batch sizes, despite the latter admitting a wider range of stepsizes, suggesting that the quality of the gradient surrogates produced by increasing $b$ is inferior to that obtained by increasing $\ell$. This can be explained by the fact that, in this setting, the stochastic gradients $v_k^\tau$ are approximated by averaging $b$ single-direction finite-difference estimators; known prior results show that structured directions yield better approximation quality than unstructured ones, and that a substantially larger number of unstructured directions is required to reach a comparable level of accuracy~\cite{rando2025structuredtouroptimizationfinite,berahas2022theoretical}. Moreover, increasing the number of unstructured directions (i.e., the batch size, in this setting) incurs a higher per-iteration cost, which, under a limited budget, translates into fewer outer iterations and hence less refinement of the candidate solution. Such a limitation, as before, is especially visible in the high-dimensional regime (e.g., $d = 100$).  We confirm this by directly comparing the performance achieved by VR-SZD when using structured versus unstructured directions. We fix the batch size to $b=1$ and, for each experiment, the number of directions $\ell$ is tuned via grid-search. Figure~\ref{fig:abl_str_vs_unstr} reports the normalized generalized gradient mapping norm as a function of $\gamma$, across problem dimension $d$ (rows) and number of inner iterations $m$ (columns), for VR-SZD using structured and unstructured directions.
\begin{figure}[H]
    \centering
    \includegraphics[width=\linewidth]{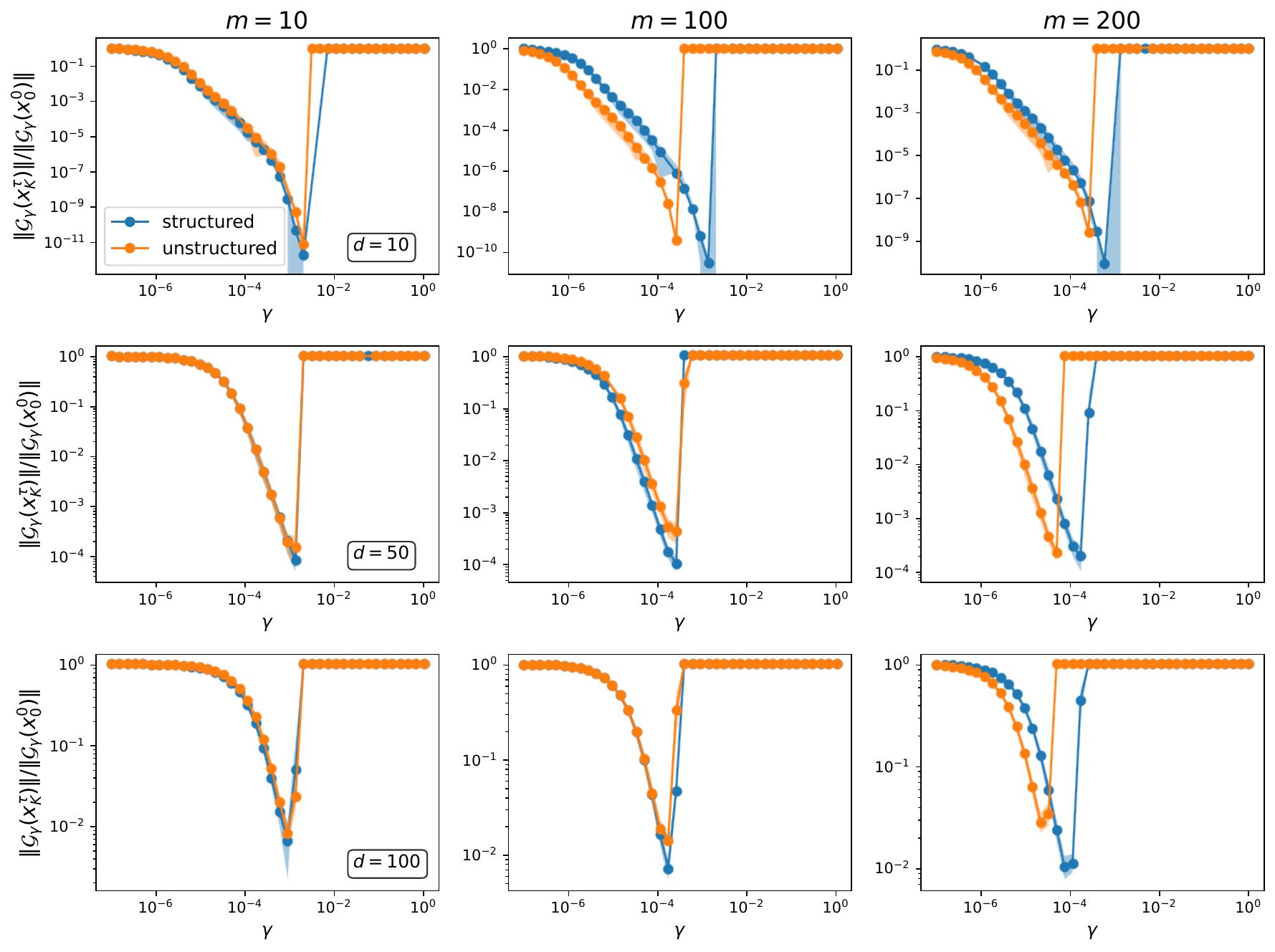}
    \caption{Normalized average generalized gradient mapping norm at the last iteration of Algorithm~\ref{algo:osvrz} as a function of the stepsize, comparing structured and unstructured directions, with the number of directions $\ell$ tuned via grid-search and $b=1$. For parameter configurations that lead to divergence, the corresponding values are clipped to one. Rows correspond to different dimensions $d \in \{10, 50, 100\}$, while columns correspond to different numbers of inner iterations $m \in \{10, 100, 200\}$.}
    \label{fig:abl_str_vs_unstr}
\end{figure}
\noindent We observe that for small values of $m$ (e.g., $m=10$), the performance of structured and unstructured directions is comparable, and in several cases nearly identical. This is because, with few inner iterations, the full-gradient approximation used to construct the reference point (which in both cases is built using coordinate directions) is updated frequently, so that the iterates never rely for long on stale gradient surrogate information, limiting the impact of the direction type on the quality of the stochastic gradient surrogate. As $m$ increases, however, structured directions yield better performance than unstructured ones, both in terms of the admissible range of stepsizes and the best attainable normalized generalized gradient mapping norm. This gap is more visible in the high-dimensional regime (e.g., $d = 100$), where, for large $m$, structured directions consistently achieve a lower final normalized generalized gradient mapping norm and tolerate larger stepsizes than their unstructured counterpart. %

\section{Limitations}\label{app:limitations}

In this appendix, we discuss the main limitations of Algorithm \ref{algo:osvrz}. The first limitation is the computational cost of our method. At each outer iteration $\tau$, the procedure computes an approximation of the gradient of the objective function $F$ and $m$ approximations of stochastic gradients using $\ell$ directions and $b$ batchsize that corresponds to $n(d + 1) + 2bm(\ell + 1)$ function evaluations. This implies that a large budget of function evaluations is required to execute the algorithm and, therefore, in cases where the objective function can only be evaluated a limited number of times, the algorithm may provides poor performance. This limitation is also present in other zeroth-order variance reduction methods (see, e.g., \cite{ji_improved_vr}) and could potentially be mitigated by using cheaper approximations of the full gradient. Another practical limitation is the number of parameters. Running Algorithm \ref{algo:osvrz} requires tuning several parameters, which can be challenging and computationally expensive. One way to address this issue is by defining adaptive strategies for parameter selection.

\section{Preliminary Results}\label{app:aux_results}
In this appendix, we state and collect lemmas and propositions required to prove the main results.
\begin{lemma}\label{lem:sum_orth}
    Let $(v_i)_{i = 1}^\ell \in \mathbb{R}^d$ such that for $i, j = 1, \cdots, \ell$ and $i \neq j$, $v_i^\intercal v_i = 1$ and $v_i^\intercal v_j = 0$.
    Then  
    \begin{equation*}
        \Bigg\| \sum\limits_{i = 1}^\ell v_i \Bigg\|^2 = \sum\limits_{i=1}^\ell \| v_i \|^2
    \end{equation*}
\end{lemma}
\begin{proof}
    This is a standard result and is covered in many linear algebra textbooks, such as \cite{lin_alg_book}.

\end{proof}

\begin{lemma}\label{lem:sph_exp}
    Let $I \in \mathbb{R}^{d \times d}$ be the identity matrix and $\sigma$ be the normalized measure on the sphere. Then
    \begin{equation*}
        \int_{\mathbb{S}^{d - 1}} v v^\intercal \, d\sigma(v) = \frac{1}{d} I.
    \end{equation*}
\end{lemma}
\begin{proof}
    This result follows the same line of \cite[Lemma 7.3, point (b)]{Gao2018}.
\end{proof}

\begin{lemma}\label{lem:equivmeasures}
    Let $f \colon \mathbb{S}^{d-1} \longrightarrow \mathbb{R}$ be an integrable function. Let $i \in [d]$.
    Then
    \begin{align*}
        \int_{O(d)} f(Ge_{i}) d\mu(G) = \int_{S^{d-1}} f(v) d\sigma^{d-1}(v).
    \end{align*}
\end{lemma}
\begin{proof} Let $h_{e_i} \colon O(d) \longrightarrow \mathbb{S}^{d-1}$ the function $G \mapsto Ge_i$. $h_{e_i}$ is a measurable function as it is continuous. Thanks to \cite[Theorem 3.7]{mattila_1995}, we know that the uniformly distributed measure $\sigma^{d-1}$ on $\mathbb{S}^{d-1}$ is the pushforward measure ${h_{e_i}}_{\#}\mu$ of the normalized Haar measure $\mu$ on $O(d)$ by $h_{e_i}$. So, by the change of variable formula,
\begin{align*}
    \int_{O(d)} f(Ge_{i}) d\mu(G) &= \int_{O(d)} f\circ h_{e_i}(G) d\mu(G) \\
    &= \int_{S^{d-1}} f(v) d{h_{e_i}}_{\#}\mu(v) \\
    &= \int_{S^{d-1}} f(v) d\sigma^{d-1}(v).
\end{align*}
\end{proof}

\paragraph*{Smoothing.} The analysis of Algorithm \ref{algo:osvrz} relies on the fact that the structured gradient approximation in eq. \eqref{eqn:stochastic_approx} %
is an unbiased estimator of the gradient of a smooth approximation of the objective function $f$. More precisely, for any $\beta > 0$, let $f_\beta$ be the smooth approximation of $f$ defined as
\begin{equation}\label{eqn:smoothing}
    f_\beta(x) := \mathbb{E}_{u \sim \mathcal{U}(\mathbb{B}^{d})}[f(x + \beta u)],
\end{equation}
where $\mathcal{U}(\mathbb{B}^{d})$ denotes the uniform distribution over the unit ball $\mathbb{B}^{d}$. Notice that $f_\beta$ is differentiable even when $f$ is not \cite{Bertsekas1973} and the relationship between $f$ and $f_\beta$  depends on the properties of $f$ - see Proposition \ref{prop:smt_properties}. Moreover, with the following Lemma, we show that the gradient of the smooth approximation $f_\beta$ can be expressed as an expectation of the surrogate in eq. \eqref{eqn:stochastic_approx}.
\begin{lemma}[Smoothing Lemma]\label{lem:smt_lemma}
Let $(\Omega, \mathcal{F}, \mathbb{P})$ and $(\Omega^\prime, \mathcal{F}^\prime, \mathbb{P}^\prime)$ be probability spaces. Let $G : \Omega \to O(d)$ and $i : \Omega^\prime \to [n]$ be random variables where $O(d)$ is the orthogonal group and $[n] := \{1, \dots, n\}$. Assume that the probability distribution of $G$ is the (normalized) Haar measure and $i$ is uniformly distributed. Let $\beta > 0$ and let $\hat{g}$ be the stochastic gradient surrogate defined in eq. \eqref{eqn:stochastic_approx}. Then,
\begin{equation}\label{eqn:smt_lemma}
    \mathbb{E}_{G}[\mathbb{E}_i[\hat{g}_{i}(x, G, \beta)]] = \nabla f_{\beta}(x).
\end{equation}
\end{lemma}
\begin{proof}
    Let $\hat{g}$ be the stochastic gradient surrogate defined in eq. \eqref{eqn:stochastic_approx}. By the definition of $\hat{g}$, we have
\begin{equation*}
    \mathbb{E}_{G}[\mathbb{E}_i[\hat{g}_{i}(x, G, \beta)]] = \mathbb{E}_{G}\left[\mathbb{E}_i\left[\frac{d}{\ell} \sum\limits_{j = 1}^\ell \frac{f_i(x + \beta G e_j) - f_i(x)}{\beta}Ge_j \right] \right].
\end{equation*}
By the definition of the objective function $f$ (eq. \eqref{eqn:problem}), we have
\begin{equation*}
    \mathbb{E}_{G}[\mathbb{E}_i[\hat{g}_{i}(x, G, \beta)]] = \mathbb{E}_{G}\left[\frac{d}{\ell} \sum\limits_{j = 1}^\ell \frac{f(x + \beta G e_j) - f(x)}{\beta}Ge_j \right].
\end{equation*}
By Lemma \ref{lem:equivmeasures}, 
\begin{equation*}
\begin{aligned}
\mathbb{E}_{G}\left[\frac{d}{\ell} \sum\limits_{j = 1}^\ell \frac{f(x + \beta G e_j) - f(x)}{\beta}Ge_j \right] &= \int_{O(d)} \frac{d}{\ell} \sum\limits_{j = 1}^\ell \frac{f(x + \beta G e_j) - f(x)}{\beta}Ge_j d \mu(G)\\
&= \frac{d}{\ell} \sum\limits_{j = 1}^\ell  \int_{S^{d - 1}} \frac{f(x + \beta v^{(j)}) - f(x)}{\beta}v^{(j)}  d\sigma^{d - 1}(v^{(j)})\\
&= d  \int_{S^{d - 1}} \frac{f(x + \beta v) - f(x)}{\beta}v  d\sigma^{d - 1}(v).
\end{aligned}
\end{equation*}
Since $v$ is uniformly random over the sphere, by symmetry, we have
\begin{equation*}
\begin{aligned}
d  \int_{S^{d - 1}} \frac{f(x + \beta v) - f(x)}{\beta}v  d\sigma^{d - 1}(v) &= d \int_{S^{d - 1}} \frac{ f(x + \beta v) }{\beta}  v  d\sigma^{d - 1}(v).    
\end{aligned}
\end{equation*}
As a consequence of Stokes’ Theorem (see \cite[Lemma 1]{flaxman2005online}), we have
\begin{equation*}
\begin{aligned}
d  \int_{S^{d - 1}} \frac{ f(x + \beta v) }{\beta}  v  d\sigma^{d - 1}(v) = \nabla f_\beta(x).    
\end{aligned}
\end{equation*}
This concludes the proof.
\end{proof}

\begin{proposition}[Smoothing Properties]\label{prop:smt_properties}
    Let $f_\beta$ be defined as in eq. \eqref{eqn:smoothing}. If $f$ satisfies Assumption \ref{ass:l_smooth}, then also $f_\beta$ satisfies it and, for every $x \in \mathbb{R}^d$, 
    \begin{equation}\label{eqn:smt_props_1}
        f_\beta(x) \leq f(x) + \frac{L}{2}\beta^2 \qquad \text{and} \qquad \|\nabla f_\beta(x) - \nabla f(x) \|^2 \leq \frac{L^2 d^2}{4}\beta^2.
    \end{equation}
    Moreover, for every $x\in\mathbb{R}^d$,
    \begin{equation}\label{eqn:smt_props_2}
        \| \nabla f_\beta(x) \|^2 \leq 2 \| \nabla f(x) \|^2 + \frac{L^2 d^2}{2} \beta^2.
    \end{equation}
\end{proposition}
\begin{proof}
    The inequalities in \eqref{eqn:smt_props_1} are standard results and the proofs can be found in different works, see for example \cite[Proposition 7.5, Proposition 6.8]{Gao2018}. To prove inequality \eqref{eqn:smt_props_2}, %
    \begin{equation*}
        \| \nabla f_\beta(x) \|^2 = \| \nabla f_\beta(x) - \nabla f(x) + \nabla f(x) \|^2 \leq 2 \| \nabla f_\beta(x) - \nabla f(x)\|^2 + 2\|\nabla f(x) \|^2,
    \end{equation*}
    where in the inequality we used $\|a + b\|^2 \leq 2\|a\|^2 + 2\|b\|^2$. %
    We conclude the proof by eq. \eqref{eqn:smt_props_1}.
\end{proof}

\subsection{Auxiliary Results}

\begin{lemma}\label{lem:prox_fun_bound}
Under Assumption \ref{ass:l_smooth}, let $x, v \in \mathbb{R}^d$, $\gamma > 0$ and $\bar{x} := \prox{\gamma h}{x - \gamma v}$. Then, for every $w \in \mathbb{R}^d$, we have
\begin{equation*}
    \begin{aligned}
            F(\bar{x}) &\leq F(w) + \scalT{\nabla f(x) - v}{\bar{x} - w} - \frac{1}{\gamma} \scalT{\bar{x} - x}{ \bar{x} - w} + \frac{L}{2}\|\bar{x} - x\|^2 + \frac{L}{2} \| w - x \|^2.
    \end{aligned}
\end{equation*}    
\end{lemma}
\begin{proof}
    The proof of this Lemma can be found in \cite[Lemma 13]{Kazemi2024}.
\end{proof}

\begin{lemma}\label{lem:prox_scal_bound}
    Under Assumption \ref{ass:l_smooth}, let $x, v \in \mathbb{R}^d$, $\gamma >0$, $\tilde{x} := \prox{\gamma h}{x - \gamma v}$ and $\bar{x} := \prox{\gamma h}{ x - \gamma \nabla f(x)}$. Then, 
    \begin{equation*}
        \scalT{\nabla f(x) - v}{\tilde{x} - \bar{x}} \leq \gamma \| \nabla f(x) - v \|^2.
    \end{equation*}
\end{lemma}
\begin{proof}
    The inequality holds by Cauchy-Schwarz and non-expansiveness of the prox operator,
    \begin{equation*}
        \scalT{\nabla f(x) - v}{\tilde{x} - \bar{x}} \leq \| \nabla f(x) - v\| \|\tilde{x} - \bar{x} \| \leq \gamma \| \nabla f(x) - v \|^2.        
    \end{equation*}
\end{proof}

\begin{lemma}[Approximation Error]\label{lem:approx_error}
    Under Assumption \ref{ass:l_smooth}, for $\beta > 0$, let $g(\cdot, \beta)$ be the deterministic gradient approximation defined in eq. \eqref{eqn:full_approx} and let $f_\beta$ be the smooth approximation of $f$ defined in eq. \eqref{eqn:smoothing}. Then, for every $x\in\mathbb{R}^d$,
    \begin{equation}\label{eqn:det_approx_error}
        \begin{aligned}
        \| g(x, \beta) - \nabla f(x) \|^2 &\leq \frac{L^2d}{4} \beta^2 \qquad \text{and} \qquad \| g(x, \beta) - \nabla f_\beta(x) \|^2 &\leq (1 + d)\frac{L^2 d}{2}\beta^2.
        \end{aligned}
    \end{equation}
    Moreover, let $\hat{g}_i(\cdot, G, \beta)$ be the stochastic gradient approximation defined in eq. \eqref{eqn:stochastic_approx} for an arbitrary $G \in O(d)$ and $h >0$. Then, for every $x, y \in \mathbb{R}^d$,
    \begin{equation}\label{eqn:sto_approx_lip}
        \mathbb{E}_{i} \mathbb{E}_G[ \| \hat{g}_i(x, G, \beta) - \hat{g}_i(y, G, \beta) \|^2] \leq 3\frac{d}{\ell}  \mathbb{E}_i \Bigg[ \| \nabla f_i(x) - \nabla f_i(y) \|^2\Bigg] + 3\frac{L^2d^2}{2 \ell} \beta^2.
    \end{equation}
\end{lemma}
\begin{proof}
    To prove eq. \eqref{eqn:det_approx_error}, we observe that, by Lemma \ref{lem:sum_orth},
    \begin{equation*}
        \begin{aligned}
            \| g(x, \beta) - \nabla f(x) \|^2 &= \Bigg\| g(x, \beta) - \sum\limits_{j = 1}^d \scalT{\nabla f(x)}{e_j}e_j \Bigg\|^2\\
            &= \frac{1}{\beta^2} \sum\limits_{j = 1}^d (f(x + \beta e_j) - f(x) - \scalT{\nabla f(x)}{ \beta e_j})^2.
        \end{aligned}
    \end{equation*}
    By the Descent Lemma \cite{polyak1987introduction}, we have
    \begin{equation*}
        |f(x + \beta e_j) - f(x) - \scalT{\nabla f(x)}{ \beta e_j}| \leq \frac{L}{2} \| \beta e_i \|^2.
    \end{equation*}
    Using such an inequality, we get the claim
    \begin{equation}\label{eqn:det_bound}
        \begin{aligned}
            \| g(x, \beta) - \nabla f(x) \|^2 &\leq \frac{L^2d}{4} \beta^2.
        \end{aligned}
    \end{equation}
    To prove the second inequality in eq. \eqref{eqn:det_approx_error}, we add and subtract $\nabla f(x)$,
    \begin{equation*}
        \begin{aligned}
        \| g(x, \beta) - \nabla f_\beta(x)\|^2 &= \| g(x, \beta) - \nabla f_\beta(x) + \nabla f(x) - \nabla f(x) \|^2\\
        &\leq 2\| g(x, \beta) - \nabla f(x) \|^2 + 2\| \nabla f(x) - \nabla f_\beta(x) \|^2.            
        \end{aligned}
    \end{equation*}
    By inequality \eqref{eqn:det_bound} and Proposition \ref{prop:smt_properties}, we get the claim,
    \begin{equation*}
        \begin{aligned}
        \| g(x, \beta) - \nabla f_\beta(x)\|^2 \leq  \frac{L^2 d}{2}\beta^2  + \frac{L^2 d^2}{2}\beta^2.            
        \end{aligned}
    \end{equation*}
    Now, we prove eq. \eqref{eqn:sto_approx_lip}. By orthogonality and since $\|G e_j\|^2 = 1$ for every $j = 1, \cdots, \ell$,
    \begin{equation*}
        \begin{aligned}
            \mathbb{E}_{i} \mathbb{E}_G[ \| \hat{g}_i(x, G,\beta) - \hat{g}_i(y, G, \beta) \|^2] &= \frac{d^2}{\ell^2 \beta^2} \sum\limits_{j = 1}^\ell \mathbb{E}_{i} \mathbb{E}_G[ (f_i(x + \beta G e_j) - f_i(x) - (f_i(y + \beta G e_j) - f_i(y)))^2].
        \end{aligned}
    \end{equation*}
    By Lemma \ref{lem:equivmeasures},
    \begin{equation*}
        \begin{aligned}
            \mathbb{E}_{i}\mathbb{E}_G[ \| \hat{g}_i(x, G, \beta) - \hat{g}_i(y, G, \beta) \|^2] &= \frac{d^2}{\ell^2 \beta^2} \sum\limits_{j = 1}^\ell\mathbb{E}_{i}\mathbb{E}_{v^{(j)}} [ (f_i(x + \beta v^{(j)}) - f_i(x) - (f_i(y + \beta v^{(j)}) - f_i(y)))^2]\\
            &= \frac{d^2}{\ell \beta^2} \mathbb{E}_{i}\mathbb{E}_{v} [ (f_i(x + \beta v) - f_i(x) - (f_i(y + \beta v) - f_i(y)))^2].
        \end{aligned}
    \end{equation*}
    Adding and subtracting $\scalT{\nabla f_i(x)}{\beta v}$ and $\scalT{\nabla f_i(y)}{\beta v}$,
    \begin{equation*}
        \begin{aligned}
            \mathbb{E}_{i}\mathbb{E}_G[ \| \hat{g}_i(x, G, \beta) - \hat{g}_i(y, G, \beta) \|^2] &= \frac{d^2}{\ell \beta^2} \mathbb{E}_{i}\mathbb{E}_{v} \Bigg[ \Bigg(f_i(x + \beta v) - f_i(x) - \scalT{\nabla f_i(x)}{\beta v}\\
            &- \Bigg( f_i(y + \beta v) - f_i(y) - \scalT{\nabla f_i(y)}{\beta v} \Bigg)\\
            & + \scalT{\nabla f_i(x)}{\beta v} - \scalT{\nabla f_i(y)}{\beta v} \Bigg)^2 \Bigg].
        \end{aligned}
    \end{equation*}
    Since $\| a + b + c \|^2 \leq 3 \|a\|^2 + 3\|b\|^2 + 3\|c\|^2$,
    \begin{equation*}
        \begin{aligned}
            \mathbb{E}_{i}\mathbb{E}_G[ \| \hat{g}_i(x, G, \beta) - \hat{g}_i(y, G, \beta) \|^2] &\leq 3\frac{d^2}{\ell \beta^2} \mathbb{E}_{i}\mathbb{E}_{v} \Bigg[\Bigg( f_i(x + \beta v) - f_i(x) - \scalT{\nabla f_i(x)}{\beta v} \Bigg)^2 \Bigg]\\
            &+ 3\frac{d^2}{\ell \beta^2} \mathbb{E}_i \mathbb{E}_{v} \Bigg[ \Bigg( f_i(y + \beta v) - f_i(y) - \scalT{\nabla f_i(y)}{\beta v} \Bigg)^2 \Bigg]\\
            &+ 3\frac{d^2}{\ell \beta^2} \mathbb{E}_i \mathbb{E}_{v} \Bigg[ \Bigg( \scalT{\nabla f_i(x) - \nabla f_i(y)}{\beta v} \Bigg)^2\Bigg].
        \end{aligned}
    \end{equation*}
    By the Descent Lemma \cite{polyak1987introduction}, we get
    \begin{equation*}
        \begin{aligned}
            \mathbb{E}_{i}\mathbb{E}_G[ \| \hat{g}_i(x, G, \beta) - \hat{g}_i(y, G, \beta) \|^2] &\leq 3\frac{L^2d^2}{2 \ell} \beta^2 + 3\frac{d^2}{\ell} \mathbb{E}_i \mathbb{E}_{v} \Bigg[ \Bigg( \scalT{\nabla f_i(x) - \nabla f_i(y)}{ v} \Bigg)^2\Bigg]\\
            &= 3\frac{L^2d^2}{2 \ell} \beta^2 + 3\frac{d^2}{\ell} \mathbb{E}_i \scalT{\nabla f_i(x) - \nabla f_i(y)}{\mathbb{E}_v[v v^\intercal] (\nabla f_i(x) - \nabla f_i(y)) }. %
        \end{aligned}
    \end{equation*}
    By Lemma \ref{lem:sph_exp}, we get the claim,
    \begin{equation*}
        \begin{aligned}
            \mathbb{E}_{i}\mathbb{E}_G[ \| \hat{g}_i(x, G, \beta) - \hat{g}_i(y, G, \beta) \|^2] &\leq 3\frac{L^2d^2}{2 \ell} \beta^2 + 3\frac{d}{\ell}  \mathbb{E}_i \Bigg[ \| \nabla f_i(x) - \nabla f_i(y) \|^2\Bigg].
        \end{aligned}
    \end{equation*}

\end{proof}

\begin{lemma}[Bound on direction]\label{lem:bound_vk}
Under Assumption \ref{ass:l_smooth}, let $v_k^\tau$ be the direction vector at inner iteration $k \in [m]$ and outer iteration $\tau \in \mathbb{N}$ defined in eq. \eqref{eqn:osvrz_v}. Then,
     \begin{equation*}
        \begin{aligned}
            \mathbb{E}_{I_k^\tau, G_k^\tau}[\|v_k^\tau - \nabla f(x_k^\tau) \|^2 ] &\leq \frac{6d}{\ell b} L^2 \|x_k^{\tau} - x_0^\tau \|^2 + \left( \frac{3 d}{\ell b} + 2 + 3d \right)L^2 d \beta_\tau^2.%
        \end{aligned}
     \end{equation*}
\end{lemma}
\begin{proof}
    To simplify the notation and improve readability, we denote
    \begin{equation}\label{eqn:direction_bound_delta}
        \delta_{j, k}^{\tau} := \hat{g}_{i_{j, k}^\tau}(x_k^\tau, G_{j,k}^\tau, \beta_\tau) - \hat{g}_{i_{j, k}^\tau}(x_0^\tau, G_{j,k}^\tau, \beta_\tau) - (\nabla f_{\beta_\tau}(x_k^\tau) - \nabla f_{\beta_\tau}(x_0^\tau))).
    \end{equation}
    Thus, adding and subtracting $\nabla f_{\beta_\tau}(x_k^\tau)$ and $\nabla f_{\beta_\tau}(x_0^\tau)$ to $v_k^\tau - \nabla f(x_k^\tau)$, we get
    \begin{equation}\label{eqn:bound_dir_eq1}
        \begin{aligned}
            \mathbb{E}_{I_k^\tau} \Big[\mathbb{E}_{G_k^\tau} \big[\|v_k^\tau - \nabla f(x_k^\tau) \|^2  \big] \Big]&\leq 2\mathbb{E}_{I_k^\tau} \left[ \mathbb{E}_{G_k^\tau} \left[ \left\| \frac{1}{b} \sum\limits_{j = 1}^b \delta_{j, k}^\tau \right\|^2 \right]  \right]\\
            &+ 4 \| g(x_0^\tau, \beta_\tau) - \nabla f_{\beta_\tau}(x_0^\tau) \|^2\\
            &+ 4\| \nabla f_{\beta_{\tau}}(x_k^\tau)- \nabla f(x_k^\tau) \|^2 .
        \end{aligned}
    \end{equation}
    Now, we bound the first term of the inequality,
    \begin{equation*}
        \begin{aligned}
            2\mathbb{E}_{I_k^\tau} \left[ \mathbb{E}_{G_k^\tau} \left[ \left\| \frac{1}{b} \sum\limits_{j = 1}^b \delta_{j, k}^\tau \right\|^2 \right]  \right] &= \frac{2}{b^2} \Bigg( \sum\limits_{j = 1}^b  \mathbb{E}_{I_k^\tau} \left[ \mathbb{E}_{G_k^\tau} \left[ \| \delta_{j, k}^\tau \|^2 \right] \right]\\
            &+ \sum\limits_{j=1}^b \sum\limits_{t \neq j} \mathbb{E}_{I_k^\tau} \left[ \mathbb{E}_{G_k^\tau} \left[ \scalT{\delta_{j, k}^\tau}{\delta_{t, k}^\tau}  \right] \right] \Bigg).
        \end{aligned}
    \end{equation*}
    Since $i_{j, k}^\tau, i_{t, k}^\tau, G_{j,k}^\tau, G_{t, k}^\tau$ are independent for every $j \neq t$, we have
    \begin{equation*}
        \begin{aligned}
            2\mathbb{E}_{I_k^\tau} \left[ \mathbb{E}_{G_k^\tau} \left[ \left\| \frac{1}{b} \sum\limits_{j = 1}^b \delta_{j, k}^\tau \right\|^2 \right]  \right] &= \frac{2}{b^2} \Bigg( \sum\limits_{j = 1}^b  \mathbb{E}_{I_k^\tau} [ \mathbb{E}_{G_k^\tau} [ \| \delta_{j, k}^\tau \|^2 ] ]\\
            &+ \sum\limits_{j=1}^b \sum\limits_{t \neq j}  \scalT{\mathbb{E}_{I_k^\tau} [ \mathbb{E}_{G_k^\tau} [\delta_{j, k}^\tau \, ] ]}{ \mathbb{E}_{I_k^\tau} [ \mathbb{E}_{G_k^\tau} [\delta_{t, k}^\tau ]  ]}  \Bigg).
        \end{aligned}
    \end{equation*}
    By Lemma \ref{lem:smt_lemma}, we have that $\mathbb{E}_{I_k^\tau} [ \mathbb{E}_{G_k^\tau} [\delta_{j, k}^\tau ]  ] = 0$. Thus,
    \begin{equation*}
        \begin{aligned}
            2\mathbb{E}_{I_k^\tau} \left[ \mathbb{E}_{G_k^\tau} \left[ \left\| \frac{1}{b} \sum\limits_{j = 1}^b \delta_{j, k}^\tau \right\|^2 \right] \right] &= \frac{2}{b^2}  \sum\limits_{j = 1}^b  \mathbb{E}_{I_k^\tau} [ \mathbb{E}_{G_k^\tau} [ \| \delta_{j, k}^\tau \|^2 ]  ].
        \end{aligned}
    \end{equation*}
    By eq. \eqref{eqn:direction_bound_delta} and Lemma \ref{lem:smt_lemma}, we have 
    \begin{equation*}
        \begin{aligned}
            \mathbb{E}_{I_k^\tau} [ \mathbb{E}_{G_k^\tau} [ \| \delta_{j, k}^\tau \|^2 ]  ] &= \mathbb{E}_{I_k^\tau} [ \mathbb{E}_{G_k^\tau} [ \| \hat{g}_{i_{j, k}^\tau}(x_k^\tau, G_{j,k}^\tau, \beta_\tau) - \hat{g}_{i_{j, k}^\tau}(x_0^\tau, G_{j,k}^\tau, \beta_\tau) \|^2 ] ] +  \| \nabla f_{\beta_\tau}(x_k^\tau) - \nabla f_{\beta_\tau}(x_0^\tau) \|^2\\
            &-2 \mathbb{E}_{I_k^\tau} \left[ \mathbb{E}_{G_k^\tau} \left[ \scalT{\hat{g}_{i_{j, k}^\tau}(x_k^\tau, G_{j,k}^\tau, \beta_\tau) - \hat{g}_{i_{j, k}^\tau}(x_0^\tau, G_{j,k}^\tau, \beta_\tau) }{ \nabla f_{\beta_\tau}(x_k^\tau) - \nabla f_{\beta_\tau}(x_0^\tau)} \right]  \right]\\
            &= \mathbb{E}_{I_k^\tau} [ \mathbb{E}_{G_k^\tau} [ \| \hat{g}_{i_{j, k}^\tau}(x_k^\tau, G_{j,k}^\tau, \beta_\tau) - \hat{g}_{i_{j, k}^\tau}(x_0^\tau, G_{j,k}^\tau, \beta_\tau) \|^2 ]  ] -  \| \nabla f_{\beta_\tau}(x_k^\tau) - \nabla f_{\beta_\tau}(x_0^\tau) \|^2\\
            &\leq \mathbb{E}_{I_k^\tau} [ \mathbb{E}_{G_k^\tau} [ \| \hat{g}_{i_{j, k}^\tau}(x_k^\tau, G_{j,k}^\tau, \beta_\tau) - \hat{g}_{i_{j, k}^\tau}(x_0^\tau, G_{j,k}^\tau, \beta_\tau) \|^2 ]  ].
        \end{aligned}
    \end{equation*}
    By using this inequality in \eqref{eqn:bound_dir_eq1} , we get
    \begin{equation*}
        \begin{aligned}
            \mathbb{E}_{I_k^\tau} [\mathbb{E}_{G_k^\tau} \left[\|v_k^\tau - \nabla f(x_k^\tau) \|^2  \right]  ] &\leq \frac{2}{b^2}  \sum\limits_{j = 1}^b \mathbb{E}_{G_k^\tau} \left[\mathbb{E}_{I_k^\tau} \left[ \left\| \hat{g}_{i_{j, k}^\tau}(x_k^\tau, G_k^\tau, \beta_\tau) - \hat{g}_{i_{j, k}^\tau}(x_0^\tau, G_k^\tau, \beta_\tau) \right\|^2 \right]  \right]\\
            &+ 4 \| g(x_0^\tau, \beta_\tau) - \nabla f_{\beta_\tau}(x_0^\tau) \|^2\\
            &+ 4\| \nabla f_{\beta_{\tau}}(x_k^\tau)- \nabla f(x_k^\tau) \|^2 .
        \end{aligned}
    \end{equation*}
    By Proposition \ref{prop:smt_properties}, Lemma \ref{lem:approx_error} and Assumption \ref{ass:l_smooth},
    \begin{equation*}
        \begin{aligned}
            \mathbb{E}_{I_k^\tau} \left[\mathbb{E}_{G_k^\tau} \left[\|v_k^\tau - \nabla f(x_k^\tau) \|^2  \right]  \right] %
            &\leq \frac{6d}{\ell b} L^2 \|x_k^{\tau} - x_0^\tau \|^2 + \left( \frac{3 d}{\ell b} + 2 + 3d \right)L^2 d \beta_\tau^2. %
        \end{aligned}
    \end{equation*}

\end{proof}

\begin{lemma}[Function Values bound]\label{lem:funval_bound}
    Under Assumption \ref{ass:l_smooth},
    \begin{equation*}
        \begin{aligned}
            \mathbb{E}_{I_k^\tau}[\mathbb{E}_{G_k^\tau}[F(x_{k + 1}^\tau) ] ] &\leq F(x_k^\tau) - \left( \frac{1}{2} - L \gamma \right) \gamma \| \mathcal{G}_\gamma(x_k^\tau) \|^2 + \frac{6d}{\ell b} L^2 \gamma \|x_k^{\tau} - x_0^\tau \|^2\\
            &- \left(\frac{1}{2\gamma} - \frac{L}{2} \right) \mathbb{E}_{I_k^\tau} \left[\mathbb{E}_{G_k^\tau} \left[ \| x_{k + 1}^\tau - x_k^\tau \|^2  \right]  \right]\\
            &+ \left( \frac{3 d}{\ell b} + 2 + 3d \right)L^2 d \gamma \beta_\tau^2.
        \end{aligned}
    \end{equation*}
\end{lemma}
\begin{proof}
    Let $\bar{x}_{k+1}^\tau := \prox{\gamma h}{x_k^\tau - \gamma \nabla f(x_k^\tau)}$. By Lemma \ref{lem:prox_fun_bound} with $\bar{x} = \bar{x}_{k + 1}^\tau, x =x_k^\tau, v = \nabla f(x_k^\tau)$ and $w = x_k^\tau$, we get
    \begin{equation}\label{eqn:fb_eq1}
        F(\bar{x}_{k + 1}^\tau) \leq F(x_k^\tau) - \left( \frac{1}{\gamma} - \frac{L}{2} \right) \| \bar{x}_{k + 1}^\tau - x_k^\tau \|^2.
    \end{equation}    
    Using again Lemma \ref{lem:prox_fun_bound} with $\bar{x} = x_{k + 1}^\tau, w = \bar{x}_{k + 1}^\tau, v = v_{k}^\tau$ and $x = x_k^\tau$, we get
    \begin{equation*}
        \begin{aligned}
            F(x_{k + 1}^\tau) &\leq F(\bar{x}_{k + 1}^\tau) + \scalT{\nabla f(x_k^\tau) - v_k^\tau}{x_{k + 1}^\tau - \bar{x}_{k + 1}^\tau} - \frac{1}{\gamma} \scalT{x_{k + 1}^\tau - x_k^\tau}{x_{k + 1}^\tau - \bar{x}_{k + 1}^\tau}\\
            &+\frac{L}{2} \| x_{k + 1}^\tau - x_k^\tau \|^2 +\frac{L}{2} \| \bar{x}_{k + 1}^\tau - x_k^\tau \|^2.
        \end{aligned}
    \end{equation*}
    By eq. \eqref{eqn:fb_eq1},
    \begin{equation*}
        \begin{aligned}
            F(x_{k + 1}^\tau) &\leq F(x_k^\tau) - \left( \frac{1}{\gamma} - L \right) \| \bar{x}_{k + 1}^\tau - x_k^\tau \|^2 +\frac{L}{2} \| x_{k + 1}^\tau - x_k^\tau \|^2 \\
            &+ \scalT{\nabla f(x_k^\tau) - v_k^\tau}{x_{k + 1}^\tau - \bar{x}_{k + 1}^\tau} - \frac{1}{\gamma} \scalT{x_{k + 1}^\tau - x_k^\tau}{x_{k + 1}^\tau - \bar{x}_{k + 1}^\tau}.
        \end{aligned}
    \end{equation*}
    By Lemma \ref{lem:prox_scal_bound}, recalling that $x_{k+1}^\tau = \prox{\gamma h}{x_k^\tau - \gamma v_k^\tau}$,
    \begin{equation}\label{eqn:fb_eq2}
        \begin{aligned}
            F(x_{k + 1}^\tau) &\leq F(x_k^\tau) - \left( \frac{1}{\gamma} - L \right) \| \bar{x}_{k + 1}^\tau - x_k^\tau \|^2 +\frac{L}{2} \| x_{k + 1}^\tau - x_k^\tau \|^2 \\
            &+ \gamma \| \nabla f(x_k^\tau) - v_k^\tau \|^2 - \frac{1}{\gamma} \scalT{x_{k + 1}^\tau - x_k^\tau}{x_{k + 1}^\tau - \bar{x}_{k + 1}^\tau}.
        \end{aligned}
    \end{equation}
    Now, to bound the last term, we add and subtract $x_k^\tau$, %
    \begin{equation*}
        \begin{aligned}
        - \frac{1}{\gamma} \scalT{x_{k + 1}^\tau - x_k^\tau}{x_{k + 1}^\tau - \bar{x}_{k + 1}^\tau} = - \frac{1}{\gamma} \| x_{k + 1}^\tau - x_k^\tau \|^2 + \frac{1}{\gamma} \scalT{x_{k +1}^\tau - x_k^\tau}{ \bar{x}_{k + 1}^\tau - x_k^\tau   }.
        \end{aligned}
    \end{equation*}
    By Young's inequality, we get
    \begin{equation*}
        \begin{aligned}
        - \frac{1}{\gamma} \scalT{x_{k + 1}^\tau - x_k^\tau}{x_{k + 1}^\tau - \bar{x}_{k + 1}^\tau} &\leq - \frac{1}{2\gamma} \| x_{k + 1}^\tau - x_k^\tau \|^2 + \frac{1}{2\gamma} \| \bar{x}_{k + 1}^\tau - x_k^\tau \|^2.
        \end{aligned}
    \end{equation*}
    Plugging this inequality in eq. \eqref{eqn:fb_eq2},
    \begin{equation*}
        \begin{aligned}
            F(x_{k + 1}^\tau) &\leq F(x_k^\tau) - \left( \frac{1}{2\gamma} - L \right) \| \bar{x}_{k + 1}^\tau - x_k^\tau \|^2\\
            &- \left(\frac{1}{2\gamma} - \frac{L}{2} \right) \| x_{k + 1}^\tau - x_k^\tau \|^2 + \gamma \| \nabla f(x_k^\tau) - v_k^\tau \|^2.
        \end{aligned}
    \end{equation*}
    Now, we take the expectations with respect to all $i_{j,k}^\tau$ and $G_{j, k}^\tau$ and conditioning on all past random variables. 
    By Lemma \ref{lem:bound_vk} and the definition of the generalized gradient $\mathcal{G}_\gamma$ (eq. \eqref{eqn:generalized_gradient}), we get
    \begin{equation*}
        \begin{aligned}
            \mathbb{E}_{G_k^\tau}[\mathbb{E}_{I_k^\tau}[F(x_{k + 1}^\tau) ] ] &\leq F(x_k^\tau) - \left( \frac{1}{2} - L \gamma \right) \gamma \| \mathcal{G}_\gamma(x_k^\tau) \|^2\\
            &- \left(\frac{1}{2\gamma} - \frac{L}{2} \right) \mathbb{E}_{G_k^\tau} \left[\mathbb{E}_{I_k^\tau} \left[ \| x_{k + 1}^\tau - x_k^\tau \|^2 \right] \right]\\
            &+ \frac{6d}{\ell b} L^2 \gamma \|x_k^{\tau} - x_0^\tau \|^2 + \left( \frac{3 d}{\ell b} + 2 + 3d \right)L^2 d \gamma \beta_\tau^2.
        \end{aligned}
    \end{equation*}
\end{proof}

\section{Proofs of the Main Results}\label{app:addproofs}

\subsection*{Proof of Theorem \ref{thm:nonconv_rate}}
We start the proof by defining the following Lyapunov function 
\begin{equation}\label{eqn:lyap}
    R_k^\tau := F(x_k^\tau) + c_k \| x_k^\tau - x_0^\tau \|^2,
\end{equation}
where $(c_k)_{k = 0}^{m} \subseteq \mathbb{R}_+$ with $c_m = 0$ will be defined recursively in the proof. Thus,
\begin{equation*}
    \begin{aligned}
        R_{k + 1}^\tau &= F(x_{k + 1}^\tau) + c_{k + 1} \| x_{k + 1}^\tau - x_0^\tau \|^2\\    
        &= F(x_{k + 1}^\tau) + c_{k + 1} \| x_{k + 1}^\tau - x_k^\tau + x_k^\tau - x_0^\tau \|^2\\
        &= F(x_{k + 1}^\tau) + c_{k + 1} \left[\| x_{k + 1}^\tau - x_k^\tau \|^2 + \| x_k^\tau - x_0^\tau \|^2 + 2 \scalT{x_{k + 1}^\tau - x_k^\tau}{x_k^\tau - x_0^\tau} \right].
    \end{aligned}
\end{equation*}
By Young's inequality with parameter $\theta_1 > 0$,
\begin{equation*}
    \begin{aligned}
        R_{k + 1}^\tau &\leq F(x_{k + 1}^\tau) + \left(1 + \frac{1}{\theta_1} \right)c_{k + 1} \| x_{k + 1}^\tau - x_k^\tau \|^2 + \left(1 + \theta_1 \right) c_{k + 1} \| x_k^\tau - x_0^\tau \|^2.
    \end{aligned}
\end{equation*}
Now, we take the expectation on all $G_{j,k}^\tau$ and $i_{j,k}^\tau$ conditioning on all past random variables. 
By Lemma \ref{lem:funval_bound} we have
\begin{equation*}
    \begin{aligned}
        \mathbb{E}_{G_k^\tau}[ \mathbb{E}_{I_k^\tau} [R_{k + 1}^\tau  ] ] &\leq F(x_k^\tau) - \left( \frac{1}{2} - L \gamma \right) \gamma \| \mathcal{G}_\gamma(x_k^\tau) \|^2\\
        &+ \left(\left(1 + \frac{1}{\theta_1} \right)c_{k + 1} + \frac{L}{2} - \frac{1}{2 \gamma} \right) \mathbb{E}_{G_k^\tau}[ \mathbb{E}_{I_k^\tau} [\| x_{k + 1}^\tau - x_k^\tau \|^2 ] ]\\
        &+ \left(\left(1 + \theta_1 \right) c_{k + 1} + \frac{6d}{\ell b} L^2 \gamma \right) \| x_k^\tau - x_0^\tau \|^2\\
        &+ \underbrace{\left( \frac{3 d}{\ell b} + 2 + 3d \right)L^2 d}_{=: C_1} \gamma \beta_\tau^2. 
    \end{aligned}
\end{equation*}
Let $c_k = \left(1 + \theta_1 \right) c_{k + 1} + \frac{6d}{\ell b} L^2 \gamma$. Then by eq. \eqref{eqn:lyap}
\begin{equation}\label{eqn:thm_nconv_1}
    \begin{aligned}
        \mathbb{E}_{G_k^\tau}[ \mathbb{E}_{I_k^\tau} [R_{k + 1}^\tau  ] ] &\leq R_k^\tau - \left( \frac{1}{2} - L \gamma \right) \gamma \| \mathcal{G}_\gamma(x_k^\tau) \|^2\\
        &+ \left(\left(1 + \frac{1}{\theta_1} \right)c_{k + 1} + \frac{L}{2} - \frac{1}{2 \gamma} \right) \mathbb{E}_{G_k^\tau}[ \mathbb{E}_{I_k^\tau} [\| x_{k + 1}^\tau - x_k^\tau \|^2  ] ]\\
        &+ C_1 \gamma \beta_\tau^2. 
    \end{aligned}
\end{equation}
For every $k \geq 0$,
\begin{equation}\label{eqn:ck_ub}
    \begin{aligned}
        c_{k} \leq c_0 = \frac{6 d L^2}{\ell b}\gamma \frac{(1 + \theta_1)^m - 1}{\theta_1} 
    \end{aligned}
\end{equation}
Thus, plugging eq. \eqref{eqn:ck_ub} in eq. \eqref{eqn:thm_nconv_1},
\begin{equation*}
    \begin{aligned}
        \mathbb{E}_{G_k^\tau}[ \mathbb{E}_{I_k^\tau} [R_{k + 1}^\tau  ] ] &\leq R_k^\tau - \left( \frac{1}{2} - L \gamma \right) \gamma \| \mathcal{G}_\gamma(x_k^\tau) \|^2\\
        &+ \left(\left(1 + \frac{1}{\theta_1} \right)c_0 + \frac{L}{2} - \frac{1}{2 \gamma} \right) \mathbb{E}_{G_k^\tau}[ \mathbb{E}_{I_k^\tau} [\| x_{k + 1}^\tau - x_k^\tau \|^2 ] ]\\
        &+ C_1 \gamma \beta_\tau^2. 
    \end{aligned}
\end{equation*}
Let $\theta_1 = \frac{1}{m}$. We have
\begin{equation*}
    c_0 \leq \frac{6 d L^2}{\ell b}\gamma \frac{\left(\left(1 + \frac{1}{m} \right)^m - 1 \right)}{\theta_1} \leq \frac{6 d L^2}{\ell b} \left( e  - 1 \right) m \gamma.
\end{equation*}
Thus, taking into account that, since $m> 1$ and $(m + 1)m \leq 2m^2$, we get
\begin{equation*}
    \begin{aligned}
        \mathbb{E}_{G_k^\tau}[ \mathbb{E}_{I_k^\tau} [R_{k + 1}^\tau ]] &\leq R_k^\tau - \left( \frac{1}{2} - L \gamma \right) \gamma \| \mathcal{G}_\gamma(x_k^\tau) \|^2\\
        &+ \left(\frac{12 d L^2}{\ell b} \left( e  - 1 \right) m^2 \gamma + \frac{L}{2} - \frac{1}{2 \gamma} \right) \mathbb{E}_{G_k^\tau}[ \mathbb{E}_{I_k^\tau} [\| x_{k + 1}^\tau - x_k^\tau \|^2  ] ]\\
        &+ C_1 \gamma \beta_\tau^2. 
    \end{aligned}
\end{equation*}
Since $\gamma < \min \left( \frac{1}{4L} , \frac{\sqrt{\ell b}}{\sqrt{32(e - 1) d} L m} \right)$, we have
\begin{equation*}
    \begin{aligned}
        \frac{12 d L^2}{\ell b} \left( e  - 1 \right) m^2 \gamma + \frac{L}{2} - \frac{1}{2 \gamma} < 0.
    \end{aligned}
\end{equation*}
This implies,
\begin{equation}\label{eqn:thm_nconv_2}
    \begin{aligned}
        \mathbb{E}_{G_k^\tau}[ \mathbb{E}_{I_k^\tau} [R_{k + 1}^\tau ] ] &\leq R_k^\tau - \left( \frac{1}{2} - L \gamma \right) \gamma \| \mathcal{G}_\gamma(x_k^\tau) \|^2 + C_1 \gamma \beta_\tau^2. 
    \end{aligned}
\end{equation}
Taking the full-expectation, summing for $k = 0, \cdots, m - 1$ and observing that $R_0^\tau = F(x_0^\tau)$ and $R_m^\tau = F(x_m^\tau) = F(x_0^{\tau + 1})$, we get
\begin{equation*}
    \begin{aligned}
        \mathbb{E} \left[F(x_0^{\tau + 1}) \right] &\leq \mathbb{E} \left[F(x_0^\tau) \right] - \left( \frac{1}{2} - L \gamma \right) \gamma \sum\limits_{k = 0}^{m - 1} \mathbb{E} \left[\| \mathcal{G}_\gamma(x_k^\tau) \|^2 \right] + C_1 m \gamma \beta_\tau^2. 
    \end{aligned}
\end{equation*}
Since $\frac{1}{2} - L \gamma  > 0$, summing for $\tau = 0, \cdots, T$, we have
\begin{equation*}
    \begin{aligned}
        \underbrace{\left( \frac{1}{2} - L \gamma \right)}_{=: \Delta} \gamma \sum\limits_{\tau = 0}^T \sum\limits_{k = 0}^{m - 1} \mathbb{E} \left[\| \mathcal{G}_\gamma(x_k^\tau) \|^2\right] &\leq \mathbb{E}\left[F(x_0^0) - F(x_0^{T + 1}) \right]  + C_1 m \gamma \sum\limits_{\tau = 0}^T \beta_\tau^2. 
    \end{aligned}
\end{equation*}
Noticing that $F(x_0^{T + 1}) \geq \min F$ and dividing by $\Delta \gamma m (T + 1)$, we conclude the proof:
\begin{equation*}
    \begin{aligned}
        \frac{1}{(T + 1) m} \sum\limits_{\tau = 0}^T \sum\limits_{k = 0}^{m - 1} \mathbb{E} \left[ \| \mathcal{G}_\gamma(x_k^\tau) \|^2 \right] &\leq \frac{ \mathbb{E}\left[F(x_0^0) - \min F \right]}{\Delta \gamma m (T + 1)}  + \frac{C_1}{\Delta (T + 1)} \sum\limits_{\tau = 0}^T \beta_\tau^2. 
    \end{aligned}
\end{equation*}

\subsection*{Proof of Corollary \ref{cor:cor1}}
From Theorem \ref{thm:nonconv_rate}, choosing the stepsize $\gamma < \min(\frac{1}{4L}, \frac{\sqrt{\ell b}}{10 L m \sqrt{d}})$, we have
\begin{equation*}%
    \begin{aligned}
        \frac{1}{(T + 1)m} \sum\limits_{\tau = 0}^T \sum\limits_{k = 0}^{m - 1} \mathbb{E} \left[ \| \mathcal{G}_\gamma (x_k^\tau) \|^2 \right] &\leq \frac{10 L \sqrt{d}}{\sqrt{\ell b}}\frac{ F(x_0^0) - \min F }{\Delta (T + 1)}\\
        &+ \left( \frac{3 d}{\ell b } + 2 + 3d \right) \frac{L^2 d}{\Delta (T + 1)} \sum\limits_{\tau = 0}^T \beta_\tau^2.
    \end{aligned} 
\end{equation*}
Note that, due to the choice of the stepsize, $\Delta \geq \frac{1}{4}$. Thus,
\begin{equation}\label{eqn:cor1_main_ineq}
    \begin{aligned}
        \frac{1}{(T + 1)m} \sum\limits_{\tau = 0}^T \sum\limits_{k = 0}^{m - 1} \mathbb{E} \left[ \| \mathcal{G}_\gamma (x_k^\tau) \|^2 \right] &\leq \frac{40 L \sqrt{d}}{\sqrt{\ell b}}\frac{ F(x_0^0) - \min F }{T + 1}\\
        &+ 4\left( \frac{3 d}{\ell b } + 2 + 3d \right) \frac{L^2 d}{T + 1} \sum\limits_{\tau = 0}^T \beta_\tau^2.
    \end{aligned} 
\end{equation}
Points (i) and (ii) can be proved by replacing the corresponding value of $\beta_\tau$ in eq. \eqref{eqn:cor1_main_ineq}. Indeed, with the choice (i), we have $\beta_\tau = \frac{\beta}{d} (\tau + 1)^{-\alpha}$ with $\alpha > 1/2$. Thus, we get
\begin{equation*}
    \begin{aligned}
        \frac{1}{(T + 1)m} \sum\limits_{\tau = 0}^T \sum\limits_{k = 0}^{m - 1} \mathbb{E} \left[ \| \mathcal{G}_\gamma (x_k^\tau) \|^2 \right] &\leq \frac{40 L \sqrt{d}}{\sqrt{\ell b}}\frac{ F(x_0^0) - \min F }{T + 1}\\
        &+ 4\left( \frac{3 d}{\ell b } + 2 + 3d \right) \frac{2 \alpha + 2}{2\alpha +1} \frac{L^2 \beta^2 }{ d (T + 1)}.
    \end{aligned}   
\end{equation*}
Instead, with the choice $\beta_\tau = \frac{\beta}{d}$ (ii) we get
\begin{equation*}
\begin{aligned}
        \frac{1}{(T + 1)m} \sum\limits_{\tau = 0}^T \sum\limits_{k = 0}^{m - 1} \mathbb{E} \left[ \| \mathcal{G}_\gamma (x_k^\tau) \|^2 \right] &\leq \frac{40 L \sqrt{d}}{\sqrt{\ell b}}\frac{ F(x_0^0) - \min F }{T + 1}\\
        &+ \underbrace{4\left( \frac{3 d}{\ell b } + 2 + 3d \right) \frac{L^2}{d}}_{=: C_2} \beta^2.
\end{aligned}
\end{equation*}
We prove point (iii) in a constructive way. We fix an $\varepsilon \in (0, 1)$. We upper bound eq.\eqref{eqn:cor1_main_ineq} with it and we consider $\beta_\tau = \beta$ i.e. 
\begin{equation*}
    \frac{40 L \sqrt{d}}{\sqrt{\ell b}}\frac{ F(x_0^0) - \min F }{T + 1}  + \underbrace{4\left( \frac{3 d}{\ell b } + 2 + 3d \right) L^2 d}_{=: C_3} \beta^2 \leq \varepsilon
\end{equation*}
Let $\beta = \sqrt{\frac{\varepsilon}{2C_3}}$. If we impose, %
\begin{equation*}
    \frac{40 L \sqrt{d}}{\sqrt{\ell b}}\frac{ F(x_0^0) - \min F }{T + 1} \leq \frac{\varepsilon}{2},
\end{equation*}
namely,
\begin{equation*}
    T + 1 \geq  \frac{80 L \sqrt{d} \left(F(x_0^0) - \min F \right)}{\sqrt{\ell b} \varepsilon},
\end{equation*}
we have $\frac{1}{(T + 1)m} \sum\limits_{\tau = 0}^T \sum\limits_{k = 0}^{m - 1} \mathbb{E} \left[ \| \mathcal{G}_\gamma (x_k^\tau) \|^2 \right] \leq \varepsilon.$ Considering that, to perform an outer iteration, $n (d + 1) + 2 m b (\ell + 1)$ (stochastic) function evaluations are required, the complexity of the algorithm is
\begin{equation*}
    \mathcal{O}\left( \left( n \frac{d \sqrt{d}}{\sqrt{\ell b}}  + \frac{m b (\ell + 1) \sqrt{d}}{\sqrt{\ell b}} \right)\varepsilon^{-1} \right).
\end{equation*}
With the choice $b = \lceil n^{2/3} \rceil$ , $m =  \lceil \sqrt{b} \rceil$ and $\ell = \lceil d / c \rceil$ for some $c > 0$, we get
\begin{equation*}
    \mathcal{O}\left( n^{2/3} d \varepsilon^{-1} \right).
\end{equation*}

\subsection*{Proof of Theorem \ref{thm:pl}}
Let $C(\gamma) := \left(\frac{1}{2\gamma} - \frac{L}{2} \right)$. By Lemma \ref{lem:funval_bound}, with total expectation, we have
\begin{equation}\label{eqn:eqn_pl_1}
    \begin{aligned}
        \mathbb{E}[F(x_{k + 1}^\tau)] &\leq \mathbb{E}[F(x_k^\tau)] - \left( \frac{1}{2} - L \gamma \right) \gamma \mathbb{E}[ \left\| \mathcal{G}_\gamma(x_k^\tau) \right\|^2]\\
        &+ \frac{6d}{\ell b}L^2 \gamma \mathbb{E}[ \left\| x_k^\tau - x_0^\tau \right\|^2]\\
        &\underbrace{- C(\gamma) \mathbb{E}[ \left\| x_{k+1}^\tau - x_k^\tau \right\|^2]}_{=:(a)} + \left( \frac{3d}{\ell b} + 2 + 3d \right) L^2 d \gamma \beta^2_\tau.
    \end{aligned}
\end{equation}
Now we bound $(a)$. Adding and subtracting $x_0^\tau$, we get
\begin{equation*}
    \begin{aligned}
        -C(\gamma) \mathbb{E}[ \left\| x_{k+1}^\tau - x_k^\tau \right\|^2] &= -C(\gamma) \mathbb{E}[ \left\| x_{k+1}^\tau - x_0^\tau + x_0^\tau - x_k^\tau \right\|^2]\\
        &= -C(\gamma) \mathbb{E}[ \left\| x_{k+1}^\tau - x_0^\tau \right\|^2] -C(\gamma) \mathbb{E}[ \left\| x_{k}^\tau - x_0^\tau \right\|^2]\\
        &- 2 C(\gamma) \scalT{x_{k + 1}^\tau - x_0^\tau}{x_0^\tau - x_k^{\tau}}.
    \end{aligned}
\end{equation*}
By Cauchy-Schwartz and Young inequality with parameter $\xi > 0$, we have
\begin{equation*}
    \begin{aligned}
        -C(\gamma) \mathbb{E}[ \left\| x_{k+1}^\tau - x_k^\tau \right\|^2] &\leq -C(\gamma) \mathbb{E}[ \left\| x_{k+1}^\tau - x_0^\tau \right\|^2] -C(\gamma) \mathbb{E}[ \left\| x_{k}^\tau - x_0^\tau \right\|^2]\\
        &+ 2 C(\gamma)\left[ \frac{1}{2 \xi} \|x_{k + 1}^\tau - x_0^\tau \|^2 + \frac{\xi}{2} \|x_0^\tau - x_k^{\tau}\|^2 \right]\\
        &= - \left( 1 - \frac{1}{\xi} \right) C(\gamma) \|x_{k + 1}^\tau - x_0^\tau \|^2 + \left( \xi - 1 \right) C(\gamma)\mathbb{E}[ \left\| x_{k}^\tau - x_0^\tau \right\|^2].
    \end{aligned}
\end{equation*}
Let $\theta_1 > 0$ and $\xi = \frac{1}{\theta_1} + 1$. Then,
\begin{equation*}
    \begin{aligned}
        -C(\gamma) \mathbb{E}[ \left\| x_{k+1}^\tau - x_k^\tau \right\|^2] &\leq -\frac{C(\gamma)}{1 + \theta_1} \mathbb{E}[ \left\| x_{k+1}^\tau - x_0^\tau \right\|^2] + \frac{C(\gamma)}{\theta_1} \mathbb{E}[ \left\| x_{k}^\tau - x_0^\tau \right\|^2].
    \end{aligned}
\end{equation*}
Plugging this upper bound in inequality \eqref{eqn:eqn_pl_1}, we get
    \begin{equation*}
        \begin{aligned}
        \mathbb{E}[F(x_{k + 1}^\tau) ] &\leq \mathbb{E}[F(x_k^\tau)] - \left( \frac{1}{2} - L \gamma \right) \gamma \mathbb{E}[\| \mathcal{G}_\gamma(x_k^\tau) \|^2] \\
        &- \frac{C(\gamma)}{1 + \theta_1} \mathbb{E} \left[ \| x_{k + 1}^\tau - x_0^\tau \|^2  \right] + \left(\frac{6d}{\ell b}L^2\gamma + \frac{C(\gamma)}{\theta_1} \right)  \mathbb{E}[\| x_k^\tau - x_0^\tau \|^2]  \\
        &+ \left( \frac{3 d}{\ell b} + 2 + 3d \right)L^2 d \gamma \beta_\tau^2.  \\ %
        \end{aligned}
    \end{equation*}
Due to the choice of $\gamma$, we have that $\gamma L \leq \frac{1}{4}$. Thus,
\begin{equation*}
    \begin{aligned}
        \mathbb{E}[F(x_{k + 1}^\tau) ]&\leq \mathbb{E}[F(x_k^\tau)] - \frac{\gamma}{4} \mathbb{E}[\| \mathcal{G}_\gamma(x_k^\tau) \|^2] \\
        &- \frac{C(\gamma)}{1 + \theta_1} \mathbb{E} \left[ \| x_{k + 1}^\tau - x_0^\tau \|^2  \right] + \left(\frac{6d}{\ell b}L^2\gamma + \frac{C(\gamma)}{\theta_1} \right)  \mathbb{E}[\| x_k^\tau - x_0^\tau \|^2]  \\
        &+ \left( \frac{3 d}{\ell b} + 2 + 3d \right)L^2 d \gamma \beta_\tau^2.
    \end{aligned}
\end{equation*}
Using the RPL property of $F$ (Assumption \ref{ass:pl}),
\begin{equation*}
\begin{aligned}
\mathbb{E}\left[F\left(x_{k+1}^{\tau}\right)\right] 
&\leq \mathbb{E}[F(x_k^\tau)] - \frac{\gamma\mu}{2} \left(\mathbb{E}[F(x_k^\tau)] - \min F\right) \\
        &\qquad- \frac{C(\gamma)}{1 + \theta_1} \mathbb{E} \left[ \| x_{k + 1}^\tau - x_0^\tau \|^2  \right] + \left(\frac{6d}{\ell b}L^2\gamma + \frac{C(\gamma)}{\theta_1} \right)  \mathbb{E}[\| x_k^\tau - x_0^\tau \|^2] \\
        &\qquad + \left( \frac{3 d}{\ell b} + 2 + 3d \right)L^2 d \gamma \beta_\tau^2. 
\end{aligned}    
\end{equation*}
Adding and subtracting $\min F$, we have
\begin{equation*}%
\begin{aligned}
\mathbb{E}\left[F\left(x_{k+1}^{\tau}\right) - \min F\right] 
&\leq \left(1 - \frac{\gamma\mu}{2}\right) \mathbb{E}\left[F(x_k^\tau) - \min F\right] \\
        &\qquad- \frac{C(\gamma)}{1 + \theta_1} \mathbb{E} \left[ \| x_{k + 1}^\tau - x_0^\tau \|^2  \right] + \left(\frac{6d}{\ell b}L^2\gamma + \frac{C(\gamma)}{\theta_1} \right)  \mathbb{E}[\| x_k^\tau - x_0^\tau \|^2]  \\
        &\qquad + \left( \frac{3 d}{\ell b} + 2 + 3d \right)L^2 d \gamma \beta_\tau^2. 
\end{aligned}    
\end{equation*}
Let $\alpha = \left( 1 -  \frac{\gamma \mu}{2} \right)$ and $\theta_1 = 2k + 1$. We obtain that,%
\begin{equation*}
    \begin{aligned}
        \frac{\mathbb{E}\left[F\left(x_{k+1}^{\tau}\right) - \min F\right]}{\alpha^{k+1}}
        &\leq \frac{\mathbb{E}\left[F(x_k^\tau) - \min F\right]}{\alpha^{k}} \\
        &- \frac{1}{\alpha^{k+1}} \left[\frac{C(\gamma)}{2k + 2} \mathbb{E} \left[ \| x_{k + 1}^\tau - x_0^\tau \|^2  \right] - \left(\frac{6d}{\ell b}L^2\gamma + \frac{C(\gamma)}{2k + 1} \right)  \mathbb{E}[\| x_k^\tau - x_0^\tau \|^2] \right]  \\
        &+ \frac{1}{\alpha^{k+1}} \left( \frac{3 d}{\ell b} + 2 + 3d \right)L^2 d \gamma \beta_\tau^2.%
    \end{aligned}
\end{equation*}
Now, summing over $k$ from $0$ to $m-1$, we have
\begin{equation*}
    \begin{aligned}
        \mathbb{E}\left[F\left(x_{m}^{\tau}\right) - \min F\right]
        &\leq \alpha^m \mathbb{E}\left[F(x_0^\tau) - \min F\right] + \alpha^m \sum_{k=0}^{m-1} \frac{1}{\alpha^{k+1}} \left( \frac{3 d}{\ell b} + 2 + 3d \right)L^2 d \gamma \beta_\tau^2 \\%     \frac{3L^2 d^2}{\ell} \gamma \beta_\tau^2 + \alpha^m \sum_{k=0}^{m-1} \frac{1}{\alpha^{k+1}}\frac{L^2 d}{2} \gamma \beta_\tau^2 \nonumber \\
        &- \alpha^m  \mathbb{E} \left[ \sum_{k=0}^{m-1}\frac{C(\gamma)}{2k + 2}  \frac{1}{\alpha^{k+1}} \| x_{k + 1}^\tau - x_0^\tau \|^2   - \sum_{k=0}^{m-1} \frac{1}{\alpha^{k+1}} \left(\frac{6d}{\ell b}L^2\gamma + \frac{C(\gamma)}{2k + 1} \right)  \| x_k^\tau - x_0^\tau \|^2 \right]  \\
        &= \alpha^m \mathbb{E}\left[F(x_0^\tau) - \min F\right] + \frac{1-\alpha^m}{1-\alpha} \left( \frac{3 d}{\ell b} + 2 + 3d \right)L^2 d \gamma \beta_\tau^2 \\%\left[\frac{3L^2 d^2}{\ell} \gamma  + \frac{L^2 d}{2} \gamma \right]\beta_\tau^2 \nonumber \\
        &- \alpha^m  \mathbb{E} \left[ \sum_{k=0}^{m-1}\frac{C(\gamma)}{2k + 2}  \frac{1}{\alpha^{k+1}} \| x_{k + 1}^\tau - x_0^\tau \|^2   - \sum_{k=1}^{m-1} \frac{1}{\alpha^{k+1}} \left(\frac{6d}{\ell b}L^2\gamma + \frac{C(\gamma)}{2k + 1} \right)  \| x_k^\tau - x_0^\tau \|^2 \right]  \\
        &= \alpha^m \mathbb{E}\left[F(x_0^\tau) - \min F\right] + \frac{1-\alpha^m}{1-\alpha} \left( \frac{3 d}{\ell b} + 2 + 3d \right)L^2 d \gamma \beta_\tau^2  \\%\left[\frac{3L^2 d^2}{\ell} \gamma  + \frac{L^2 d}{2} \gamma \right]\beta_\tau^2 \nonumber \\
        &- \alpha^m  \mathbb{E} \left[ \sum_{k=0}^{m-2}  \frac{1}{\alpha^{k+2}} \left( \left(\frac{1}{2\gamma} - \frac{L}{2} \right)\left(\frac{\alpha}{2k + 2} - \frac{1}{2k + 3}\right) - \frac{6d}{\ell b}L^2\gamma \right)\| x_{k + 1}^\tau - x_0^\tau \|^2   \right].        
    \end{aligned}
\end{equation*}
Let $\displaystyle \theta = 1 - \frac{\mu \gamma}{2} - \mu\gamma m$. Due to the choice of the stepsize $\gamma < \min\left\{\frac{1}{4L}, \frac{2}{\mu}\frac{1}{(2m+1)},\frac{\sqrt{b\theta}}{10mL}\sqrt{\frac{\ell}{d}}\right\}$, we have that $\theta >0$ and, for every  $k \in \{0, 1, \cdots, m-1\}$,
\begin{equation*}
     \left( \left(\frac{1}{2\gamma} - \frac{L}{2} \right)\left(\frac{\alpha}{2k + 2} - \frac{1}{2k + 3}\right) - \frac{6d}{\ell b}L^2\gamma \right) \geq 0.
\end{equation*}
Indeed, since $\gamma < \frac{2}{\mu}\frac{1}{(2m+1)}$, we have that for all $k \in \{0, 1, \cdots, m-1\}$,
\begin{equation*}
\begin{aligned}
    \frac{6d}{\ell b}L^2\gamma &\leq \left(\frac{1}{2\gamma} - \frac{L}{2} \right)\left(\frac{\alpha}{2k + 2} - \frac{1}{2k +3}\right)\\
    &\leq\left(\frac{1}{2\gamma} - \frac{L}{2} \right)\left(\frac{\alpha}{2m} - \frac{1}{2m + 1}\right) \\
    &= \frac{1}{\gamma} \left(\frac{1}{2} - \frac{L\gamma}{2} \right) \frac{\alpha + 2m(\alpha -1 )}{2m(2m+1)} .
\end{aligned}    
\end{equation*}
Since $\gamma L \leq \frac{1}{4}$, we have
\begin{equation*}
\begin{aligned}
    \frac{6d}{\ell b}L^2\gamma &\leq \frac{3}{8\gamma} \frac{\theta}{6 m^2}. %
\end{aligned}    
\end{equation*}
Notice that this inequality is true since $\gamma^2 \leq \frac{b\theta}{100m^2L^2}\frac{\ell}{d}$.
Thus, we have
\begin{equation*}
    \begin{aligned}
        \mathbb{E}\left[F\left(x_{m}^{\tau}\right) - \min F\right] &\leq \alpha^m \mathbb{E}\left[F(x_0^\tau) - \min F\right] + \frac{1-\alpha^m}{1-\alpha} \left( \frac{3 d}{\ell b} + 2 + 3d \right)L^2 d \gamma \beta_\tau^2.%
    \end{aligned}
\end{equation*}
Using the facts that $x_{m}^{\tau} = x_{0}^{\tau+1}$, we can write
\begin{equation}
    \mathbb{E}\left[F\left(x_{0}^{\tau+1}\right) - \min F\right] \leq \alpha^m \mathbb{E}\left[F(x_0^\tau) - \min F\right] + \frac{1-\alpha^m}{1-\alpha} \left( \frac{3 d}{\ell b} + 2 + 3d \right)L^2 d \gamma \beta_\tau^2.%
\end{equation}
Let $\bar{\alpha} := \alpha^{m}$ and $\Psi_\tau := \frac{\mathbb{E}\left[F(x_0^\tau) - \min F\right]}{\bar{\alpha}^\tau}$. Then, multiplying both sides by $\frac{1}{\bar{\alpha}^{\tau + 1}}$, we get
\begin{equation*}
    \Psi_{\tau + 1} \leq \Psi_\tau  + \frac{1-\alpha^m}{1-\alpha} \left( \frac{3 d}{\ell b} + 2 + 3d \right)L^2 d \gamma \frac{\beta_\tau^2}{\bar{\alpha}^{\tau + 1}}.%
\end{equation*}
Now, summing for $i$ from $0$ to $\tau$, we get
\begin{equation*}
    \Psi_{\tau + 1} \leq \Psi_0  + \frac{1-\alpha^m}{1-\alpha} \left( \frac{3 d}{\ell b} + 2 + 3d \right)L^2 d \gamma \sum\limits_{i = 0}^\tau \frac{\beta_i^2}{\alpha^{(i + 1)m}}.%
\end{equation*}
Multiplying by $\alpha^{(\tau + 1)m}$, we get the claim
\begin{equation*}
    \begin{aligned}
    \mathbb{E}[F(x_0^{\tau + 1}) - \min F] &\leq \alpha^{(\tau + 1)m} [F(x_0^{0}) - \min F]\\
    &+ \alpha^{(\tau + 1)m}\frac{1-\alpha^m}{1-\alpha} \left( \frac{3 d}{\ell b} + 2 + 3d \right)L^2 d \gamma \sum\limits_{i = 0}^\tau \frac{\beta_i^2}{\alpha^{(i + 1)m}}.        
    \end{aligned}
\end{equation*}

\subsection*{Proof of Corollary \ref{cor:cor2}}

By Theorem \ref{thm:pl}, we have
\begin{equation*}
    \begin{aligned}
    \mathbb{E}\left[F\left(x_{0}^{\tau+1}\right) - \min F\right] &\leq \alpha^{m(\tau+1)}\left[F(x_0) - \min F\right]  \\
    &\qquad + \alpha^{m(\tau+1)}\frac{1-\alpha^m}{1-\alpha}\gamma\left[\frac{3L^2 d^2}{\ell}  + \frac{L^2 d}{2} \right]\sum_{i=0}^{\tau}\frac{\beta_i^2}{\alpha^{m(i + 1)}},
    \end{aligned} 
\end{equation*}
where $\alpha = 1 - \frac{\gamma \mu}{2}$. To prove point (i), we choose $\beta_\tau = \alpha^{m(\tau + 1)/2} \pi_{\tau}$ with $\sum\limits_{i =0}^{+\infty} \pi_{i}^{2} \leq \tilde{C} < + \infty$. Such a choice yields 
\begin{equation*}
    \begin{aligned}
    \mathbb{E}\left[F\left(x_{0}^{\tau+1}\right) - \min F\right] &\leq \alpha^{m(\tau+1)}\left[F(x_0) - \min F\right]  \\
    &+ \alpha^{m(\tau+1)}\frac{1-\alpha^m}{1-\alpha}\gamma\left[\frac{3L^2 d^2}{\ell}  + \frac{L^2 d}{2} \right]  \sum_{i=0}^{\tau} \pi_{i}^{2} %
    \end{aligned} 
\end{equation*}
Since $\sum_{i=0}^{\tau} \alpha^{(\tau + 1)m} = \frac{1 - \alpha^{(\tau + 1)m}}{1 - \alpha^{m}}$ and $\sum\limits_{i =0}^{+\infty} \pi_{i}^{2} \leq \tilde{C}$, we have 
\begin{equation*}
    \begin{aligned}
    \mathbb{E}\left[F\left(x_{0}^{\tau+1}\right) - \min F\right] &\leq \alpha^{m(\tau+1)}\left[F(x_0) - \min F\right]  \\
    &+ \alpha^{m(\tau+1)}\frac{1-\alpha^m}{1-\alpha} \frac{1 - \alpha^{m(\tau + 1)}}{1 - \alpha^m} \gamma\left[\frac{3L^2 d^2}{\ell}  + \frac{L^2 d}{2} \right] \tilde{C}\\
    &=\alpha^{m(\tau+1)}\left[F(x_0) - \min F\right]  \\
    &+ \alpha^{m(\tau+1)} \frac{1 - \alpha^{m(\tau + 1)}}{1 - \alpha} \gamma\left[\frac{3L^2 d^2}{\ell}  + \frac{L^2 d}{2} \right] \tilde{C}. %
    \end{aligned} 
\end{equation*}
Observing that $1 - \alpha^{m(\tau + 1)} \leq 1$ and due to the definition of $\alpha$, we get the first claim i.e.
\begin{equation*}
    \begin{aligned}
    \mathbb{E}\left[F\left(x_{0}^{\tau+1}\right) - \min F\right] &\leq \alpha^{m(\tau+1)}\left[F(x_0) - \min F\right] +\frac{\alpha^{m(\tau+1)}}{1 - \alpha} \gamma\left[\frac{3L^2 d^2}{\ell}  + \frac{L^2 d}{2} \right] \tilde{C}\\%\beta^2\\
    &=\alpha^{m(\tau+1)}\left[F(x_0) - \min F\right] + \alpha^{m(\tau+1)}  \frac{2}{\mu}\left[\frac{3L^2 d^2}{\ell}  + \frac{L^2 d}{2} \right]\tilde{C}%
    \end{aligned} 
\end{equation*}
To prove point (ii) of the Corollary, we fix $\beta_\tau = \beta$. Thus, we get
\begin{equation*}
    \begin{aligned}
    \mathbb{E}\left[F\left(x_{0}^{\tau+1}\right) - \min F\right] &\leq \alpha^{m(\tau+1)}\left[F(x_0) - \min F\right] \\
    &\qquad + \left(\alpha^{m(\tau+1)}\frac{1-\alpha^m}{1-\alpha}\gamma\left[\frac{3L^2 d^2}{\ell}  + \frac{L^2 d}{2} \right]  \sum_{i=0}^{\tau}\frac{1}{\alpha^{(i+1)m}} \right) \beta^2 .
    \end{aligned} 
\end{equation*}
Again, since $\alpha^{(\tau + 1)m}\sum_{i=0}^{\tau}\frac{1}{\alpha^{(i+1)m}} = \frac{1 - \alpha^{(\tau + 1)m}}{1 - \alpha^{m}}$, we conclude the proof
\begin{equation*}
    \begin{aligned}
    \mathbb{E}\left[F\left(x_{0}^{\tau+1}\right) - \min F\right] &\leq \alpha^{m(\tau+1)}\left[F(x_0) - \min F\right] \\
    &+ \left(\frac{1-\alpha^m}{1-\alpha} \frac{1 - \alpha^{m(\tau+1)}}{1 - \alpha^{m}} \gamma\left[\frac{3L^2 d^2}{\ell}  + \frac{L^2 d}{2} \right]   \right) \beta^2\\
    &= \alpha^{m(\tau+1)}\left[F(x_0) - \min F\right]\\
    &+ \left( \frac{1 - \alpha^{m(\tau+1)}}{1 - \alpha} \gamma\left[\frac{3L^2 d^2}{\ell}  + \frac{L^2 d}{2} \right]   \right) \beta^2.\\
    \end{aligned} 
\end{equation*}
Since $1 - \alpha^{m(\tau+1)} \leq 1$,
\begin{equation}\label{eqn:cor_pl_ii}
    \begin{aligned}
    \mathbb{E}\left[F\left(x_{0}^{\tau+1}\right) - \min F\right] &\leq \alpha^{m(\tau+1)}\left[F(x_0) - \min F\right] + \left( \frac{1}{1 - \alpha} \gamma\left[\frac{3L^2 d^2}{\ell}  + \frac{L^2 d}{2} \right]   \right) \beta^2\\
    &=\alpha^{m(\tau+1)}\left[F(x_0) - \min F\right] + \underbrace{  \frac{2}{\mu} \left[\frac{3L^2 d^2}{\ell}  + \frac{L^2 d}{2} \right] }_{=:\bar{C}} \beta^2.
    \end{aligned} 
\end{equation}
We prove point (iii) starting from point (ii) of the Corollary. Let $\varepsilon \in (0, 1)$, we choose $\beta_\tau = \beta >0$ and we upper-bound eq. \eqref{eqn:cor_pl_ii} with $\varepsilon$ i.e.
\begin{equation*}
    \begin{aligned}
    \mathbb{E}\left[F\left(x_{0}^{\tau+1}\right) - \min F\right] &\leq \alpha^{m(\tau+1)}\left[F(x_0) - \min F\right] + \bar{C} \beta^2 \leq \varepsilon.
    \end{aligned} 
\end{equation*}
Let $\beta = \sqrt{\frac{\varepsilon}{2\bar{C}}}$ and recalling that $\alpha = \left( 1 - \frac{\gamma \mu}{2} \right) < 1$, we have that for $T \geq \mathcal{O}\left(\frac{1}{\gamma \mu m} \log \frac{1}{\varepsilon} \right)$,
\begin{equation*}
    \mathbb{E}[F(x_0^T) - \min F] \leq \varepsilon
\end{equation*}
and, considering that an outer iteration requires $n(d + 1) + 2mb(\ell + 1)$ function evaluations, the complexity is
\begin{equation*}
    \begin{aligned}
    \mathcal{O} \left(\frac{nd}{\gamma \mu m} \log \frac{1}{\varepsilon} + \frac{b \ell}{\gamma \mu} \log \frac{1}{\varepsilon} \right).
    \end{aligned} 
\end{equation*}

\end{document}